\documentclass[11pt]{article}

\usepackage{amsmath, amsthm, amssymb, amscd}
\usepackage{graphicx}
\usepackage{bm}
\usepackage{tikz}
\usepackage[title]{appendix}
\usepackage{comment}

\usepackage{hyperref}

\newcommand{\vt}{\vartheta}

\newcommand{\argmin}{\mathop{\rm arg\min}}
\newcommand{\argmax}{\mathop{\rm arg\max}}

\allowdisplaybreaks[4]
\numberwithin{equation}{section}

\swapnumbers

\numberwithin{equation}{section}
\newtheorem{thm}{Theorem}[section]
\newtheorem{lem}[thm]{Lemma}
\newtheorem{rem}[thm]{Remark}
\newtheorem{prop}[thm]{Proposition}
\newtheorem{cor}[thm]{Corollary}

\newtheorem{defi}[thm]{Definition}

\newcommand{\Var}{\operatorname{Var}}

\newcommand{\E}{\mathbb{E}}
\renewcommand{\P}{\mathbb{P}}

\providecommand{\eps}{\varepsilon}

\newcommand{\MLE}{\operatorname{MLE}}
\newcommand{\KS}{\operatorname{KS}}
\newcommand{\Leb}{\operatorname{Leb}}

\newcommand{\sign}{\operatorname{sign}}

\newcommand{\Exp}{\operatorname{Exp}}

\newcommand{\TV}{\operatorname{TV}}

\newcommand{\Pois}{\operatorname{Pois}}

\newcommand{\Unif}{\operatorname{Unif}}

\newcommand{\PC}{\operatorname{PC}}

\newcommand{\Lip}{\operatorname{Lip}}
\newcommand{\fMLE}{\widehat f^{\MLE}}
\newcommand{\fMLEKs}{\widehat f^{\MLE}_{K,\bs}}
\newcommand{\mon}{\operatorname{mon}}

\newcommand{\R}{\mathbb{R}}

\newcommand{\bt}{\mathbf{t}}
\newcommand{\ba}{\mathbf{a}}

\newcommand{\bs}{\mathbf{s}}
\newcommand{\bb}{\mathbf{b}}

\newcommand{\mC}{\mathcal{C}}
\newcommand{\mF}{\mathcal{F}}
\newcommand{\mM}{\mathcal{M}}

\renewcommand{\theta}{\vartheta}
\begin{document}

\title{Nonparametric Bayesian analysis\\ of the compound Poisson
prior \\ for support boundary recovery}

\author{\parbox[t]{6cm}{\centering Markus Rei\ss\\Institute of Mathematics\\ Humboldt-Universit\"at zu Berlin\\ mreiss@math.hu-berlin.de}  \hspace{1cm} \parbox[t]{6cm}{\centering Johannes Schmidt-Hieber\\ Mathematical Institute\\ Leiden University \\ schmidthieberaj@math.leidenuniv.nl} }

\date{}
\maketitle

\begin{abstract}
Given data from a Poisson point process with intensity $(x,y) \mapsto n \mathbf{1}(f(x)\leq y),$ frequentist properties for the Bayesian reconstruction of the support boundary function $f$ are derived. We mainly study compound Poisson process priors with fixed intensity proving that the posterior contracts with nearly optimal rate for monotone and piecewise constant support boundaries and adapts to H\"older smooth boundaries with smoothness index at most one. We then derive a non-standard Bernstein-von Mises result for a compound Poisson process prior and a function space with increasing parameter dimension. As an intermediate result the limiting shape of the posterior for random histogram type priors is  obtained. In both settings, it is shown that  the marginal posterior of the functional $\vartheta =\int f$ performs an automatic bias correction and contracts with a faster rate than the MLE. In this case, $(1-\alpha)$-credible sets are also asymptotic $(1-\alpha)$-confidence intervals. As a negative result, it is shown that the frequentist coverage of credible sets is lost for linear functions indicating that credible sets only have frequentist coverage for priors that are specifically constructed to match properties of the underlying true function.

\medskip



\end{abstract}

%
%

\newpage

\section{Introduction}

The estimation of support boundary  functions does not only have numerous applications, but also poses intriguing mathematical questions, see Gijbels {\it et al.} \cite{gijbels1999}, Chernozhoukov and Hong \cite{chernozhukov2004} as well as Korostelev and Tsybakov \cite{korostelev1993} for some overview. Here, we consider the fundamental observation model of a  Poisson point process (PPP) $N$ on $[0,T]\times \mathbb{R},$ $T>0$, with intensity
\begin{align}
	\lambda (x,y) = \lambda_f (x,y) = n \mathbf{1}(f(x)\leq y).
	\label{eq.mod}
\end{align}
We thus observe points $(X_i,Y_i)_{i\ge 1}$ on the epigraph of the boundary function $f:[0,T] \rightarrow \mathbb{R}.$ The goals is to recover the support boundary $f$ nonparametrically, see Figure \ref{fig.intro}. In a similar way as the Gaussian white noise model is the continuous analogue of nonparametric regression with centered errors, support boundary recovery occurs as the continuous limit of nonparametric regression with one-sided errors, see Meister and Rei\ss\ \cite{meisterreiss2013} for related asymptotic equivalence results. The fundamental difference is the information geometry: for the Gaussian white noise model this is the $L^2$-geometry, whereas for support boundary recovery it is induced by the $L^1$-norm and the laws are not mutually absolutely continuous.  As a consequence, not only convergence rates differ, but also the asymptotic distributions of estimators are non-classical. Moreover, the maximum-likelihood estimator (MLE) is often not efficient and Bayesian methods are advocated. At a methodological level we explore here to  what extent this remains true for non- and semi-parametric problems. This is particularly interesting because for many function classes a nonparametric MLE exists in the PPP model. In the related problem of boundary detection in images under Gaussian noise, the Hellinger distance is also of $L^1$-type, cf. Li and Ghosal \cite{LiGhoshal2015} for posterior contraction results, but the observation laws are mutually absolutely continuous and a nonparametric MLE usually does not exist.

\begin{figure}
\begin{center}
	\includegraphics[scale=0.7]{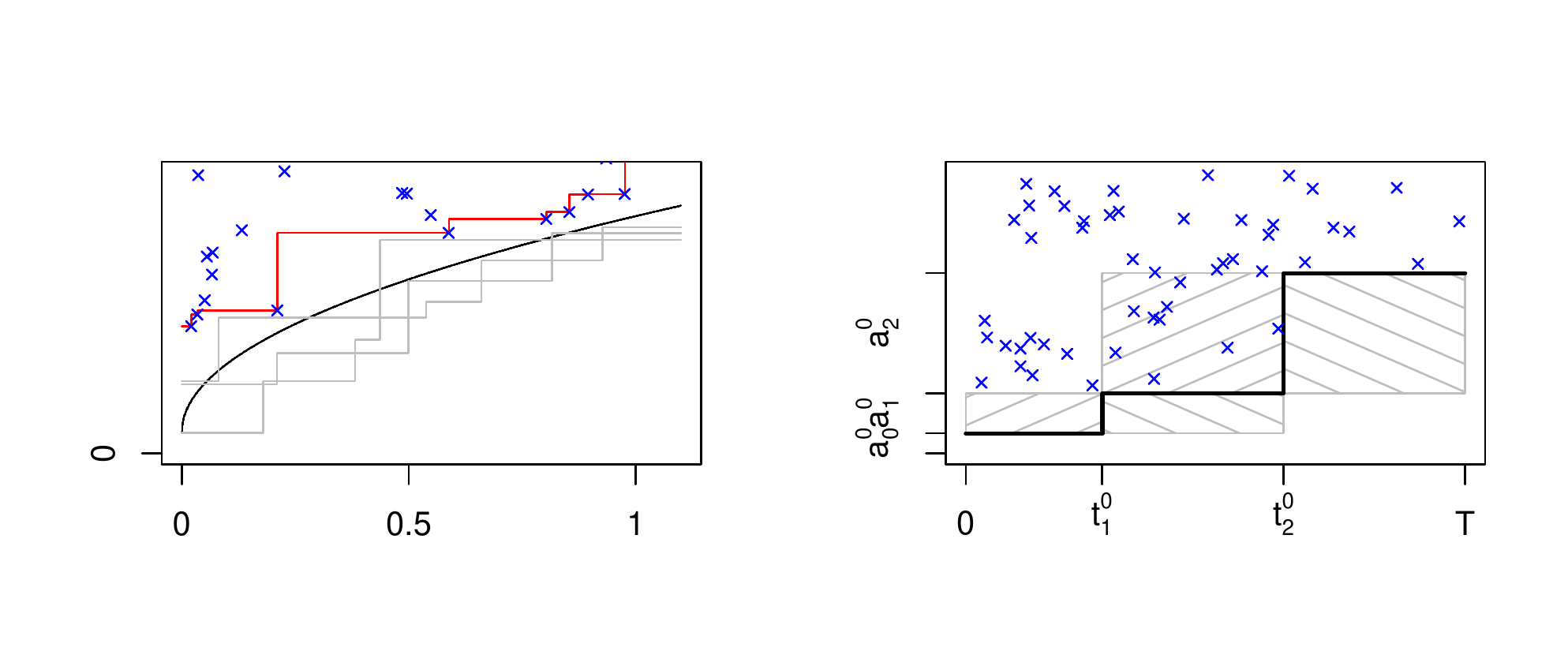}
	\vspace{-1.3cm}
	\caption{\label{fig.intro} {Two simulated data examples for the PPP model with true boundary (black) and observations (blue). Left:  MLE (red), posterior draws (gray). Right: shaded gray areas related to Definition \ref{assump.min_rect} below.}}
\end{center}
\end{figure}

A second major goal is to understand the performance of compound Poisson processes (CPP) as nonparametric priors. CPPs are probabilistically well understood, are easy to sample and can be equivalently understood as piecewise constant priors, where the jump locations are uniform, the jump sizes are i.i.d. random and the number of jumps is chosen by a Poisson hyperprior. For binary regression, CPP priors were studied by Coram and Lalley \cite{coram2006}, establishing nonparametric consistency, and they are often recommended in practice, e.g. as priors for monotone functions in Holmes and Heard \cite{holmes2003} with applications to gene expression data. We prove below that under CPP priors  optimal posterior contraction rates (sometimes up to logarithmic factors) are attained for H\"older functions, piecewise constant functions with number of jumps growing to infinity and for monotone functions. They even adapt automatically to the unknown H\"older smoothness or number of jumps. Given that the jump intensity remains fixed, this shows how powerful and versatile simple CPP priors are. The derivation of the contraction rates is based on the general theory developed in the companion paper \cite{reiss2018a}. The theory for monotone functions extends to subordinator priors, that is monotone L\'evy processes, which  have been studied  in survival analysis by Kim and Lee \cite{kim2004}, but not in the context of nonparametric posterior contraction rates.

Going beyond rate results, most effort is required to study Bernstein-von Mises (BvM) theorems for the function $f$ and its mean $\theta=\int f$, a basic semiparametric functional. Before turning to CPP priors, we study the simpler case of piecewise constant function priors where only the jump sizes are random and refer to them as random histogram priors. Concerning the frequentist approach, the nonparametric MLE $\fMLE$ exists for H\"older balls with smoothness index $\beta \leq 1,$ monotone functions and piecewise constant functions  and  achieves the minimax estimation rate. For functionals such as $\theta,$ however,  the MLE $\widehat  \vartheta^{\MLE} = \int \widehat f^{\MLE}$ converges usually with a suboptimal rate. A rate-optimal estimator can be obtained if we subtract a term that scales with the number of observations lying on the boundary of the MLE and consider
\begin{align}
	\widehat \vartheta = \int \widehat f^{\MLE} - \frac{\text{number of data points } (X_i,Y_i) \text{ on the boundary of } \widehat f^{\MLE}}{n},
	\label{eq.fctal_bias_correction}
\end{align}
see Rei\ss\ and Selk \cite{reiss2014}. This bias correction accounts for the fact that $\widehat f^{\MLE}$ overshoots the true boundary function $f$ considerably.
In the case of a constant function $f$ and for more general parametric setups, Bayes estimators correct the bias of the MLE by distributing the posterior mass correctly below $\widehat f^{\MLE}$, cf. Kleijn and Knapik \cite{Kleijn2012}.

It is therefore natural to ask whether a nonparametric Bayesian approach also performs this correction automatically. Here we show that the answer can be positive as well as negative. As a positive result, we prove that for piecewise constant and monotone support boundaries under random histogram and CPP priors the posterior concentrates around $\widehat \vartheta$ with the optimal contraction rate. Optimal frequentist estimation of piecewise constant and monotone functions in Gaussian noise has attracted a lot of attention recently, see Gao {\it et al.} \cite{Gao2017} and the references discussed there.  Furthermore, we obtain intervals which are simultaneously asymptotic $(1-\alpha)$-credible and $(1-\alpha)$-confidence intervals of rate-optimal length. The Bayesian approach clearly outperforms the MLE in this case. As a negative example, we consider a linear support boundary $f$. The posterior contracts around the true support boundary $f$ with the optimal rate, but the bias correction of the marginal posterior for $\vartheta$ is of incorrect order. In this case, credible sets have asymptotically no frequentist coverage. In conclusion, the bias correction induced by the Bayes approach must always be carefully studied.

Conceptionally, we  study BvM results for increasing parameter dimensions with the hyperprior on the number of jumps determining the model dimension. For linear and exponential family models this has been treated by Ghosal \cite{ghosal1999,Ghosal2000} and by Bontemps \cite{bontemps2011} for Gaussian regression. Panov and Spokoiny \cite{Panov2015} explore the scope of BvM results for regular i.i.d. models of growing dimension and find a critical dimension related to ours, see the discussion in Section \ref{SecHistoNonpar} below. A bias problem for functional estimation by adaptive Bayesian methods has
been exhibited by Castillo and Rousseau \cite{castillo2015b} and Rousseau and Rivoirard \cite{rivoirard2012}, which bears some similarity with our approach, but at a parametric $\sqrt n$-rate.

Related to CPP priors are many popular piecewise constant prior prescriptions. First of all, there are priors on regression trees, such as Bayesian CART (Denison {\it et al.} \cite{denison1998}) and BART (Chipman {\it et al.} \cite{chipman2010}). Regression trees subdivide the space of covariates and then put a constant value on each of the cells. These priors are henceforth supported on piecewise constant functions. Posterior contraction for BART has been derived only recently by Rockova and van der Pas \cite{rockova2017}.
For density estimation, histogram priors are well studied. Scricciolo \cite{Scricciolo2007} considers random histograms with fixed bin width and the number of bins a hyperprior. It is shown that near optimal contraction rates are obtained if the true density is either H\"older with index at most one or piecewise constant.

So far, only little theory has been developed  for nonparametric Bayes under shape constraints. Exceptions are
Salomond \cite{salomond2014} for monotone densities and  Mariucci {\it et al.} \cite{mariucci2017} for log-concave densities. In both cases mixtures of Dirichlet processes are taken as priors. To the best of our knowledge the present paper is the first one that derives Bernstein-von Mises type results under a shape constraint.

The subsequent article is organized as follows. Contraction rates for compound Poisson process  and subordinator priors are investigated in Section \ref{sec.post_contr}. In an interlude, Section \ref{sec.BvM_general} discusses a general description of the asymptotic posterior shape, in which the results thereafter can be embedded. Bernstein-von Mises type theorems and results on the frequentist coverage of credible sets can be found in Section \ref{sec.BvM} for random histogram priors and in Section \ref{SecBvMCPP} for CPP priors. The proofs for Sections \ref{sec.post_contr}, \ref{sec.BvM} and \ref{SecBvMCPP} are gathered in Appendices \ref{AppA}, \ref{AppB} and \ref{AppC}, respectively. Proofs for the monotone MLE are delegated to Appendix \ref{AppMLE}, and Appendix \ref{AppTV} contains some independent total variation results.

{\it Notation.} We write $N = \sum_i \delta_{(X_i,Y_i)}$ for a random point measure on $[0,1] \times \mathbb{R}$ and denote the support points by $(X_i, Y_i)_i.$ Whenever $N$ is observed, it is natural to call the support points  observations. Moreover, we use the standard terminology $\mathbf{1}_A := \mathbf{1}(\cdot \in A),$ $(x)_+ := \max(x,0)$ and $\|\cdot\|_p$ for the $L^p([0,1])$-norm.

\section{Posterior contraction}
\label{sec.post_contr}

{\bf Bayes formula.} Let us first recall the Bayes formula for the PPP model as derived in \cite{reiss2018a}. Let $(\Theta,d)$ be a Polish space equipped with its Borel $\sigma$-algebra and $d$ a stronger metric than the $L^1$-norm. For $f_0\in L^1([0,T])$, a prior $\Pi$ on $\Theta$ and a Borel set $B\subset\Theta,$ Lemma 2.2 in \cite{reiss2018a} gives an explicit Bayes formula under the law $P_{f_0}$:
\begin{align}
	\Pi(B | N) = \frac{\int_B e^{n\int_0^T f} \mathbf{1} (\forall i : f(X_i) \leq Y_i) \, d\Pi(f)}{\int_\Theta  e^{n\int_0^T f} \mathbf{1}(\forall i : f(X_i) \leq Y_i) \, d\Pi(f)}
	= \frac{\int_B  e^{-n\int_0^T (f_0-f)_+} \frac{dP_{f\vee f_0}}{dP_{f_0}} (N) \, d\Pi(f)}{\int_\Theta  e^{-n\int_0^T (f_0-f)_+} \frac{dP_{f\vee f_0}}{dP_{f_0}} (N) \, d\Pi(f)}\quad P_{f_0}\text{-a.s.}
	\label{eq.Bayes}
\end{align}
The default is $T=1$ but in Section \ref{SecBvMCPP} it is convenient to work with $T>1.$

{\bf Compound Poisson prior.} The main focus of this section is to study posterior contraction for compound Poisson process priors defined on the space $\Theta=D[0,1]$ of c\`adl\`ag functions, equipped with the Skorokhod topology. A compound Poisson process $Y$ on $[0,1]$ can be written as
$Y_t= \sum_{i=1}^{N_t} \Delta_i$
with  a Poisson process $(N_t)_{t\geq 0}$ of intensity $\lambda>0$ and  an i.i.d. sequence $(\Delta_i)$ of random variables, independent of the Poisson process. We denote the distribution of $\Delta_1$ by $G$. We randomize the starting value $X_0 = \Delta_0$ according to a distribution $H$ and consider
\begin{align}\label{EqCPP}
	X_t = \Delta_0 + \sum_{i=1}^{N_t} \Delta_i = \sum_{i=0}^{N_t} \Delta_i,
\end{align}
with $\Delta_0 \sim H$ independent of $(\Delta_i)_{i\ge 1}$ and $(N_t)_{t\geq 0}.$

A CPP can equivalently be viewed as a hierarchical prior on $f$ in the spirit of \cite{spa, castillo2015}. The hierarchical CPP construction picks in a first step a model dimension prior $\pi \sim \Pois(\lambda).$ The order statistics property of a Poisson process (\cite{Embrechts2003}, p.186) says that conditionally on the event that the CPP jumps $K$ times on $[0,1],$ the ordered jump locations $(t_1,\ldots, t_K),$ $t_0:=0\leq t_1 \leq \ldots \leq t_K \leq 1,$ have the same distribution as the order statistic of $K$ i.i.d. $U([0,1])$ random variables. The Lebesgue density of $(t_1, \ldots, t_K)|K$ is therefore $K!\mathbf{1}(0\leq t_1 \leq t_2 \leq \ldots \leq t_K \leq 1).$ The last step is then to assign the starting value $a_0$ and the jump sizes $a_1, \ldots, a_K.$  Assuming that the distributions $G, H$ have Lebesgue densities $g$ and $h,$ respectively, we can write the CPP prior in closed form as a prior on $K, \bt,$ and $\bb$
\begin{align}
	(K, \bt, \ba) \mapsto e^{-\lambda} \lambda^K h(a_0)\prod_{j=1}^K g(a_j) \mathbf{1}\big(0< t_1 < t_2 < \ldots <t_K<1 \big)
	\label{eq.CPP_explicit}
\end{align}
generating random c{\`a}dl{\`a}g functions $f= \sum_{j=0}^K a_j \mathbf{1}_{[t_j, 1] }$ with $t_0:=0.$

Since $\lambda$ is fixed, for most draws of the prior the number of jumps will be of order $\lambda$. As we show below, the CPP prior puts still enough mass around functions with an increasing number of jumps to ensure nearly optimal posterior contraction rates for H\"older functions and for piecewise constant function with an increasing number of pieces. Let us also mention that the CPP prior randomizes over the jump points and should therefore be able to adapt to local smoothness. This might be an advantage compared to random histogram priors where the function jumps on a fixed grid.


{\bf Function classes.}
A natural parameter space for CPP priors are piecewise constant functions with $K$ pieces:
\begin{align}
	\PC(K,R) = \Big\{ f: f= \sum_{j=1}^K b_j \mathbf{1}_{[t_{j-1}, t_j)} \ & \text{with} \  0:=t_0 <t_1<\ldots < t_K:=1, |b_j|\leq R\Big\}.
	\label{eq.PC_def}
\end{align}
We are interested in the case where the number of pieces $K=K_n$ grows with $n.$ In this case we have $2K_n-1$ parameters. Since the squared parametric rate is $n^{-1}$, we expect the best possible contraction rate to be $K_n/n.$ Moreover, we denote by $\mC^\beta(R)$ the ball of $\beta$-H\"older functions $f:[0,1]\to\R$ with  H\"older norm $\|f\|_{C^\beta}$ bounded by $R.$
The CPP prior allows to build in monotonicity as prior knowledge by choosing a positive jump distribution. We define the space of montone functions which are bounded by $R$ as
\begin{align*}
	\mM(R) :=\{ f: \ f \ \text{monotone increasing and} \ -R \leq f(0)\leq f(1) \leq R\}.
\end{align*}

\begin{thm}
\label{thm.contr_CPP}
Consider the CPP prior \eqref{EqCPP} with a positive and continuous Lebesgue density $h$ on $\mathbb{R}.$ If there are constants $\gamma,L>0$ such that $\P ( |\Delta_i | \geq s ) \leq L^{-1} e^{-L s^{\gamma}}$ for all $s\geq 0,$ then, there exist positive constants $M$ and $c$ such that
\begin{itemize}
\item[(i)]  if $g$ is positive and continuous on $\mathbb{R}^+,$
\begin{align*}
	\sup_{f_0\in \mM(R)} E_{f_0} \Big[ \Pi\Big(f : \|f-f_0\|_1\geq M \sqrt{\frac{\log n}{n}}   \,  \Big| \, N\Big) \Big] \leq e^{-c\sqrt{n\log n}};
\end{align*}
\item[(ii)] if $g$ is positive and continuous on $\mathbb{R},$
\begin{align*}
	\sup_{f_0\in \mC^\beta(R)} E_{f_0} \Big[ \Pi\Big(f : \|f-f_0\|_1\geq M  (\log n/n)^{\beta/(1+\beta)} \, \Big| \, N\Big) \Big] \leq e^{-cn(\log n/n)^{\beta/(1+\beta)}};
\end{align*}
\item[(iii)] if $g$ is positive and continuous on $\mathbb{R}$ and if $n^{\rho} \lesssim K_n=o(n/\log n) $ for some $\rho >0,$
\begin{align*}
	\sup_{f_0\in \PC(K_n, R)} E_{f_0} \Big[ \Pi\Big(f : \|f-f_0\|_1\geq M \frac {K_n}n \log n \, \Big| \, N\Big) \Big] \leq n^{-cK_n}.
\end{align*}
\end{itemize}
\end{thm}

In all cases the rate is  expected to be optimal up to the $\log n$ factor. Compound Poisson processes thus furnish a very versatile prior adapting to unknown smoothness and shape.

The proof is based on a  Ghosal-Ghosh-van der Vaart type result from \cite{reiss2018a}. To check the conditions we derive lower bounds on the one-sided small ball probabilities of the CPP prior for the function classes considered above. These bounds could be used to derive contraction rates for other nonparametric models.

{\bf Subordinators.} CPPs form the subclass of L\' evy processes with finite jump intensity. Allowing also for infinitely many jumps, subordinators, that is L\' evy processes with monotone sample paths, generate a rich class of monotone function priors. We consider only subordinators without drift, characterized by their characteristic function
\[ \phi_t(u)=\E[e^{iuY_t}]=\exp\Big(t\int_{\R^+}(e^{iux}-1)\nu(dx)\Big),\quad t\ge 0,\]
where the L\'evy measure $\nu$ is a $\sigma$-finite measure on $\R^+$, satisfying $\int_{\R^+} (x \wedge 1)\nu(dx)<\infty$. Its intensity is $\lambda=\nu(\R^+)\in[0,\infty]$ and in the finite intensity case a subordinator is just a compound Poisson process of intensity $\lambda$ with jump distribution $G=\nu/\lambda$.

Among  subordinators of infinite intensity prominent examples are the Gamma  and inverse Gaussian processes, see  \cite{Sato2013} for a comprehensive treatment. Dirichlet processes belong to the most frequently used priors in nonparametric Bayesian methods and can be viewed as time-changed and normalized Gamma processes, see \cite{ghoshal2017}, Section 4.2.3.
Subordinators as priors have been studied in the context of survival models by \cite{kim2004}. There the target of estimation is the cumulative hazard function, which can be estimated at the parametric rate $n^{-1/2}$. Subordinators as priors for shape-constrained estimation problems in regression or density-type models do not seem to have been analyzed yet  so that the result below can be of independent interest.

{\bf The randomly initialized subordinator prior.} As priors we consider randomly initialized subordinators of the form
\begin{align*}
	X_t =Y_0 +Y_t, \quad \text{with} \ (Y_t)_{t\geq 0} \ \text{ a subordinator and } Y_0 \sim H \ \text{independent of } (Y_t)_{t>0},
\end{align*}
where $H$ is assumed to have a positive and continuous Lebesgue density on $\mathbb{R}.$ Moreover, we suppose that the L\'evy measure $\nu$ has a Lebesgue density which by some slight abuse of notation is called $\nu(x)$ and is assumed to be continuous and positive on $\mathbb{R}^+.$

\begin{thm}
\label{thm.contr_subord}
Consider the randomly initialized subordinator prior. If there exist constants $\gamma, L>0$ such that $\nu(x) \leq L x^{-3/2}$ for all $x>0$ and  $\int_s^\infty \nu(x) dx \leq Le^{-L^{-1}s^\gamma}$ for all $s\geq 1,$ then there are  constants $M,c>0$ such that
\begin{align*}
	\sup_{f_0 \in \mM(R)} E_{f_0}\Big[ \Pi \Big(f: \|f -f_0\|_1 \geq M \sqrt{ \frac {\log n}n} \, \Big | \, N \Big) \Big]
	\leq e^{-c\sqrt{n\log n}}.
\end{align*}
\end{thm}

\section{On the generalized Bernstein-von Mises phenomenon}
\label{sec.BvM_general}

Before we move on and derive the posterior limit for the CPP prior, we briefly discuss the extension of the Bernstein-von Mises theorem beyond regular models. The classical Bernstein-von Mises theorem assumes a parametric model $(P_\theta^n : \theta \in \Theta)$ that is differentiable in quadratic mean and has nonsingular Fisher information $I_{\theta,n}.$ Then, for a continuous and positive prior, the posterior can be approximated in total variation distance by
\begin{align*}
	\mathcal{N}\big( \widehat{\theta}_n^{\MLE}, I_{\theta_0,n}^{-1}\big)
\end{align*}
if the i.i.d. data are generated from $P_{\theta_0}^n,$ $\theta_0 \in \Theta,$ see \cite{vdvaart1998}, Section 10.2 for a precise statement. It can also be easily seen that if we observe $Y_i= \theta_0 + \eps_i,$ $i=1,\ldots,n$, with independent $\eps_i \sim \Exp(1)$, then $\widehat \theta_n^{\MLE}=\min(Y_1,\ldots,Y_n)\sim \theta_0+\eps$ with $\eps\sim \Exp(n)$. For a continuous and positive prior we obtain in the limit the posterior $(\widehat \theta_n^{\MLE} - \widetilde \eps) \, | \, \widehat \theta^{\MLE}$ with $\widetilde \eps \sim \Exp(n)$ and $\widetilde \eps$ independent of $\eps,$ see \cite{2012arXiv1210.6204K}.

This suggests that a generalized Bernstein-von Mises theorem should be of the following form: If there exists a MLE $\widehat \theta_n^{\MLE}$ such that
\begin{align}
	\widehat \theta_n^{\MLE} = \theta_0 + \eps_n(\theta_0),
	\label{eq.MLE_expansion}
\end{align}
with $\eps_n(\theta_0)$ some random variable, then, under  standard assumptions on the prior, the posterior should be close to the conditional distribution of
\begin{align}
	\big(\widehat{\theta}_n^{\MLE} - \widetilde \eps_n(\theta_0)\big) \, \big| \, \widehat{\theta}_n^{\MLE},
	\label{eq.BvM_general}
\end{align}
where $\widetilde \eps_n(\theta_0)$ has the same distribution as $\eps_n(\theta)$ but is independent of it. This unifies both cases above. For problems with increasing model dimension, we can additionally build in a model selection prior such that the posterior concentrates on smaller models. If the posterior puts asymptotically all mass on one model, then \eqref{eq.MLE_expansion} and \eqref{eq.BvM_general} have to be replaced by the corresponding expressions in this model, see \cite{castillo2015}, Section 2.4 for an example. The posterior limit distributions that occur in the subsequent chapters are exactly of this form.

\section{Bernstein-von Mises for random histogram type priors}
\label{sec.BvM}

\subsection{An asymptotic shape result for the full posterior} \label{SecHistoNonpar}

The CPP prior charges all piecewise constant functions by randomizing over the number of jumps, the jump locations and the jump sizes. As an intermediate step, it seems natural to ask first about the limiting posterior shape if the number of jumps and the jump locations are fixed and only the jump sizes are random. Since such a prior generates piecewise constant functions looking like histograms, only without normalisation and non-negativity constraint, we  refer to this prior as histogram prior.

Given a positive integer $K$ and $(t_0, t_1, \ldots, t_K)$ with $0=:t_0 <t_1< \ldots < t_{K}:=1,$ consider the space of piecewise constant functions with fixed jump times:
\begin{align*}
	\PC^*(K, (t_0, \ldots, t_{K}), R) =\Big\{ f: f= \sum_{j=1}^{K} a_j \mathbf{1}_{[t_{j-1}, t_j)}, \, |a_j|\leq R\Big\}.
\end{align*}
The underlying parameter vector is $\ba=(a_1, \ldots, a_K)\in [-R,R]^K.$
We are mainly interested in the regime with increasing parameter dimension $K=K_n \rightarrow \infty.$ For convenience we omit the dependence on $n$ and write $t_j$ for the $n$-dependent jump points. As prior density on the vector $\ba$ consider
\begin{align}
	\pi(\ba) = \prod_{j=1}^K g(a_j)
	\label{eq.prior_PC*}
\end{align}
with a fixed Lebesgue density $g$. Compared to the CPP prior \eqref{eq.CPP_explicit}, we do not parametrize the jump sizes itself here. The MLE over  $(a_1,\ldots, a_K) \in \Theta = \R^K$ is
\begin{align*}
	\widehat f^{\MLE} = \sum_{j=1}^K \widehat a_j  \mathbf{1}_{[t_{j-1}, t_j)}, \quad \text{with} \ \ \widehat a_j := \min_{i: X_i \in [t_{j-1}, t_j) } Y_i,
\end{align*}
recalling that $(X_i,Y_i)_{i\ge 1}$ denote the observations  of the PPP $N$.
Write $f_0 = \sum_{j=1}^{K} a_j^0 \mathbf{1}_{[t_{j-1}, t_j)}$ for the true function. Under $P_{f_0}$,  we have $\widehat a_j -a_0^j \sim \Exp(n(t_j-t_{j-1}))$ because
\[ P_{f_0}(\widehat a_j -a_0>y) = P_{f_0}(N([t_{j-1},t_j) \times [a_0,a_0+y])=0)=e^{-n(t_j-t_{j-1})y},\quad y\ge 0.
\]

The main result in this section provides simple conditions under which the posterior can be approximated in total variation by the conditional distribution of
\begin{align*}
	\widehat f^{\MLE}- \sum_{j=1}^K\eta_j \mathbf{1}_{[t_{j-1}, t_j)}=\sum_{j=1}^K \big( \widehat a_j - \eta_j \big)\mathbf{1}_{[t_{j-1}, t_j)},
\end{align*}
for independent $\eta_j \sim \Exp(n(t_j-t_{j-1}))$ given the data $(X_i, Y_i)_i.$ Notice that this process is of the form \eqref{eq.BvM_general}.

The posterior $\Pi(\cdot|N)$ on the vector $(a_1,\ldots, a_K)$ is a measure on $(\R^K, \mathcal{B}(\R^K)).$ Let $Q^n$ be the distribution of $(\widehat a_1 - \eta_1, \ldots, \widehat a_K - \eta_K)$ on $(\R^K, \mathcal{B}(\R^K))$ for given $(\widehat a_1, \ldots, \widehat a_K)$ and denote by $\|\cdot\|_{\TV}$ the total variation norm.

\begin{thm}
\label{thm.BvM}
Consider the prior \eqref{eq.prior_PC*} and assume that $g$ is positive and $\beta$-H\"older continuous for some $0< \beta \leq 1.$ If
\begin{align}
	n \frac{\inf_{j=1, \ldots, K_n} |t_j-t_{j-1} | }{K_n^{1/\beta} \log (K_n)} \rightarrow \infty,
	\label{eq.BvM_cond}
\end{align}
then
\begin{align*}
	\sup_{f_0\in \PC^*(K_n, (t_0, \ldots, t_{K_n}), R) } E_{f_0}\big[ \big\| \Pi (\cdot | N) - Q^n \big\|_{\TV} \big] \rightarrow 0.
\end{align*}
\end{thm}

In contrast to the parametric Bernstein-von Mises theorem (\cite{vdvaart1998}, Theorem 10.1), Theorem \ref{thm.BvM} also assumes H\"older smoothness on the marginal prior densities $g.$ Together with the condition \eqref{eq.BvM_cond} this ensures that the prior washes out in the limit. The maximal speed at which $K_n$ can tend to infinity in Theorem \ref{thm.BvM} depends on the H\"older index $\beta$ and the rate at which the minimal grid length $t_{j}-t_{j-1}$ decreases. If the $t_j$ are on a regular grid in the sense that $\inf_j (t_j-t_{j-1}) \asymp 1/K_n$, then
\begin{align}
	K_n = o\Big( \frac{n}{\log n}\Big)^{\beta/(\beta+1)}
	\label{eq.Kn_constraint}
\end{align}
suffices. For $\beta=1$ we can allow $K_n$ to be almost $\sqrt{n}.$

Let us describe the reason for the rate \eqref{eq.Kn_constraint} in more detail. It can be shown that the posterior concentrates on a set $\mathcal{U}$ where each $a_i$ is localized up to a term of order $(K_n/n) \log n.$   For the Bernstein-von Mises theorem to hold, the variation of the prior on $\mathcal{U}$ must asymptotically vanish, that is, $\sup_{\ba,\ba' \in\mathcal{U}}|\pi(\ba)-\pi(\ba')|=o(1).$ For simplicity, assume that $(0,0,\ldots,0) \in \mathcal{U}$ and assume that the prior is $g(x)=1+|x|^\beta$ in a neighborhood of $x=0.$ This is a $\beta$-H\"older function. Observe that
\begin{align*}
	\pi\Big(\frac{K_n}{n}\log n, \ldots, \frac{K_n}{n}\log n\Big)= g\Big(\frac{K_n}{n}\log n\Big)^{K_n} = \Big(1+\Big(\frac{K_n}{n}\log n\Big )^\beta \Big)^{K_n} \approx e^{K_n (\frac {K_n} n \log n)^\beta}.
\end{align*}
To ensure that the prior variation over $\mathcal{U}$ vanishes we must have $K_n (K_n/n \log n)^\beta \rightarrow 0$ and rewriting this yields condition \eqref{eq.Kn_constraint}.

Condition \eqref{eq.Kn_constraint} should be compared to the Bernstein-von Mises phenomenon for increasing parameter dimension which requires a number of parameters smaller than $n^{1/3},$ cf. \cite{Panov2015}. Moreover, \cite{Hermansen2015} establishes a limiting shape result in the nonparametric regression model with Gaussian errors using a random histogram prior of the form \eqref{eq.prior_PC*} and $a_j$ drawn from a normal distribution. Using conjugacy, it can be shown that the prior washes out if the number of pieces is of a smaller polynomial order than $n^{1/2}.$

While the MLE overshoots each true parameters $a_0^j$ by an exponential distribution with parameter $n(t_j-t_{j-1}),$ asymptotically the posterior distribution "corrects" for that bias by subtracting independent $\eta_j$ with the same distribution. This is the reason why Bayesian methods are advocated for related parametric boundary estimation problems in the frequentist literature, see e.g. the discussion in \cite{chernozhukov2004}.  In the special case $K_n=1$ the true function $f_0=a_0^1$ is  a constant and the corresponding likelihood is proportional to $e^{n a_1} \mathbf{1}(a_1 \leq \min_i Y_i).$ The same likelihood is obtained in the model, where we observe $n$ i.i.d. copies of $Y=a_1+ \eps$ with $\eps \sim \Exp(1).$ This establishes the equivalence between Theorem \ref{thm.BvM} and Theorem 1.1 in \cite{2012arXiv1210.6204K} for $K_n=1.$


\subsection{A specific semi-parametric Bernstein-von Mises result}

We study the Bernstein-von Mises phenomenon and frequentist coverage of credible sets for the functional $\vt =\int f,$ which serves as a prototype of a linear functional of $f$. For the class of piecewise constant functions, the MLE is $\widehat \vt^{\MLE} = \int \widehat f^{\MLE}.$
By the explicit law of $\widehat\vt^{\MLE}$, we can derive $\widehat \vt^{\MLE} - \vt =K_n/n +O_P(\sqrt{K_n}/n).$ The MLE has thus rate of convergence $K_n/n$ whereas the bias corrected estimator $\widehat \vt = \int \widehat f^{\MLE} - K_n/n$ attains the faster rate $O_P(\sqrt{K_n}/n)$.

The bias correction term $K_n/n$ can also be derived from \eqref{eq.fctal_bias_correction} since there are almost surely $K_n$ points on the MLE for the parameter space $\PC^*(K_n, (t_0, \ldots, t_{K_n}), R).$

\begin{cor}
\label{cor.BvM_fctal}
Consider the prior \eqref{eq.prior_PC*} and work under the assumptions of Theorem \ref{thm.BvM}. Denote by $\Pi(\vartheta \in \cdot |N)$ the marginal posterior of the integral $\vartheta := \int f.$ If $K_n \rightarrow \infty,$ then
\begin{align*}
	\sup_{f_0\in \PC^*(K_n, (t_0, \ldots, t_{K_n}), R) } E_{f_0} \Big[ \Big\| \Pi( \vartheta  \in \cdot \,| N ) - \mathcal{N}\Big( \widehat \vt^{\MLE} - \frac{K_n}{n}, \frac{K_n}{n^2}\Big)\Big\|_{\TV}\Big]\rightarrow 0.
\end{align*}
The asymptotic $(1-\alpha)$-credible interval
\begin{align}
	I(\alpha) = \Big[ \widehat \vt^{\MLE} - \frac{K_n}{n} + \frac{\sqrt{K_n}}{n}\Phi^{-1}\big(\alpha/2\big),
	 \widehat \vt^{\MLE} - \frac{K_n}{n} + \frac{\sqrt{K_n}}{n}\Phi^{-1}\big(1-\alpha/2\big) \Big]
	 \label{eq.def_I(alpha)}
\end{align}
is moreover an honest asymptotic confidence set,
\begin{align*}
	\sup_{f_0\in \PC^*(K_n, (t_0, \ldots, t_{K_n}), R) }\Big|
	P_{f_0} \big( \textstyle\int f_0 \in I(\alpha) \big)-\big(1-\alpha\big)\Big|\to 0.
\end{align*}
\end{cor}

One of the interesting consequences of this result is that asymptotically for $K_n \rightarrow \infty$ the credible set does  not contain the MLE $\widehat \vartheta^{\MLE}=\int \widehat f^{\MLE}.$ Hence, the posterior distribution automatically corrects for the bias and is not misguided by the high values of the likelihood around $\widehat \vartheta^{\MLE}.$

\subsection{A negative result on frequentist coverage of credible sets under model misspecification}

We have shown that credible sets are asymptotic confidence sets for priors on the space of piecewise constant functions, provided the true function $f_0$ is also piecewise constant. In the frequentist estimation theory it is known that for Lipschitz-continuous functions $f_0$ a bias-corrected MLE over piecewise constant functions at jump points $t_j=j/K_n$ remains rate-optimal if $K_n\asymp n^{1/2}$, cf. the block-wise estimator in \cite{reiss2014}.
In the same spirit, the nonparametric Bayes result in Section \ref{sec.post_contr} establishes good posterior contraction rates for Lipschitz functions given a CPP prior generating piecewise constant functions. As we shall see here, the automatic bias correction by a Bayes method fails in the case of piecewise linear functions. A consequence is that credible sets may have asymptotically no frequentist coverage at all.

We consider the same piecewise-constant prior as in the previous subsection with jump locations $t_j =j/K_n$ and study the limiting shape of the posterior for data generated by a piecewise linear support boundary
\begin{align}
	f_0(x)=x+ \sum_{j=1}^{K_n} a_j \mathbf{1}\Big(x\in \Big[\frac{j-1}{K_n}, \frac j{K_n}\Big)\Big).
	\label{eq.def_piecewise_lin}
\end{align}

As a benchmark result for Bayesian procedures, we show in a first step that there exists a frequentist method which does equally well for piecewise constant functions, but is also able to return converging confidence sets if the true function is of the form \eqref{eq.def_piecewise_lin}. Consider the space of piecewise 1-Lipschitz functions
\begin{align*}
	\Lip_{K_n}=\Big\{ f: |f(x) -f(y)| \leq |x-y|, \ \forall  x,y \in \Big(\frac{j-1}{K_n}, \frac{j}{K_n}\Big] ,\ \forall j=1,\ldots,K_n\Big\}
\end{align*}
and notice that this space contains all piecewise constant functions as well as  all functions of the form \eqref{eq.def_piecewise_lin}.

\begin{lem}
\label{lem.freq_CI}
Let $K_n \geq 1,$ then for $0< \alpha <1$ there exists a frequentist confidence interval $C(\alpha)$ such that
\begin{align*}
	\inf_{f_0 \in \Lip_{K_n}} P_{f_0}\Big( \int f_0(x) dx \in C(\alpha) \Big) \geq  1-\alpha
\end{align*}
and
\begin{align*}
	\sup_{f_0 \in \Lip_{K_n}} \operatorname{length}\big( C(\alpha)\big) \lesssim \frac{\sqrt{K_n}}{n} + \frac{1}{\sqrt{K_nn}}.
\end{align*}
\end{lem}

For $K_n \geq \sqrt{n},$ a Bayesian credible set should therefore contract with the rate $\sqrt{K_n}/n$ even if the function is piecewise linear. If $K_n \rightarrow \infty,$ we shall see that the marginal posterior distribution still converges in the Bernstein-von Mises sense to
\begin{align*}
	\mathcal{N}\Big( \widehat \vartheta^{\MLE} -\frac {K_n}n , \frac{K_n}{n^2}\Big)
\end{align*}
where $\widehat \vartheta^{\MLE}$ is the MLE over the space of piecewise constant functions. As in the previous section,
\begin{align}
	I(\alpha) = \Big[ \widehat \vt^{\MLE} - \frac{K_n}{n} + \frac{\sqrt{K_n}}{n}\Phi^{-1}\big(\alpha/2\big),
	 \widehat \vt^{\MLE} - \frac{K_n}{n} + \frac{\sqrt{K_n}}{n}\Phi^{-1}\big(1-\alpha/2\big) \Big]
	 \label{eq.def_I(alpha)_RECALL}
\end{align}
is therefore a $(1-\alpha)$-credible set. This means that the posterior credible set contracts with the correct rate $\sqrt{K_n}/n.$ But if $\sqrt{n} \leq K_n,$ Proposition \ref{prop.MLE_misspecified} shows that for piecewise linear support boundary of the form \eqref{eq.def_piecewise_lin}, $\Var_{f_0}(\widehat \vt^{\MLE} )\lesssim K_n/n^2$ and
\begin{align*}
	E_{f_0}\Big[ \widehat \vt^{\MLE} - \frac {K_n}n \Big] \leq \int f_0 - \frac{n}{2^7K_n^3}.
\end{align*}
For $K_n =o(n^{4/7}),$ we find $\sqrt{K_n}/n = o(n/K_n^3)$ and consequently $I(\alpha)$ does not cover the true parameter $\int f_0$ for all sufficiently large $n.$ This implies that asymptotically the credible set has zero coverage.

The next theorem gives the precise conditions. Since this is a negative result it is sufficient to work with one specific prior. For technical convenience, we consider a uniform prior on the function values that allows for a wider range of $K_n$ as in Theorem \ref{thm.BvM}.

\begin{thm}
\label{thm.credible_not_conf}
Let $f_0(x)=x$ and $\vt_0 = \int_0^1 f_0(x) dx =1/2.$ Consider the prior \eqref{eq.prior_PC*} with $t_j= j/K_n,$ $K_n \leq n/\log n,$ $K_n \rightarrow \infty,$ and $g(a_i) = (2R)^{-1}\mathbf{1}_{[-R,R]}$ for fixed $R>3.$ Then,
\begin{align*}
	E_{f_0} \Big[ \Big\| \Pi( \vartheta  \in \cdot \,| N ) - \mathcal{N}\Big( \widehat \vt^{\MLE} - \frac{K_n}{n}, \frac{K_n}{n^2}\Big)\Big\|_{\TV}\Big]\rightarrow 0
\end{align*}
and $I(\alpha)$ as defined in \eqref{eq.def_I(alpha)_RECALL} is an asymptotic $(1-\alpha)$-credible set. On the other hand, if $K_n=o(n^{4/7})$ and $\rho_n= 2^{-8}(nK_n^{-3/2} \wedge n^2K_n^{-7/2}),$ then
\begin{align}
	P_{f_0} \Big( \vt_0 \leq \widehat \vt^{\MLE} - \frac{K_n}{n} + \frac{\sqrt{K_n}}{n} \rho_n \Big) \rightarrow 0
	\label{eq.credible_set_overshoot}
\end{align}
and in particular $I(\alpha)$ has asymptotically no frequentist coverage:
\begin{align*}
	P_{f_0} \big( \vt_0 \in I(\alpha) \big) \rightarrow 0.
\end{align*}
\end{thm}

In parametric models a similar phenomenon has been observed in the case of model misspecification, cf. \cite{Kleijn2012}. For nonparametric models, it is sometimes possible to take a ball that covers $1-\alpha$ of the posterior mass and to show that enlarging the radius of the ball by a constant, results in frequentist coverage tending to one, cf. \cite{Szabo2015}. The result above implies that in order to achieve frequentist coverage the radius needs to be multiplied by a sequence that tends to infinity with polynomial rate in the sample size. If $K_n \asymp \sqrt{n}$ the blow-up factor needs to be at least of the order $n^{1/4}.$

If the Gaussian white noise model is considered with the same prior, one can show that even for linear functions, credible sets  form asymptotic confidence sets. The main reason is that in a  model of piecewise constant functions the sample means  on each block form a sufficient statistic in the Gaussian white noise model, while in the support boundary detection model the sufficient statistics are the blockwise sample minima. The law of the sample mean in a Gaussian shift model is the same for a constant function and a linear  function with the same mean. The law of the sample minimum for a linear boundary, however, deviates significantly from that for a constant boundary.

\section{Limiting shape of the posterior for the CPP prior}\label{SecBvMCPP}

Generalizing the last chapter, we now consider the CPP prior. Compared to the random histogram prior, the difficulties lie in the additional mixing over the model size and the randomness of the jump locations. For the model size, we show that the full posterior concentrates on the true number of jumps under minimal signal strength assumptions. The randomness of the jump locations induces additional randomness of the limiting shape.

We study support boundaries that are piecewise constant and monotone. This function class has received a lot of attention recently in nonparametric statistics, see \cite{Gao2017, chatterjee2015}. Due to the imposed monotonicity, the nonparametric MLE exists and we believe that this is crucial for the posterior to have a tractable limit distribution, see also Section \ref{sec.BvM_general}.

\subsection{The limiting shape of the full posterior}

We first derive the limiting shape of the full posterior and then study the marginal distribution of the functional $\vartheta = \int f.$

{\bf Model.} The likelihood taken over all increasing functions on $[0,T]$ is unbounded. This is caused by functions that have an extremely steep jump close to the right boundary of the observation interval $[0,T].$ Similar boundary phenomena are well-known in the nonparametric maximum likelihood theory under shape constraints. The unboundedness of the likelihood causes the Bayes formula to be extremely sensitive to values close to the right boundary. Since we are interested in a framework that avoids these extreme spikes at the boundary we therefore consider the PPP model \eqref{eq.mod} with $T> 1$ assuming that the true function is constant on the interval $[1,T].$ For jump functions, this is the same as saying that all jumps occur before time one.

{\bf Function class.} We consider piecewise constant,  right-continuous functions that are monotone increasing assuming that all jumps occur up to time one:
\begin{align*}
	\mM(K, R) := \Big\{f = \sum_{\ell=0}^{K} a_\ell \mathbf{1}_{[t_\ell,T]} \, : \, 0 \leq a_\ell \leq R, \,  0\leq t_1 \leq \hdots \leq t_{K_n} \leq 1\Big\}.
\end{align*}
For a generic function in $\mM(K, R)$ we write $f = \sum_{\ell=0}^{K} a_\ell \mathbf{1}(\cdot \geq t_\ell)$ with ordered jump locations $0=: t_0 \leq t_1 \leq \hdots \leq t_{K} \leq 1<  t_{K+1}:=T.$ We assume that there is a minimal signal strength. Without such a constraint one cannot exclude the case that the number of true jumps is consistently underestimated, see for instance \cite{Frick2014}, Section 2.1. Typically, conditions of this type occur when there is an underlying model selection problem, compare with the $\beta$-min conditions for high-dimensional problems.

\begin{defi}
\label{assump.min_rect}
A function $f_0 \in \mM(K_n, R)$ belongs to the subclass $\mM_S(K_n, R)$ if and only if for all $k=1,\ldots,K_n$
\begin{align*}
	a_k^0 (t_{k+1}^0 -t_k^0) \wedge a_k^0 (t_k^0 -t_{k-1}^0)  \geq 2K_n \log (eK_n) \frac{ \log^3 n}{n},\quad
	 a_k^0 \geq \frac{2\log n}{\sqrt n}, \quad
		(t_{k+1}^0 -t_k^0) \geq \frac{2}{\sqrt n},
\end{align*}
and the two last inequalities also hold for $k=0.$
\end{defi}

\begin{rem}\label{RemMS}
Since $\sum_{k=0}^{K_n-1}(t_{k+1}^0-t_k^0)\le 1$, the last condition implies implicity $K_n={\cal O}(n^{1/2})$. In view of $\max_ka_k^0\le R$ the first condition even implies $K_n^2\log(eK_n)\le R n/\log^3(n)$, in particular $K_n=o(n^{1/2})$. This is the same condition as in the case of smooth random histogram priors.
\end{rem}

The expressions $a_i^0 (t_{i+1}^0 -t_i^0)$ and $a_i^0 (t_{i}^0 -t_{i-1}^0)$ are the areas in Figure \ref{fig.intro}(right). Let us briefly discuss the imposed lower bound on these areas. Since the PPP has intensity $n$ on the epigraph of the support boundary, in order  to ensure that each of the $K_n$ sets contains at least one support point of the PPP, all of them need  to have an area of at least order $\log(K_n)/n$. One might therefore wonder whether the factor $K_n$ in the lower bound for the areas is necessary to ensure strong model selection. We shall see that the posterior has to choose among a huge number of models, cf. the proof of Proposition \ref{prop.MLE_with_Kn_jumps}. To find the correct model might therefore indeed require a larger lower bound on the areas.

{\bf Prior.} By assumption all jumps occur before time one. We therefore draw the prior from a CPP on $[0,1]$ and then extend it continuously to a prior on $[0,T]$ by appending a constant function on $(1,T].$  The Lebesgue density of $(t_1, \ldots, t_K)|K$ is $K!\mathbf{1}(0\leq t_1 \leq t_2 \leq \ldots \leq t_K \leq 1),$ see Section \ref{sec.post_contr}. To model the monotonicity, the process should have positive jumps and thus the jump distribution should be supported on the positive real line. It turns out that there is one natural prior on the jump sizes. The construction is as follows: choose the random starting value of the CPP according to $a_0 \sim \Exp(1)$ and independently draw i.i.d. jump sizes $a_\ell \sim \Gamma(2,1)$ for $\ell=1, \ldots, K.$ With
\begin{align}
	g_K(\ba) =  e^{-\sum_{k=0}^K a_k} \prod_{k=1}^K a_k,\quad \ba=(a_0,\ldots,a_K)\in\R_+^{K+1}
	\label{eq.g_def}
\end{align}
the prior \eqref{eq.CPP_explicit} takes therefore the more specific form
\begin{align}
	(K, \bt, \ba) \mapsto e^{-\lambda} \lambda^K g_K(\ba)  \mathbf{1}\big(0\leq t_1 \leq t_2 \leq \ldots \leq t_K \leq 1 \big).
	\label{eq.g_def2}
\end{align}
We can also rewrite the prior as a prior on functions of the form $f= \sum_{k=0}^{K} b_k\mathbf{1}_{[t_k,t_{k+1})}.$ Under this reparametrization, we obtain $g_K(\bb) = e^{-b_K} \prod_{k=1}^K (b_k-b_{k-1})_+.$

Since $f(0)=a_0,$ this means in particular that all paths generated by the prior are non-negative. To put different priors on $a_0$ and $a_\ell,$ $\ell \geq 1,$ turns out to be natural. For this specific choice the marginal posterior of any $a_k$ follows approximately an exponential distribution. This is a crucial property that allows us to derive tight bounds for the numerator and denominator in the Bayes formula, see also the proofs of Lemma \ref{lem.denom} and Lemma \ref{lem.ub_num} for more details.

{\bf MLE.} Over all monotone functions on $[0, T], T>1,$ that are constant on $[1,T],$ there exists a nonparametric MLE $\widehat f^{\MLE}$ (unique almost surely). Existence follows from the general theory because the class of monotone functions is closed under the maximum, see \cite{reiss2014}. Almost surely, the MLE is piecewise constant with finitely many jumps and bounded. This implies in particular, that $\widehat f^{\MLE}$ is also the MLE over all piecewise constant monotone functions with jumps on $[0,1].$ Furthermore $f\leq  \widehat f^{\MLE}$ for all piecewise constant and monotone functions satisfying $f(X_i) \leq Y_i$ for all $i.$ Denoting the number of jumps by $M,$ we write
\[\widehat{f}^{\MLE}(t)= \sum_{\ell=0}^M \widehat a_\ell^{\MLE}\mathbf{1}(t\geq \widehat t_\ell^{\MLE}), \quad t\in [0,T]\]
with $0=: \widehat t_0^{\MLE}<\widehat t_1^{\MLE} < \cdots<\widehat t_M^{\MLE} \leq 1.$ This MLE should not be confused with the monotone MLE on $[0,T]$ without the restriction that the functions are constant on $[1,T].$

{\bf Construction of the majorant process $\widetilde f$.} We consider two sequences of observation points that are close to the true jump points of the unknown regression function $f_0.$ Recall that $t_0^0=0$ and $t_{K_n+1}^0=T.$ For $k=0,1 \ldots, K_n,$ consider
\begin{align}
	(X_k^*, Y_k^*)  := \argmin_{(X_i, Y_i) \ \text{observation point}} \big\{ Y_i : X_i \in [t_k^0, t_{k+1}^0) \big\}
	\label{eq.*seq_def}
\end{align}
and for $k=1, \ldots, K_n,$ with  $ R_k:=\big\{ (X_i,Y_i) \text{ observation} : X_i \in [t_{k-1}^0, t_k^0), Y_i \leq f_0(t_k^0) \big\}$
\begin{align}
	(X_k', Y_k')  := \begin{cases} \argmax_{(X_i, Y_i)}\{X_i:(X_i,Y_i)\in R_k\} ,&\text{ if } R_k\not=\varnothing\\ (t_{k-1}^0, f_0(t_{k-1}^0)),&\text{ otherwise.}\end{cases}
	\label{eq.'seq_def}
\end{align}

\begin{figure}
\begin{center}
	\includegraphics[scale=0.42]{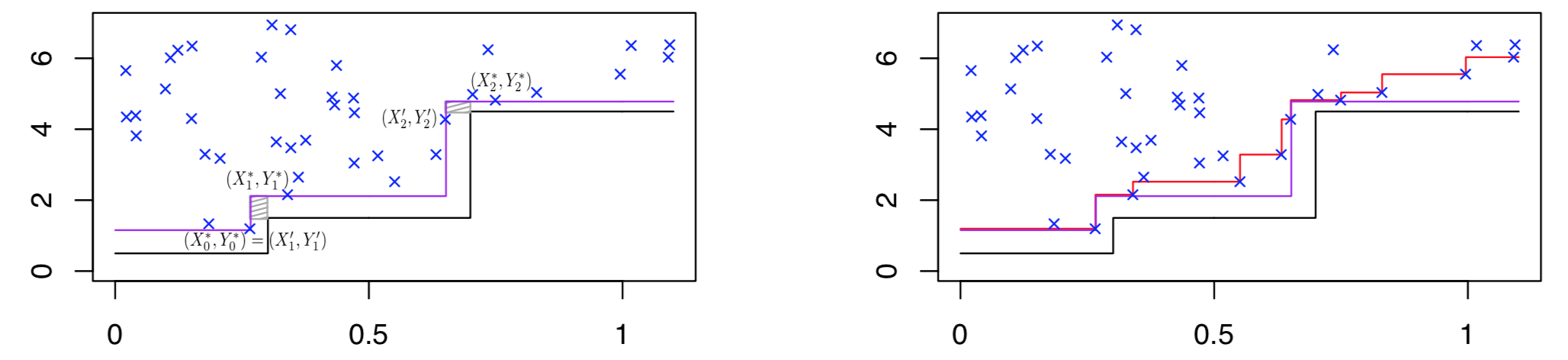}
	\caption{\label{fig.MLE} Left: Data example with true boundary (black), the function $\widetilde f$ (purple) and the sequences $(X_k^*,Y_k^*),$ $(X_k',Y_k').$ Right: If none of the observations fall into the gray areas then the sequences $(X_k^*,Y_k^*),$ $(X_k',Y_k')$ lie on the MLE over monotone functions (red).}
\end{center}
\end{figure}

We also set $X_0':=0$ and $X_{K_n+1}':=T.$ With probability one, the sequences are unique, see also Figure \ref{fig.MLE}. The assigned values for the case $R_k =\varnothing$ do not affect the asymptotic analysis but are convenient choices giving the guarantee that the subsequent formulas are well-defined. The key object for the limiting shape result of the posterior is the process
\begin{equation}\label{eq.widetildef_def}
	\widetilde f = \sum_{k=0}^{K_n}Y_k^* \mathbf{1}_{[X_k',X_{k+1}')},
\end{equation}
 a realization of which is displayed in Figure \ref{fig.MLE}. Since $\widetilde f \geq f_0,$ we call $\widetilde f$ also the \textit{majorant process} (of $f_0$). Observe that the majorant process is piecewise constant with $K_n$ jumps. As the support boundary is unknown, the majorant process cannot be computed from the data alone. As we shall see next, $\widetilde f$ coincides asymptotically with the MLE over monotone functions with the correct number of $K_n$ jumps.

\begin{prop}
\label{prop.MLE_with_Kn_jumps}
If $\widehat f^{\MLE}_{K_n}$ denotes the MLE in the space $\mM(K_n,\infty),$ then
\begin{align*}
	\inf_{f_0\in \mM_S(K_n,R)} P_{f_0}\big( \widetilde f = \widehat f^{\MLE}_{K_n} \big) \to 1.
\end{align*}
In particular, $\inf_{f_0\in \mM_S(K_n,R)} P_{f_0}( \widetilde f \ \text{is monotone}\, ) \to 1.$
\end{prop}

{\bf Limit distribution.} We now describe the sequence of distributions that asymptotically approximates the posterior. For convenience we ignore the dependence on $n$ and refer to this sequence as the limit distribution. Working conditionally on the sequences $(X_k')_k$ and $(Y_k^*)_k$ the limit distribution $\Pi_{f_0,n}^\infty$ is then the distribution on the Skorokhod space $D([0, T])$ of
\begin{align}
	f=\sum_{k=0}^{K_n} (Y_k^*-E_k^*)\mathbf{1}_{[X_k' +E_k', X_{k+1}'+E_{k+1}')}
	\label{eq.limit_distr_2_def}
\end{align}
with independent $E_k^* \sim \Exp(n(X_{k+1}'-X_k'))$ and $E_k' \sim \Exp(n(Y_k^*-Y_{k-1}^*)) \wedge (X_{k+1}'-X_k')$, $k\leq K_n$, and $E_0':=E_{K_n+1}':=0.$

Given  the majorant process $\widetilde f,$ we can  draw from the limit distribution by moving each jump location independently to the right by a (truncated) exponential distribution with scale parameter $n(Y_k^*-Y_{k-1}^*).$ Moreover, the function value on each piece is decreased by another independently generated exponential random variable. With Proposition \ref{prop.MLE_with_Kn_jumps} it follows that the limit is of the generalized form discussed in Section \ref{sec.BvM_general}.

\begin{thm}[Limiting shape result for CPP prior]
\label{thm.BvM2}
Let $K_n \leq n^{1/2-\delta}$ for some $\delta>0.$ For the prior \eqref{eq.g_def} and $\Pi_{f_0,n}^\infty$ as defined in \eqref{eq.limit_distr_2_def},
\begin{align*}
	\lim_{n\to\infty}\sup_{f_0\in \mM_S(K_n,R)} E_{f_0}^n \Big[\big\| \Pi(\cdot | N) -\Pi_{f_0,n}^\infty \big\|_{\TV}\Big] = 0.
\end{align*}
\end{thm}

Since we work with one specific prior, we call this a limiting shape result instead of a Bernstein-von Mises theorem. Using \eqref{eq.limit_distr_2_def}, one can show that the posterior contracts with rate $K_n/n.$ We conjecture that the MLE only achieves the slower rate $K_n \log n /n.$ One of the heuristic reasons  is that the MLE overshoots the true model dimension $K_n$ by choosing a model with order $K_n \log n$ many jumps, see Figure \ref{fig.MLE} and Lemma \ref{lem.ub_jumps_of_MLE}.  It is conceivable that each of the additional jumps introduces an error of size $1/n$ which then gives the rate $K_n \log n /n.$ A similar phenomenon occurs in the nonparametric regression model, see  Prop. 2.1 in \cite{Gao2017}.

The proof is non-standard. It follows immediately from the likelihood that the posterior only puts mass on paths that lie below the monotone MLE $\widehat f^{\MLE}.$ Let $f$ be a piecewise constant function with $K$ jumps such that there exists a function $f_{>}$ with $K-1$ jumps such that $f \leq f_{>}\leq \widehat f^{\MLE}.$ Interestingly, it can be shown that the posterior puts negligible mass on the union over all such functions and all $K.$ The remaining paths have more structure. We use this to introduce a parametrization from which we can derive sufficiently sharp bounds over the corresponding integrals in the Bayes formula. The proof also requires many properties of the monotone MLE which might be of independent interest and are collected in Appendix \ref{AppMLE}.

\subsection{A positive result on frequentist coverage of functionals}

For the functional $\vartheta = \int_0^T f,$ we have under the limit distribution $\Pi_{f_0,n}^\infty,$
\begin{align}
	\vartheta = \int_0^T \widetilde f - \sum_{k=0}^{K_n} E_k^* (X_{k+1}'-X_k')
	- \sum_{k=1}^{K_n} E_k' (Y_k^*-Y_{k-1}^*)
	- \sum_{k=0}^{K_n} E_k^* (E_{k+1}'-E_k').
	\label{eq.vartheta_under_limit}
\end{align}
In this section, we show that this converges to a normal distribution with mean $\int \widetilde f - (2K_n+1)/n$ and variance $(2K_n+1)/n.$ Given two probability measures $P,Q$ on $(\mathbb{R}, \mathcal{B}(\mathbb{R})),$ define the Kolmogorov-Smirnov distance
\begin{align*}
	\|P-Q\|_{\KS} := \sup_{x\in \mathbb{R}} \big| P\big((-\infty,x]\big) - Q\big((-\infty,x]\big) \big|.
\end{align*}

\begin{cor}
\label{cor.marg_BvM_CPP}
Consider the prior \eqref{eq.g_def}. Then, for any sequence $K_n\to\infty$ with  $ K_n \leq n^{1/2-\delta}$ for some $\delta>0$
\begin{align*}
	\sup_{f_0\in \mM_S(K_n,R)} E_{f_0}^n \Big[\Big\| \Pi(\vartheta \in \cdot | N) -\mathcal{N}\Big(\int_0^T \widetilde f - \frac{2K_n+1}{n}, \frac{2K_n+1}{n^2}\Big) \Big\|_{\KS}\Big] \rightarrow 0.
\end{align*}
\end{cor}

By Lemma \ref{prop.MLE_with_Kn_jumps}, the majorant process $\widetilde f$ in the limit distribution can be replaced by $\widehat f^{\MLE}_{K_n}.$ The result is formulated in terms of the Kolmogorov-Smirnov distance, which suffices to describe asymptotic probabilities for credible intervals. It is not clear whether a total variation version holds as well because point masses enter into the proof argument and are difficult to control.

The observations that lie on the majorant process are $(X_k',Y_k'),$ $k=1, \ldots, K_n$ and $(X_k^*,Y_k^*),$ $k=0, \ldots, K_n.$ This means that $2K_n+1$ observations lie on the boundary of $\widetilde f$ (almost surely). The bias correction term $(2K_n+1)/n$ is consequently of the same form as for the bias-corrected MLE in \cite{reiss2014}. We can now argue as in Corollary \ref{cor.BvM_fctal} to construct a $(1-\alpha)$-credible interval that is also an asymptotic $(1-\alpha)$-confidence interval and shrinks with the correct rate $O(\sqrt{K_n}/n).$

\subsection{A negative result on posterior coverage for the CPP prior}
\label{sec.negative_BvM_CPP}

We consider the same statistical model: we observe a PPP on $[0,T] \times \mathbb{R}$ with intensity $\lambda_f(x,y) = n \mathbf{1}(y \geq f(x)).$ We are now interested in the coverage of credible sets if the support boundary function is not piecewise constant. For the specific choice $f_0(x) = (x + 1/2)\wedge 3/2$ of the support boundary function it is shown that the credible sets do not have asymptotic coverage. Notice that $f_0$ is constant on $[1,T].$

{\bf Class of priors.} Consider a (generalized) CPP prior. Given the number of jumps $K$, the jump heights $\ba=(a_0,a_1, \ldots, a_K)$ are assumed to be independent but not necessarily identically distributed and the prior is of the form
\begin{align}
	g_K(\ba) = \prod_{k=0}^K g_k(a_k).
	\label{eq.gK_recall}
\end{align}
For the marginal prior on the individual jumps we assume that there exist constants $c>0,$ $\gamma \geq 0,$ such that
\begin{align}
	g_k(x) \geq  cx^\gamma, \quad \forall \ x\in [0,1] , \  k\ge 0.
	\label{eq.gk_lb}
\end{align}
In particular this is satisfied by the prior \eqref{eq.g_def} with $\gamma=1$ and $c=e^{-1}.$

\begin{figure}
\begin{center}
	\includegraphics[scale=0.6]{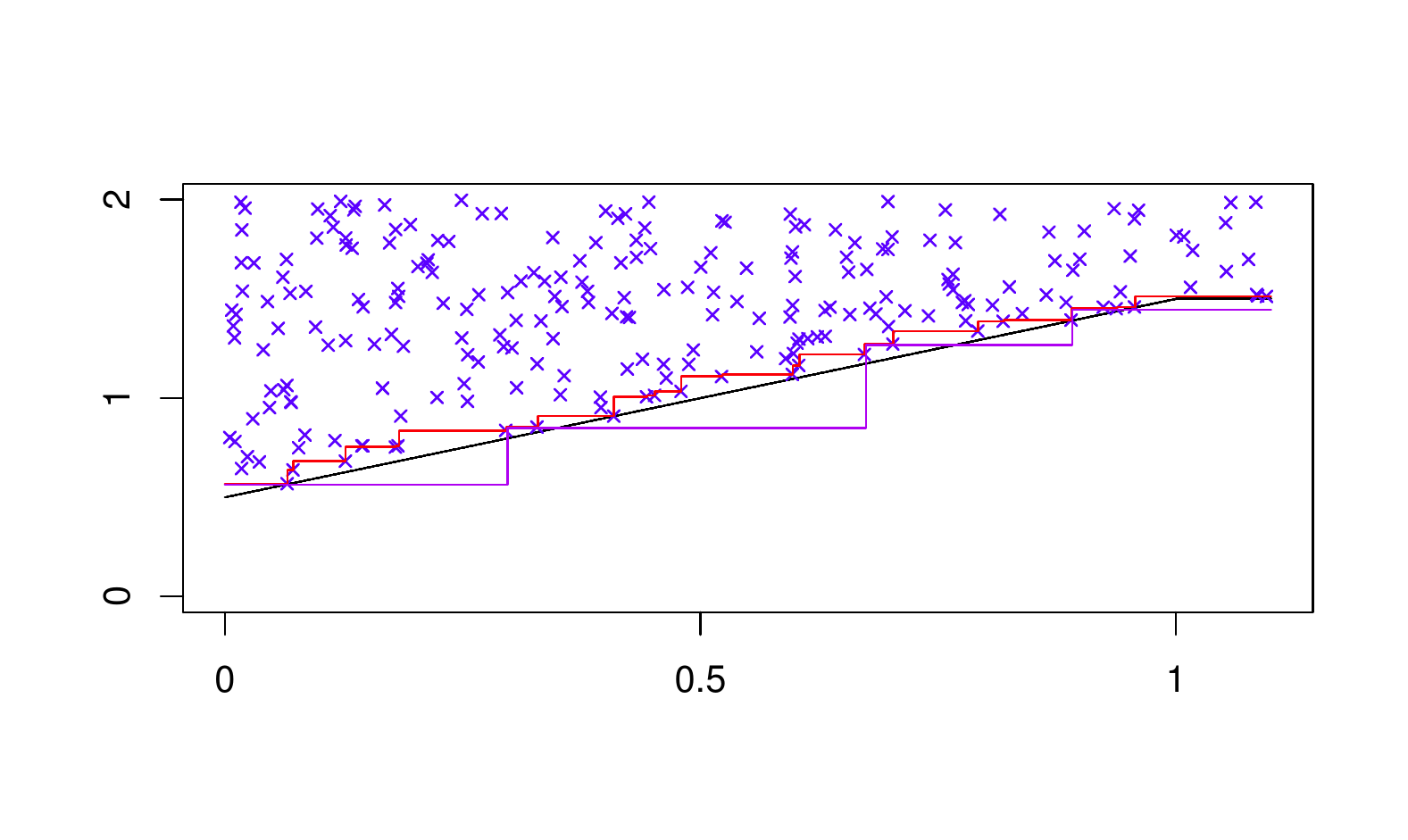}
	\vspace{-1cm}
	\caption{\label{fig.lb} The argument for the lower bound with monotone MLE (red), true function (black) and a function $\fMLEKs$ with few jumps (purple). The posterior puts asymptotically all mass on paths with much fewer jumps than the monotone MLE. This creates a downwards bias of the posterior for the marginal posterior of the integral $\int_0^1 f$.}
\end{center}
\end{figure}

 The first result shows that the posterior concentrates on models with size $\sqrt{n/\log n}.$ This is of a slightly smaller order than the MLE, which has of the order $\sqrt{n}$ many jumps. This causes then a downwards bias of the posterior, compare Figure \ref{fig.lb}. Interestingly, a similar phenomenon occurs in the Gaussian white noise model, cf. Prop. 2 in \cite{castillo2015b}.

\begin{prop}
\label{prop.jump_nr_lb}
Consider a CPP prior with jump distribution satisfying \eqref{eq.gK_recall} and \eqref{eq.gk_lb}. For $f_0 = (\tfrac 12 + \cdot) \wedge \tfrac 32$ there exists a positive constant $c_*$ such that
\begin{align*}
	E_{f_0}\Big[\Pi\Big(K\geq  c_* \sqrt{\frac  n {\log n}} \, \Big | N \Big) \Big] \rightarrow 0.
\end{align*}
\end{prop}

\begin{thm}
\label{thm.marg_BvM2}
If $f_0 = (\cdot + 1/2) \wedge 3/2,$ then there exists a positive constant $c^*,$ such that for the marginal posterior on the functional $\vartheta=\int_0^1 f$
\begin{align*}
	E_{f_0}\Big[\Pi\Big( \vartheta \geq \int_0^1 f_0(x) dx - \widetilde c\sqrt{\frac{\log n}{n}} \Big | \, N \, \Big)\Big] \rightarrow 0.
\end{align*}
This means that the whole posterior mass lies asymptotically below the true value.
\end{thm}

We conjecture that the negative result continues to hold in the case of piecewise constant functions with at least $\sqrt{n}$ jumps because the posterior will put all asymptotic mass on models of dimension $O(\sqrt{n/\log n})$, underestimating the number of true jumps by at least a logarithmic factor.

\appendix

\section{Proofs for Section \ref{sec.post_contr}} \label{AppA}

Denote by $N(\eps, \mF, d)$ the $\eps$-covering number of $\mF \subset L^1([0,1])$ with respect to the distance $d.$ The one-sided bracketing number $N_{[}(\delta, \mF)$ is the smallest number $M$ of functions $\ell_1,\ldots,\ell_M\in L^1([0,1])$ such that for any $f\in \mF$ there exists $j\in \{1,\ldots, M\}$ with $\ell_j\leq f$ (almost everywhere) and $\int(f-\ell_j)\leq \delta.$ The functions $\ell_j$ are not required to be in $\mF.$

\begin{thm}[Theorem 2.3 and Corollary 2.6 in \cite{reiss2018a}]
\label{thm.main_ub}
If for some $\Theta_n \subset \Theta,$ some rate $\eps_n \rightarrow 0$ and constants $C, C',C''\geq 1$, $A>0$

\begin{tabular}{ll}
 (i) & \quad
   $N_{[}\big(\eps_n, \Theta_n\big) \leq C'' e^{C'n \eps_n};$ \\[0.3cm]
  (ii) & \quad $\Pi(f: \|f-f_0\|_1 \leq A\eps_n, f\leq f_0) \geq e^{-Cn\eps_n};$ \\[0.3cm]
  (iii) & \quad $\Pi(\Theta_n^c)\leq C''e^{-(C+A+1) n \eps_n},$ \\
\end{tabular}

then there exists a constant $M$ such that
\begin{align*}
	&E_{f_0}\big[\Pi\big( f : \| f-f_0\|_1 \geq M\eps_n | N \big) \big]  \leq 3C'' e^{- n\eps_n}.
\end{align*}
\end{thm}

\subsection{Proof of Theorem \ref{thm.contr_CPP}}

It is convenient to use the notation $\P(X\in A) := \Pi(f \in A)$ to prove generic properties of the compound Poisson process $X$ defined in \eqref{EqCPP}.

\begin{lem}
\label{lem.SB_for_CPP}
Consider the CPP prior \eqref{EqCPP} with a positive and continuous Lebesgue density $h$ on $\mathbb{R}.$
\begin{itemize}
\item[(i)] If $g$ is positive and continuous on $\mathbb{R}^+,$ there exists a positive constant $c = c(R),$ such that
\begin{align*}
	\inf_{f\in \mM(R)} \P\big( \| X-f \|_1 \leq 2\eps, X\leq f\big) \geq e^{-2\lambda} (1\wedge \lambda)^{4R/\eps}\eps^{c \eps^{-1}}, \quad \text{for all} \ \ 0<\eps\leq R \wedge \tfrac 12.
\end{align*}
\item[(ii)] If $g$ is positive and continuous on $\mathbb{R},$ then for $0< \beta \leq 1$ there exists a positive constant $c = c(\beta, R)$ such that
\begin{align*}
	\inf_{f\in \mC^\beta(R)}\P\big( \| X-f \|_\infty \leq \eps \big) \geq e^{-2\lambda} (1\wedge \lambda)^{(4R/\eps)^{1/\beta}} \eps^{c\eps^{-1/\beta}} \quad \text{for all} \  \ 0<\eps\leq \tfrac {R\wedge 1}2;
\end{align*}
\item[(iii)] If $g$ is positive and continuous on $\mathbb{R},$ then, there exists a positive constant $c=c(R)$ such that
\begin{align*}
	\inf_{f\in \PC(K_n,R)} \P\big( \| X-f \|_1 \leq \eps, X\leq f\big) \geq e^{-\lambda} (1\wedge \lambda)^{K_n} \Big( \frac{\eps}{K_n} \Big)^{cK_n}, \quad \text{for all} \ \ 0 < \eps\leq \tfrac 12.
\end{align*}
\end{itemize}
\end{lem}

{\it Proof of $(i):$} For fixed $f\in \mM(R),$ we construct a deterministic step function $f_-$ with $f_-\leq f$ and $\|f_- -f\|_1 \leq \eps.$ It is then enough to show
\begin{align}
	\P\big( \| X-f_- \|_1 \leq \eps, X\leq f_-\big) \geq e^{-2\lambda} (1\wedge \lambda)^{4R/\eps} \eps^{c/\eps}, \quad \text{for all} \ 0< \eps \leq R \wedge 1/2.
	\label{eq.SB_MR_to_show}
\end{align}
If $\eps \leq R,$ there exists $\delta$ such that $\eps/(4R)\leq \delta \leq \eps/(2R)$ and $N:=1/\delta$ is a positive integer. Let $r(j, \delta):= f(j\delta) - f((j-1)\delta)$ for $j\geq 1.$ Define the step functions
\begin{align*}
	f_- := \sum_{j=0}^{N-1} f(j\delta)\mathbf{1}_{[j\delta,(j+1)\delta)}
	= f(0) + \sum_{j=1}^{N-1} r(j, \delta) \mathbf{1}_{[j\delta,1]}
\end{align*}
and $f_+ := \sum_{j=1}^{N} f(j\delta)\mathbf{1}_{[(j-1)\delta,j\delta)}.$ Since $f$ is monotone increasing, $f_- \leq f\leq f_+$ and $\|f - f_-\|_1 \leq \|f_+-f_- \|_1 = \delta (f(1)-f(0))\leq \eps.$ By the assumptions on $g$ and $h,$ $c_0 := \inf_{-R-1 \leq x \leq R} h(x) \wedge \inf_{0 \leq y \leq R+1} g(y)$ is positive. Due to \eqref{eq.CPP_explicit} and $e^{-\lambda}/\lambda \geq e^{-2\lambda},$
\begin{align*}
	&\P\big( \| X-f_- \|_1 \leq \eps, X\leq f_-\big)\\
	&\geq \Pi\Big( k= N-1, f(0)-\eps \leq a_0 \leq f(0)-\frac{\eps}2,  r(j,\delta) \leq a_j \leq r(j,\delta)+  \frac{\eps\delta}{2}, t_j \in \Big[ j\delta, j\delta+ \frac{\eps\delta}{2}\Big]\Big) \\
	&\geq e^{-\lambda} \lambda^{N-1} \Big( c_0\frac{\eps\delta}{2}\Big)^N \Big(\frac{\eps\delta}{2}\Big)^{N-1} \\
	&\geq e^{-2\lambda} (1\wedge \lambda)^{4R/ \eps} \Big( \sqrt{c_0} \frac{\eps^2}{4R}\Big)^{2/\delta},
\end{align*}
where the probability $\Pi$ is taken over all $j=1, \ldots, N-1.$ This yields \eqref{eq.SB_MR_to_show} and proves $(i).$

{\it Proof of $(ii):$} The argument is very similar to $(i).$ Let now $\delta$ be such that $(\eps/(4R))^{1/\beta} \leq \tfrac 12 (\eps/(2R))^{1/\beta} \leq \delta  \leq (\eps/(2R))^{1/\beta}$ and $N:=1/\delta$ is a positive integer. With $r(0,\delta):=f(0)$ and  $r(j, \delta):= f(j\delta) - f((j-1)\delta)$ for $j\geq 1,$ define
\begin{align*}
	f_- := \sum_{j=0}^{N-1} f(j\delta)\mathbf{1}_{[j\delta,(j+1)\delta)}
	= \sum_{j=0}^{N-1} r(j, \delta) \mathbf{1}_{[j\delta,1]}.
\end{align*}
Now, $\delta \leq (\eps/(2R))^{1/\beta}$ and $f\in \mC^\beta(R)$ give $\|f-f_-\|_\infty \leq \eps/2.$ It is thus enough to prove  $\P( \| X-f_- \|_\infty \leq \eps/2)\geq e^{-2\lambda} (1\wedge \lambda)^{(4R/\eps)^{1/\beta}} \eps^{c\eps^{-1/\beta}}.$ By assumption, $g$ and $h$ are continuous and positive and therefore $c_0 := \inf_{-2R-1 \leq x \leq 2R} g(x) \wedge h(x)$ is positive. Due to \eqref{eq.CPP_explicit}, $|r(j, \delta)|\leq 2R$ and $e^{-\lambda}/\lambda \geq e^{-2\lambda},$
\begin{align*}
	&\P\big( \| X-f_- \|_\infty \leq \eps/2\big)\\
	&\geq \Pi\Big( k= N-1, r(j,\delta) - \frac{\eps\delta}{4} \leq a_j \leq r(j,\delta), t_j \in \Big[ j\delta, j\delta+ \frac{\eps\delta}{4}\Big], j=0, \ldots, N-1\Big) \\
	&\geq e^{-\lambda} \lambda^{N-1}  \Big( c_0\frac{\eps\delta}{4}\Big)^N \Big(\frac{\eps\delta}{4}\Big)^{N-1} \\
	&\geq e^{-2\lambda} (1\wedge \lambda)^{2 (2R/ \eps)^{1/\beta}} \Big( \frac{\sqrt{c_0}}{8} (2R)^{-1/\beta} \eps^{\frac{\beta+1}{\beta}}\Big)^{2/\delta}.
\end{align*}
Choosing $c = c(\beta, R)$ large enough, the result follows.

{\it Proof of $(iii):$} Let $f = \sum_{j=1}^{K_n} a_j \mathbf{1}_{[t_{j-1},t_j)}$ be an arbitrary function in $\PC(K_n,R).$ Without loss of generality, we can assume that $R\geq 2.$ Choose $\delta$ such that $\eps/(4RK_n) \leq \delta \leq \eps/(2RK_n)$ and $N:=1/\delta$ is a positive integer. Define $f_-= \sum_{j=1}^N \min_{x\in [(j-1)\delta, j\delta]} f(x) \mathbf{1}_{[(j-1)\delta, j\delta)}.$ Obviously $f_-\leq f$ and $\|f-f_-\|_1 \leq K_n R \delta \leq \eps/2.$ We can then write $f_-= \sum_{j=0}^{K_n^*} b_j^* \mathbf{1}_{[t_j^*, 1)}$ with $K_n^* \leq K_n,$ $|b_j^*|\leq 2R,$ $0= t_0^* < t_1^* < \ldots < t_{K_n^*}^* <1$ and $t_j^*$ a multiple of $\delta$ (only incorporating points $j \delta$ where $f_-$ actually jumps). Let $I_j$ denote the interval with endpoints $t_j^*$ and $t_j^*+ \delta \sign(f_-(t_j^*) -f_-(t_{j-1}^*))/2.$ Let $c_0:= \inf_{-2R-1 \leq x\leq 2R} g(x) \wedge h(x).$ Arguing as in \textit{(i)}, $c_0>0$ and
\begin{align*}
	\P\big( \| X-f_- \|_1 \leq \eps,& X\leq f_-\big)\\
	&\geq \Pi\Big( k= K_n^*, b_j^* - \delta/2 \leq b_j \leq b_j^*, t_\ell \in I_\ell, j=0, \ldots, K_n^*, \ell=1, \ldots, K_n^*\Big)\\
	&\geq e^{-\lambda} \lambda^{K_n^*} \Big(c_0 \frac{\delta}{2}\Big)^{K_n^*+1} \Big( \frac{\delta}{2}\Big)^{K_n^*}\\
	&\geq e^{-\lambda} (1\wedge \lambda)^{K_n} \Big( \frac{\eps}{K_n} \Big)^{cK_n}
\end{align*}
for some $c=c(R).$ \hfill $\qed$\\

\begin{lem}
\label{lem.CPP_entropy}
Consider the randomly initialized CPP \eqref{EqCPP} and assume that there are constants $\gamma,L>0$ such that $\P ( |\Delta_i | \geq s ) \leq L^{-1} e^{-L s^{\gamma}}$ for all $s\geq 0.$ Then for any $M>0$, any $\eps>0,$ and any  $K>1$ there exists a Borel set $\Theta$ and constants $C',C''$ that only depend on $M, L, \gamma,$ such that
\begin{align*}
	\P(X \notin \Theta) \leq C'' K^{-MK} \quad \text{and}  \ \ \ N_{[}\big( \eps, \Theta, \|\cdot \|_1 \big) \leq C'' \Big(\frac{K}{\eps}\Big)^{C'K}.
\end{align*}
\end{lem}

\begin{proof}
If $N \sim \Pois(\lambda)$, $K\ge 1$ and $M \geq \max(2\lambda e, 1),$ then, using Stirling's formula,
\begin{align}
	\P\big( N \geq  M K \big)
	&= e^{-\lambda} \sum_{k= \lceil MK\rceil }^\infty \frac{\lambda^k}{k!}
	\leq \sum_{k= \lceil MK\rceil }^\infty \Big(\frac{\lambda e}{k}\Big)^k
	\leq \sum_{k= \lceil MK\rceil }^\infty \Big(\frac 1{2K}\Big)^k \leq K^{-MK}.
	\label{eq.contr_CPP0}	
\end{align}
With $t:= ((MK+1)L^{-1}\log K)^{1/\gamma}$ and the assumption on the tail behavior of the jump heights, we obtain
\begin{align}
	\P\Big(\{N \ge MK\}\cup\Big\{\max_{i=0, \ldots,N} |\Delta_i| \geq t\Big\}\Big)
	&\leq \P(N \ge MK) + MK \P\big(|\Delta_1| \geq t\big) \notag \\
&  \leq (1+M/L) K^{-MK}.
	\label{eq.contr_CPP1}
\end{align}
Define $\Theta$ as the space of piecewise constant functions $f$ with $|f(0)|\leq t$, maximal jump size bounded by $t$ and less than $MK$ jumps. By the computations above, $\P(X\notin \Theta) \leq (1+M/L) K^{-MK}.$

Next, we compute the bracketing number of $\Theta$ with respect to the $L^1$-norm. Let $r_{\eps}$ be such that $\eps /(4M K t) \leq  r_{\eps} \leq \eps /(2M K t)$ and $1/r_{\eps}$ is an integer. Define $x_j := j r_{\eps}$ for $0\leq j < 1/r_{\eps}.$ In $y$-direction, consider the grid points $y_\ell := \ell \eps/2,$ $\ell = - S_{\eps}, \ldots, S_{\eps}$ with $S_{\eps} = \lfloor 2MK t/\eps \rfloor.$ Let $\Theta^0\subset \Theta$ be the space of piecewise constant functions in $\Theta$ with all jumps locations on the grid points $x_j,$ and function values in the discrete set $\{y_\ell : \ell = - S_{\eps}, \ldots, S_{\eps}\}.$  We prove that for any function $f \in \Theta,$ there exists a function $h\in \Theta^0$  such that $h \leq f$ and $\|h-f\|_1 \leq \eps.$ Consider
\begin{align*}
	h = \sum_{j=1}^{ 1/r_{\eps}} \max\Big \{y_\ell : y_\ell \leq \min_{x\in [x_{j-1}, x_j]} f(x)\Big \} \mathbf{1}_{[x_{j-1}, x_j)}.
\end{align*}
Obviously, $h\in \Theta^0$ and $h \leq f.$ Let us show $\|h-f\|_1 \leq \eps.$ Observe that $\|h-\widetilde h\|_\infty \leq \eps/2$ with $\widetilde h = \sum_{j=1}^{ 1/r_{\eps}}\min_{x\in [x_{j-1}, x_j)} f(x)\mathbf{1}_{[x_{j-1}, x_j)}.$ If $f$ jumps $k$ times on the interval $[x_{j-1}, x_j)$ then $\sup_{x\in [x_{j-1}, x_j)}|f(x) - \widetilde h(x) | \leq k t.$ Since the total number of jumps is bounded by $MK,$ $\|f-\widetilde h\|_1 \leq M K t r_{\eps} = \eps/2$ implying $\|f-h\|_1 \leq \eps.$ There are at most $\binom{1/r_{\eps}}{\ell}(2S_{\eps}+1)^{\ell+1}$ functions in $\Theta^0$ with $\ell$ jumps. The cardinality of $\Theta^0$ is therefore bounded by
\begin{align*}
	\sum_{\ell =0}^{MK} \binom{1/r_{\eps}}{\ell} (2S_{\eps}+1)^{\ell+1}
	&\leq \sum_{\ell =0}^{MK} r_{\eps}^{-\ell} (2S_{\eps}+1)^{\ell+1}
	\leq 2 r_{\eps}^{-MK}(2S_{\eps}+1)^{MK+1}\\
	&\leq C'' \Big( \frac{K}{\eps}\Big)^{C'K}
\end{align*}
for suitable constants $C'$ and $C''.$
\end{proof}

\begin{proof}[Proof of Theorem \ref{thm.contr_CPP}]
For all three cases we apply Lemma \ref{lem.SB_for_CPP} and Lemma \ref{lem.CPP_entropy} to verify the conditions of Theorem \ref{thm.main_ub}. For \textit{(i)} we choose $\eps = (\log n/n)^{\beta/(\beta+1)}$ and $K = (n/\log n)^{1/(\beta+1)}$ in Lemma \ref{lem.CPP_entropy}. \textit{(ii)} can be proved in the same way with $\beta=1.$ For \textit{(iii)}, observe that by Lemma \ref{lem.CPP_entropy} there exists $\Theta_n$ such that $N_{[}\big( \eps_n, \Theta_n, \|\cdot \|_1 \big) \leq C''e^{Cn \eps_n}$ and $\Pi(\Theta_n^c) \leq e^{-cM n \eps_n}$ if
\begin{align*}
	K_n \log \Big( \frac{K_n}{\eps_n} \Big)  \leq n \eps_n \quad \text{and} \quad K_n \log K_n \geq cn \eps_n.
\end{align*}
If $\eps_n  \gtrsim n^{\rho -1}$ for some $\rho>0$ and $n$ is sufficiently large, then $\log(n\eps_n)/\log n$ remains positive and $K_n = n\eps_n/\log n$ satisfies both inequalities for some $c>0$.
\end{proof}

\subsection{Proof of Theorem \ref{thm.contr_subord}}

\begin{prop}\label{PropSmallBallSubord}
Consider the randomly initialized subordinator prior. If $\nu$ satisfies $\nu(x) \leq C x^{-3/2}$ for all $x,$ then, there exists a positive constant $c>0$ such that
\[ \inf_{f_0\in \mM(R)} P\Big(\|X-f_0\|_1\le 3\eps,\,X\le f_0\Big)\ge \eps^{c\eps^{-1}} \quad \text{ for all \ } \eps\in(0,1/2). \]
\end{prop}

\begin{proof}
We shall use the following small ball probability of an $\alpha$-stable subordinator around zero:
\[ \lim_{\eps\to 0} \eps^{\alpha/(1-\alpha)}\log(P(\|X\|_\infty\le\eps))\in (-\infty,0),
\]
which follows from Proposition 1 in \cite{Simon2004} noting that for non-decreasing functions starting in zero the 1-variation equals the supremum norm. This result shows that the $\alpha$-stable subordinators satisfy the small ball probability in $L^\infty$ with rate $e^{-c\eps^{-1}}$ if and only if $\alpha\le 1/2$.

Introducing $\nu_>(x)=(\nu(x) \wedge \nu(1)){\bf 1}(0\leq x\leq 1)+\nu(x){\bf 1}(x>1)$ and $\nu_< = \nu-\nu_>,$ we can decompose $X$ as $X_0+X^<+X^>$ with two independent L\'evy processes $X^<,X^>$ having L\'evy densities $\nu_<,\nu_>$, respectively. The small jump process $X^<$ is a subordinator whose L\'evy density is smaller than $\nu_{1/2}(x)= Cx^{-3/2}{\bf 1}(x>0)$, the L\'evy density of a stable subordinator $X^{(1/2)}$ of index $\alpha=1/2$. We can thus couple $X^<$ and $X^{(1/2)}$ such that $X^<_t\le X^{(1/2)}_t$ holds for all $t\ge 0$ a.s. By the above result, this gives
\[ \log\big(P(\|X^<\|_\infty\le\eps)\big)\gtrsim -\eps^{-1}.\]
Because of $\lambda :=\int \nu_> \leq \nu(1) + \int_1^\infty \nu < \infty,$ the process $X^>$ is a CPP with jump distribution $G=\nu_>/\lambda.$ If $f_0 \in \mM(R)$ and $\eps \leq R,$ then $f_0-\eps \in \mM(2R)$ and by Lemma \ref{lem.SB_for_CPP} (i),
\[  \inf_{f_0 \in \mM(R)}P\big(\| X_0+X^>-(f_0-\eps)\|_1\le 2\eps,\,X\le f_0-\eps\big)\Big) \ge e^{-2\lambda} (1\wedge \lambda)^{8R/\eps}\eps^{c \eps^{-1}}.\]
By independence, we conclude for $X= X_0+X^<+X^>$:
\begin{align*}
&\log\Big(P\big(\|X-f_0\|_1\le 3\eps,\,X\le f_0\big)\Big)\\
&\ge  \log\Big(P\big(\| X_0+X^>-(f_0-\eps)\|_1\le 2\eps,\, X_0+X^>\le f_0-\eps,\,X^<\le\eps\big)\Big)\\
&\gtrsim -\eps^{-1}\log(\eps^{-1})-\eps^{-1}.
\end{align*}
This gives the result.
\end{proof}

\begin{lem}
\label{lem.CPP_entropy2}
Consider the randomly initialized subordinator prior. Assume that there are constants $\gamma,L>0$ such that $\nu(x) \leq L x^{-3/2}$ for all $x$ and $\int_s^\infty \nu(x) +h(x) + h(-x) \, dx \leq L^{-1} e^{-L s^{\gamma}}$ for all $s\geq 1.$ Then for any $M,A >0$ there exist Borel sets $(\Theta_n)_n$ and constants $C',C'',$ such that for all sufficiently large $n,$
\begin{align*}
	\P(X \notin \Theta_n) \leq C'' e^{-M\sqrt{n\log n}} \quad \text{and}  \ \ \ N_{[}\Big( A(\sqrt{\log n /n}), \Theta_n, \|\cdot \|_1 \Big) \leq C'' e^{C'\sqrt{n\log n}}.
\end{align*}
\end{lem}

\begin{proof}
Let $\delta = 1/(2M\sqrt{n\log n}).$ We can decompose the subordinator in $X=X_0+X_{<}+X_{>},$ where $X_{<}$ and $X_{>}$ are subordinators with L\'evy densities $\nu_{<}(x)=\nu(x) \mathbf{1}(x\leq \delta)$ and $\nu_{>} = \nu-\nu_{<},$ respectively. Observe that by the L\' evy-Khintchine formula, extended to the moment-generating function,
\begin{align*}
	P\big(X_{<}(1) > 1 \big)
	&\leq \frac{E[e^{ \delta^{-1}X_{<}(1)}]} {e^{\delta^{-1}}}
	= \exp\Big( \int_0^{\delta} (e^{x/\delta}-1) \nu(x) dx  - \frac 1\delta\Big) \\
	&\leq  \exp\Big( \int_0^{\delta} (e-1)\frac x \delta \nu(x) dx -\frac 1\delta\Big)\le \exp\Big(\frac{2L(e-1)\delta^{1/2}-1}{\delta}\Big)
	\leq e^{-M\sqrt{n \log n}}
\end{align*}
for all sufficiently large $n.$ The process $X_{>}$ is a CPP with intensity $\lambda = \int_\delta^\infty \nu(x) \leq 2L\delta^{-1/2}$ and jump density $\nu_{>}(x)/\lambda.$ If $N\sim \Pois(\lambda)$ denotes the number of jumps of $X_{>}$ on $[0,1],$ we find by \eqref{eq.contr_CPP0}, $P( N \geq \max(2\lambda e, 1) m)\leq m^{- m}.$ Let $\Delta_0:=X_0$ and denote the jump heights of the CPP $X_{>}$ by $\Delta_i,$ $i=1,\ldots.$ Let $c_0:=\inf_{x\in [1,2]} \nu(x)$ and observe that $c_0>0$ because $\nu$ is continuous and positive. Arguing as for \eqref{eq.contr_CPP1}, with $t:= 1\vee (L^{-1}(m+1)\log m)^{1/\gamma},$
\begin{align*}
	\P\Big(\max_{i=0, \ldots, N} |\Delta_i| \geq t\Big)
	&\leq \P\big(|\Delta_0| \geq t\big)  + m  \max(2\lambda e, 1)  \frac{\int_t^\infty \nu}{\lambda} + m^{-m} \\
	&\leq \Big(2+\frac{m}{L} \max(2e,1/c_0) \Big) e^{-Lt^\gamma} + m^{-m}
	\leq \Big(\frac{1}{L} \max(2e,1/c_0) + 3\Big) m^{-m}.
\end{align*}
Put $m=4M\sqrt{n/\log n}$ and define $\Theta_n^>$ as the space of piecewise constant functions $f$ with $|f(0)| \leq t$, less than $m$ jumps, minimal jump size $\delta$ and maximal jump size bounded by $t$. For all sufficiently large $n,$ $m^{-m} \leq e^{-2M\sqrt{n/\log n}(\log n-\log\log n)}\leq e^{-M \sqrt{n\log n}}.$ From the computations above, $P(X_{>} \notin \Theta_n^>) \leq \text{const.} \times e^{-M \sqrt{n\log n}}.$ Let $\Theta_{\mon, \delta} = \{g : g \ \text{monotone}, g \leq 1 \text{ and all jumps are} \leq \delta\}$ and  $\Theta_n =\{f=g+h : g\in \Theta_{\mon, \delta}, h \in \Theta_n^{>}\}$ then also $\P(X\notin \Theta_n) \leq \text{const.} \times e^{-M \sqrt{n\log n}}$ due to the uniqueness of the decomposition $f=g+h$ in $\Theta_n.$

Notice that
\begin{align*}
	N_{[}\big( \eps, \Theta_n, \|\cdot \|_1 \big)
	\leq N_{[}\big( \eps/2, \Theta_{\mon,0}, \|\cdot \|_1 \big)N_{[}\big( \eps/2, \Theta_n^>, \|\cdot \|_1 \big).
\end{align*}
It is well known (\cite{AadWell}, 2.7.5 Theorem) that $N_{[}\big( \eps/2, \Theta_{\mon,0}, \|\cdot \|_1 \big) \leq e^{K/\eps}$ for some constant $K.$ A bound for the second factor follows from the proof of Lemma \ref{lem.CPP_entropy} with $K_n = m.$ This completes the proof.
\end{proof}

\begin{proof}[Proof of Theorem \ref{thm.contr_subord}] Using Lemma \ref{lem.SB_for_CPP} with $\eps= \sqrt{\log n/n}$ and \ref{lem.CPP_entropy2} yield the conditions of Theorem \ref{thm.main_ub} for contraction rate $\sqrt{\log n/n}.$
\end{proof}

\section{Proofs for Section \ref{sec.BvM}} \label{AppB}

\subsection{Proof of Theorem \ref{thm.BvM}}

Set $r_{j,n}:=n(t_j-t_{j-1}),$ $r_n = \min_j r_{j,n}$ and $c:= \inf_{-2R\leq x \leq 2R} g(x)>0$ ($g$ is continuous and positive).   Let
\begin{align*}
	A = \Big\{0\leq  \max_j r_{j,n} (\widehat a_j -a_j)_+ \leq 2\log K_n \Big\}.
\end{align*}
By Lemma \ref{lem.TV_for_cond_distribs} it remains to prove that uniformly over $f_0\in \PC^*(K_n, (t_0, \ldots, t_{K_n}), R),$
\begin{itemize}
\item[(i)] $E_{f_0}[ \| \Pi (\cdot \cap A | N)/\Pi(A |N) - Q^n(\cdot |A) \|_{\TV}] \rightarrow 0$
\item[(ii)] $E_{f_0} [ \Pi(A^c |N)+Q^n(A^c) ] \rightarrow 0.$
\end{itemize}
{\it Proof of $(i):$} Recall that $-R\leq a_j^0 \leq \widehat a_j.$ Then on the event $\{\max_j \widehat a_j \leq 3R/2\}\cap A$
\begin{align*}
	g(a_j) \leq g\big( \widehat a_j \big) + \|g\|_{\mC^\beta} \Big( \frac{2\log K_n}{r_n}\Big)^\beta
	\leq g\big( \widehat a_j \big) ( 1+  R_n) \text{ with }
	R_n = \frac{ \|g\|_{\mC^\beta}}c \Big( \frac{2\log K_n}{r_n}\Big)^\beta.
\end{align*}
Similarly, we find $ g(a_j) \geq g\big( \widehat a_j \big) (1-R_n).$ For an arbitrary Borel set $B$ and $f_{\ba}= \sum_{j=1}^{K_n} a_j \mathbf{1}_{[t_{j-1},t_j)}$ we obtain
\begin{align*}
	\frac{\Pi(B\cap A | N)}{\Pi(A|N)}
	&=\frac{\int_{B \cap A} e^{n\int f_a} \mathbf{1}(\forall i: f_a(X_i) \leq Y_i) \prod_j g(a_j) d\ba}{\int_{A} e^{n\int f_a} \mathbf{1}(\forall i: f_a(X_i) \leq Y_i) \prod_j g(a_j) d\ba} \\
	&\leq \frac{\sup_{(a_1, \ldots, a_{K_n}) \in A} \prod_j g(a_j)}{\inf_{(a_1, \ldots, a_{K_n}) \in A} \prod_j g(a_j)} Q^n(B|A)\\
	&\leq \Big(\frac{1+R_n}{1-R_n}\Big)^{K_n} Q^n(B|A).
\end{align*}
By assumption, there is some $N_0$ such that  $R_n \leq 1/2$ for all $n \geq N_0.$ This gives
\begin{align*}
	\Big(\frac{1+R_n}{1-R_n}\Big)^{K_n}
	\leq (1+4R_n )^{K_n} \leq \exp(4R_n K_n),
\end{align*}
which proves that
\begin{align}
	\sup_B \Big(\frac{\Pi(B\cap A | N)}{\Pi(A|N)} - Q^n(B|A)\Big) \leq \exp(4R_n K_n)-1.
	\label{eq.TV_one_direction}
\end{align}
Analogous arguments together with the Bernoulli inequality imply
\begin{align*}
	\frac{\Pi(B\cap A | N)}{\Pi(A|N)}
	\geq \Big(\frac{1-R_n}{1+R_n}\Big)^{K_n} Q^n(B|A)
	\geq (1-R_n)^{K_n}  Q^n(B|A) \geq Q^n(B|A) -R_nK_n.
\end{align*}
 Together with \eqref{eq.TV_one_direction} and the assumption $R_n K_n \rightarrow 0$ this gives
\begin{align}
	\| \Pi (\cdot \cap A | N)/\Pi(A |N) - Q^n(\cdot |A) \|_{\TV} \mathbf{1}\big( \max_j \widehat a_j \leq 3R/2 \big) \to 0.
	\label{eq.TV_restricted}
\end{align}
Notice that under $P_{f_0},$ $\widehat a_j -a_0 \sim \Exp(r_{j,n})$.  Thus with $\xi_j \sim \Exp(r_{j,n}),$
\begin{align}
	P_{f_0}\Big( \max_j \widehat{a}_j \geq 3R/2 \Big)
	\leq \sum_{j=1}^{K_n} \P( \xi_j \geq R/2) \leq K_n e^{-r_n R/2} \rightarrow 0.
	\label{eq.ld_hataj}
\end{align}
Together with \eqref{eq.TV_restricted}, this proves $(i).$

{\it Proof of $(ii):$} The density of $Q^n$ factorizes as $\prod_{j=1}^{K_n} r_{j,n} e^{r_{j,n} ( a_{j} -\widehat a_{j})} \mathbf{1}(a_j \leq \widehat a_j ).$ By a union bound we obtain
\begin{align*}
	Q^n\Big( \max_j r_{j,n} (\widehat a_j -a_j)_+ \geq 2\log K_n \Big)
	\leq
	\sum_{j=1}^{K_n} Q^n\big( (\widehat a_j -a_j)_+ \geq 2r_{j,n}^{-1}\log K_n \big) = \frac{1}{K_n} \rightarrow 0
\end{align*}
and thus $E_{f_0} [ Q^n(A^c) ] \rightarrow 0.$
Next, we  show that
\begin{align*}
	\max_j E_{f_0}\big[\Pi\big( (\widehat a_j -a_j)_+ \geq 2r_{j,n}^{-1}\log K_n \big| N \big)\big]
	\lesssim \frac 1{K_n^2},
\end{align*}
which together with a union bound completes the proof for $(ii).$

Since the likelihood factorizes as $e^{n \int f_\ba} \mathbf{1}\big(\forall i : f_\ba(X_i) \leq Y_i \big)= \prod_{j=1}^{K_n} e^{r_{j,n} a_j } \mathbf{1}\big( a_j \leq \widehat a_j\big),$ we find, using $\|g\|_\infty \leq \|g\|_{\mC^\beta},$
\begin{align*}
	\Pi\big( (\widehat a_j -a_j )_+ \geq 2r_{j,n}^{-1}\log K_n \big| N \big)
	&= \frac{\int_{-\infty}^{\widehat a_j - 2\log (K_n) /r_{j,n}} e^{r_{j,n} a_j } g(a_j) da_j}{\int_{-\infty}^{\widehat a_j} e^{r_{j,n} a_j} g(a_j) da_j} \\
	&\leq \frac{\|g\|_{\mC^\beta} e^{r_{j,n} \widehat a_j}}{K_n^2 r_{j,n} \int_{-\infty}^{\widehat a_j} e^{r_{j,n} a_j} g(a_j) da_j}.
\end{align*}
Recall that $|a_j^0|\leq R$ and $a_j^0 \leq \widehat a_j.$ As in $(i)$ we work on the event $\widehat a_j \leq 3R/2.$ Then in the denominator we can bound from below
\begin{align*}
	\int_{-\infty}^{\widehat a_j} e^{r_{j,n} a_j} g(a_j) da_j
	\geq c \int_{\widehat a_j-R/2}^{\widehat a_j} e^{r_{j,n} a_j} da_j
	\geq \frac{c}{r_{j,n}} e^{r_{j,n}\widehat a_j} \big(1 -e^{-r_{j,n} R/2}\big).
\end{align*}
Let $N_0'$ such that $r_n \geq 2/R$ for all $n\geq N_0'.$ Then, for $n \geq N_0',$
\begin{align*}
	\Pi\big( (\widehat a_j -a_j )_+ \geq 2r_{j,n}^{-1}\log K_n \big| N \big) \mathbf{1}\big(\widehat a_j \leq 3R/2\big)
	\leq \frac{\|g\|_{\mC^\beta}}{K_n^2 c(1-e^{-1})}.
\end{align*}
Together with \eqref{eq.ld_hataj} and $K_ne^{-r_nR/2} \lesssim 1/K_n^2,$ this yields
\begin{align*}
	\max_j E_{f_0}\big[\Pi\big( (\widehat a_j -a_j)_+ \geq 2r_{j,n}^{-1}\log K_n \big| N \big)\big]
	\lesssim \frac 1{K_n^2}.
\end{align*}
This shows $(ii)$ and completes the proof.

\subsection{Proof of Corollary \ref{cor.BvM_fctal}}

By Theorem \ref{thm.BvM}, the total variation distance between the marginal posterior  of $\vartheta$  under $P_{f_0}$ and the distribution of $\int \widehat f^{\MLE} - \sum_{j=1}^{K_n} (t_j-t_{j-1})\eta_j$ with independent $\eta_j \sim \Exp(n(t_j-t_{j-1}))$ converges to zero. For $\xi_j=n(t_j-t_{j-1})\eta_j \sim \Exp(1)$ we deduce from Lemma \ref{lem.CLT_in_TV} below
\begin{align}
	\TV\Big( \int \widehat f^{\MLE} - \frac {\sum_{j=1}^{K_n} \xi_j}n , \mathcal{N}\Big(\int \widehat f^{\MLE} -\frac {K_n} n, \frac{K_n}{n^2} \Big) \Big)
	\rightarrow 0.
	\label{eq.CLTinTV}
\end{align}
This completes the proof of the first assertion. It also implies that $I(\alpha)$ is an asymptotic $(1-\alpha)$-credible interval.

It remains to prove that $I(\alpha)$ is also an honest  confidence interval. By the explicit law of $\widehat f^{\MLE}$, we conclude that under $P_{f_0},$ $\int \widehat f^{\MLE}= \int f_0 + n^{-1}\sum_{j=1}^{K_n} \xi_j'$ holds with independent $\xi_j' \sim \Exp(1)$.  Thus, uniformly in $f_0$ we have
\begin{align*}
	P_{f_0}\Big( \int f_0 \in I(\alpha) \Big)
	= \P \Big( \Phi^{-1}\big( \alpha/2\big)  \leq K_n^{-1/2}\sum_{j=1}^{K_n} (\xi_j'-1) \leq \Phi^{-1}\big(1-\alpha/2\big) \Big)
	\rightarrow 1-\alpha,
\end{align*}
using the standard central limit theorem.

\subsection{Proof of Lemma \ref{lem.freq_CI}}

A brief inspection of the proof shows that Theorem 2.1 in \cite{reiss2014} also holds for functions which are $C^\beta(R)$ on each interval $[kh,k(h+1)).$ Define $I_k := [(k-1)/K_n, k/K_n),$ $Y_k^* := \min_{i:X_i \in I_k}Y_i,$
\begin{align*}
	\widehat \vt_k := \Big( Y_k^* +\frac{1}{K_n}\Big) - \frac{K_n}{n}
	\sum_{i\geq 1} \mathbf{1}\Big( X_i \in I_k , Y_i \leq  Y_k^*+ \frac{1}{K_n}\Big),
\end{align*}
and $\widehat \vt^{\operatorname{block}}= K_n^{-1}\sum_{k=1}^{K_n} \widehat \vt_k.$ To obtain the expectation and a bound on the variance of $\widehat \vt^{\operatorname{block}},$ we can apply Theorem 2.1 in \cite{reiss2014} with $w=1, \beta=R=1$ and $h=1/K_n$ since the true support boundary function is in $\Lip_{K_n}.$ This gives $E_{\vt_0}[\widehat\vt^{\operatorname{block}}] = \vt_0$ and $\Var(\widehat\vt^{\operatorname{block}}) = 2/(K_n n)+K_n/n^2.$ For
\[C(\alpha) = \big[\widehat\vt^{\operatorname{block}} - \alpha^{-1/2}(2/(K_n n)+K_n/n^2)^{1/2}, \widehat\vt^{\operatorname{block}} + \alpha^{-1/2}(2/(K_n n)+K_n/n^2)^{1/2}\big]\]
we obtain
\begin{align*}
	P_{f_0}\big(\vt_0 \notin C(\alpha) \big)
	\geq  P_{f_0}\big( \big| \widehat\vt^{\operatorname{block}} - \vt_0 \big| \le \alpha^{-1/2}\Var_{f_0}(\widehat \vt^{\operatorname{block}})^{1/2} \big)
	\leq \alpha
\end{align*}
by Chebyshev's inequality. The length of $C(\alpha)$ is $O( \sqrt{K_n}/n+1/\sqrt{K_nn})$. \qed

\subsection{Proof of Theorem \ref{thm.credible_not_conf}}
\begin{prop}
\label{prop.MLE_misspecified}
Consider data generated by  $f_0$ of the form \eqref{eq.def_piecewise_lin}. Then the MLE taken over the class of piecewise constant functions $$\widehat f^{\MLE} = \sum_{j=1}^{K_n} \widehat a_j^{\MLE} \mathbf{1}_{[(j-1)/K_n, j/K_n)}$$ with $\widehat a_j^{\MLE} = \min_{i: X_i \in [(j-1)/K_n, j/K_n)} Y_i$ can be written in distribution as
\begin{align*}
	\widehat a_j^{\MLE} = a_j^0 + \frac{j-1+ V_{jn}}{K_n},
\end{align*}
where $(V_{jn})_j$ is i.i.d. with distribution defined by
\begin{align*}
	P_{f_0}(V_{jn} \geq y)
	&= \exp \Big( - \frac{n}{2K_n^2}(y \wedge 1)^2 - \frac{n}{K_n^2} (y-1)_+ \Big),\quad y\ge 0.
\end{align*}
Moreover, for $K_n \geq \sqrt{n}$ we have $\Var_{f_0}(\widehat \vt^{\MLE} )^{1/2}\lesssim \sqrt{K_n}/n$ and
\begin{align*}
	E_{f_0}\Big[ \widehat \vt^{\MLE} - \frac {K_n}n \Big] \leq \int f_0 - \frac{n}{2^7K_n^3}.
\end{align*}
\end{prop}

\begin{proof} The first assertion follows from a simple PPP probability calculation. Let us derive bounds for the expectation and the second moment of $V_{jn}.$ Let $r>0.$ The identity $\int_0^1 ye^{-ry^2}dy = (1-e^{-r})/(2r)$ and integration by parts give
\begin{align*}
	\int_0^1 e^{-r y^2} dy + \frac 1{2r} e^{-r}= \frac{1}{2r} + \int_0^1 \int_0^z e^{-ry^2} dy dz.
\end{align*}
With $r= n/(2K_n^2),$ $E[V_{jn}]$ can therefore be rewritten as
\begin{align}
	E\big[V_{jn}\big] &= \int_0^\infty P\big(V_{jn} \geq y \big) dy= \int_0^1 e^{-\frac{n}{2K_n^2}y^2} \, dy  + \frac{K_n^2}n e^{-\frac{n}{2K_n^2}}
	= \int_0^1 \int_0^z e^{-\frac{n}{2K_n^2}y^2} \, dy \, dz + \frac{K_n^2}{n}
	\notag \\
	&\leq
	\frac{3}{8} + \frac 18 e^{- \frac{n}{8K_n^2}}+ \frac{K_n^2}{n}
	\leq \frac{7}{16} \vee \Big( \frac{1}{2} - \frac{n}{2^7K_n^2}\Big)  + \frac{K_n^2}{n},
	\label{eq.E_V_jn}
\end{align}
where for the first inequality, we decomposed the double integral into $\int_0^1 = \int_0^{1/2}+ \int_{1/2}^1$ and $\int_0^z = \int_0^{1/2}+ \int_{1/2}^z$ for $z\geq 1/2$ and for the second inequality used $e^{-x} \leq 1/2 \vee (1-x/2)$ for $x\geq 0.$ Moreover,
\begin{align}
	E\big[V_{jn}^2\big]
	&= \int_0^\infty P\big(V_{jn} \geq \sqrt{y} \big) dy
	= \int_0^1 e^{- \frac{n}{2K_n^2} y} dy + e^{\frac{n}{2K_n^2}}\int_1^\infty e^{- \frac{n}{K_n^2} \sqrt{y}} dy \notag \\
	&= 2\frac{K_n^2}{n} \big( 1 - e^{-\frac{n}{2K_n^2}}\big) +  2e^{\frac{n}{2K_n^2}}\int_1^\infty v e^{-\frac{n}{K_n^2} v} dv\leq 2\frac{K_n^2}{n} + 8 \frac{K_n^4}{n^2}.
	\label{eq.var_V_jn}
\end{align}
For $K_n \geq \sqrt{n},$ we have by \eqref{eq.E_V_jn}, $E[V_{jn}]\leq \tfrac 12 - n/(2^7K_n^2) + K_n^2/n$ and together with $\widehat a_j =a_j^0 +(j-1+V_{jn})/K_n,$
\begin{align*}
	E_{f_0}\Big[ \int \widehat f^{\MLE} - \frac {K_n}n \Big]
	&= 	\frac 1{K_n} \sum_{j=1}^{K_n} E_{f_0}[\widehat a_j] - \frac {K_n}n
	= \int f_0 - \frac 12 + \frac{K_n(K_n-1)}{2K_n^2} + \frac{E[V_{1n}]}{K_n} - \frac {K_n}n \\
	&\leq \int f_0 - \frac{n}{2^7K_n^3} \quad \quad  \text{for all} \ K_n \geq \sqrt{n}.
\end{align*}
\end{proof}

{\textit Proof of Theorem \ref{thm.credible_not_conf}:} We first prove the Bernstein-von Mises type result $E_{f_0} [\|\Pi( \cdot | N) - Q^n\|_{\TV}] \rightarrow 0$ with $Q^n$ as defined in Theorem \ref{thm.BvM}. For $(a_1, \ldots, a_{K_n}) \sim Q^n$ we have
\begin{align*}
	Q^n\big( \min_j a_j < -R \big) \leq K_n \int_{-\infty}^{-R} \frac{K_n}{n} e^{\frac{n}{K_n}(a_j-\widehat a_j)} da_j \leq  K_ne^{-\frac{n}{K_n}R} \leq K_nn^{-R} \rightarrow 0.
\end{align*}
Arguing as in the proof of Theorem \ref{thm.BvM}, it follows that
\begin{align}
	\textstyle \|Q^n - Q^n (\cdot | \min_j a_j \geq -R )\|_{\TV} \leq 2Q^n (\min_j a_j < -R) \rightarrow 0.
	\label{eq.Qn_tv}
\end{align}
On the event $\mathcal{A} := \{\max_j \widehat a_j \leq R\}$ we have equality
\begin{align*}
	\Pi((a_1,\ldots, a_K) \in B |N )
	 = \frac{\int_{B \cap [-R,R]^{K_n}} e^{\frac{n}{K_n} \sum_j a_j} \mathbf{1}(\forall j:a_j \leq \widehat a_j) d\ba}
	 {\int_{ [-R,R]^{K_n}} e^{\frac{n}{K_n} \sum_j a_j} \mathbf{1}(\forall j:a_j \leq \widehat a_j) d\ba}
	 =
	 Q^n \big( B \big | \min_j a_j \geq  -R \big).
\end{align*}
Thus,
\begin{align}
\textstyle	E_{f_0} [\|\Pi( \cdot | N) - Q^n (\cdot |\min_j a_j \geq  -R )\|_{\TV}] \leq P_{f_0}(\mathcal{A}^c).
	\label{eq.post_tv}
\end{align}
By Proposition \ref{prop.MLE_misspecified}, $\widehat a_j = (j-1+V_{jn})/K_n,$ in distribution and thus
\begin{align*}
	P_{f_0}(\mathcal{A}^c) \leq  K_n P_{f_0}\big(V_{jn} \geq K_n(R-1)\big) \leq K_n e^{-\frac{n}{2K_n} (R-1)} \leq K_n n^{-(R-1)/2} \rightarrow 0,
\end{align*}
where the last step follows because of $R>3.$ With \eqref{eq.Qn_tv} and \eqref{eq.post_tv}, we obtain $E_{f_0} [\|\Pi( \cdot | N) - Q^n\|_{\TV}] \rightarrow 0.$ Arguing as in the proof of Corollary \ref{cor.BvM_fctal}, we can then conclude that for the marginal posterior of $\vartheta = \int f,$
\begin{align*}
	E_{f_0} \Big[ \Big\| \Pi( \vartheta  \in \cdot \,| N ) - \mathcal{N}\Big( \int \widehat f^{\MLE} - \frac{K_n}{n}, \frac{K_n}{n^2}\Big)\Big\|_{\TV}\Big]\rightarrow 0.
\end{align*}
This proves $E_{f_0}[ \Pi ( \vartheta \in I(\alpha) \, | \, N )] \rightarrow 1-\alpha.$

We turn to proving \eqref{eq.credible_set_overshoot}. Recall that $f_0(x)=x$ and $\rho_n  = 2^{-8}(nK_n^{-3/2} \wedge n^2K_n^{-7/2}).$ With $\sigma_n^2 := \Var_{f_0}(V_{1n})$ and
$$A_n := \frac{K_n^{3/2}}{\sigma_n} \Big( \frac{1}{2K_n} + \frac{K_n}{n} - \frac{E[V_{1n}]}{K_n} - \frac{\sqrt{K_n}}{n} \rho_n \Big),$$
we obtain by Chebyshev inequality
\begin{align}
	P_{f_0} \Big( \frac 12 \leq  \int \widehat f^{\MLE} - \frac {K_n}n + \frac{\sqrt{K_n}}{n} \rho_n \Big)
	=
	P_{f_0} \Big( \frac{\sum_{j=1}^{K_n} V_{j,n} - E[V_{j,n}]}{\sqrt{K_n} \sigma_n}
	\geq  A_n \Big)
	\leq \frac 1{A_n^2},
	\label{eq.cheb_lb}
\end{align}
If $K_n \leq \sqrt{n/8},$ then, with \eqref{eq.E_V_jn} and \eqref{eq.var_V_jn},
\begin{align*}
	A_n
	\geq \frac{K_n^{3/2}}{\sigma_n} \Big( \frac{1}{16K_n}- \frac{\sqrt{K_n}}{n} \rho_n \Big) \rightarrow \infty.
\end{align*}
On the other hand, if $\sqrt{n/8} \leq K_n =o(n^{4/7}),$ with \eqref{eq.E_V_jn} and \eqref{eq.var_V_jn},
\begin{align*}
	A_n
	\geq \frac{K_n^{3/2}}{\sigma_n} \Big( \frac{n}{2^7 K_n^3} - \frac{\sqrt{K_n}}{n} \rho_n \Big) \rightarrow \infty.
\end{align*}
Together with \eqref{eq.cheb_lb} this proves \eqref{eq.credible_set_overshoot}. The last claim follows from the definition of $I(\alpha)$ in \eqref{eq.def_I(alpha)} and the fact that $\rho_n \rightarrow \infty$ for $K_n=o(n^{4/7})$. \qed

\section{Proofs for Section \ref{SecBvMCPP}} \label{AppC}

\subsection{Properties of the MLE}

We gather here the results on the MLE $\widehat f^{\MLE}$, obtained over monotone functions $f$ on $[0,T]$ that are constant on $[1,T]$, which are proved in Appendix \ref{AppMLE} below.
If not otherwise stated, we work with a generic $f_0\in \mM_S(K_n, R)$ and under $P_{f_0}.$ An event $A$ is said to have {\it probability converging uniformly to one} if $\inf_{f_0\in \mM_S(K_n, R)} P_{f_0}(A) \rightarrow 1$ and we then write
\[ P_{f_0}(A)\xrightarrow{u} 1.\]

Define the event
\begin{align}
	H:=&\Big\{\forall k: X_k' \geq t_k^0 - \frac 1{2\sqrt{n}} \text{\, and} \  Y_k^* \leq f_0(t_k^0) + \frac{ \log n}{2\sqrt{n}} \Big\}
	\cap \Big\{ \widehat f^{\MLE}(1)\leq (K_n+2)R  \Big\},
	\label{eq.H_def}
\end{align}
where all $X_k'$ for $k=1, \ldots, K_n $ and all $Y_k^*$ for $k=0,1, \ldots, K_n $ are considered. Lemma \ref{lem.diff_lb}  states
\begin{equation}\label{MLE1}
P_{f_0}(H)\xrightarrow{u} 1
\end{equation}
and on $H$ for all $k=1, \ldots, K_n,$
\begin{align}
X_k'-X_{k-1}'\geq \frac{1}{\sqrt n}, \quad \ & \label{MLE3}\\
Y_k^*-Y_{k-1}^* \geq \frac{\log n}{\sqrt n } \quad \text{and } Y_0^* &\geq \frac{\log n }{\sqrt{n}}, \label{MLE4}\\
(Y_k^*-Y_{k-1}^*)\big[(X_{k+1}'-X_k') \wedge (X_k'-X_{k-1}') \big] &\geq \frac{K_n \log (eK_n) \log^3 n}{4 n}.\label{MLE5}
\end{align}
Moreover, we frequently use $X_{K_n+1}'-X_{K_n}' \geq T-1 \geq 1/\sqrt{n}$ for all sufficiently large $n.$ Recall that $\widetilde f = \sum_{k=0}^{K_n}Y_k^* \mathbf{1}_{[X_k',X_{k+1}')}.$ Lemma \ref{lem.seqs_on_MLE} shows that asymptotically there are no observations in $[X_k',t_k^0]\times[f(t_k^0),Y_k^*]$ and thus
\begin{equation}\label{MLE7}
	P_{f_0}(\widetilde f \leq \widehat f^{\MLE})\xrightarrow{u} 1.
\end{equation}
 Lemma \ref{lem.ub_jumps_of_MLE} shows that the number of jumps in each interval $[t_k^0,t_{k+1}^0]$ is at most of order $\log n$:
there exists a constant $C(R)$  such that
\begin{equation}\label{MLE8}
P_{f_0}\Big(\forall k=0,\ldots,K_n: \#\{\widehat t_\ell^{\MLE}\in[t_k^0,t_{k+1}^0\wedge 1]\,|\,\ell=0,\ldots,M\}\le C(R)\log n\Big)\xrightarrow{u} 1.
\end{equation}
Finally, in Lemmas \ref{lem.ub_areas}, \ref{lem.lb_areas} we are able to uniformly bound the area of certain rectangles which will be used later for bounding integrals: with probability tending uniformly to one
\begin{align}
	&\max_{ k=0, \ldots, K_n}\max_{\ell: t_k^0\le \widehat t_\ell^{\MLE}\le t_{k+1}^0} \big( \widehat f^{\MLE}(\widehat t_\ell^{\MLE})-f_0(t_k^0) \big)\big(t_{k+1}^0 - \widehat t_\ell^{\MLE} \big) \leq \frac{\log K_n + C' \log \log n} n,\label{MLE10}\\
	&\min_{k=1, \ldots, K_n} (Y_k^*-Y_k') (X_k^*-X_k') \geq \frac{\log^2 n}{n}.\label{MLE11}
\end{align}

\subsection{Posterior model selection}

In this section, we show that under a minimal signal strength condition the posterior concentrates asymptotically on the true dimension $K_n.$ Given a  set $B$ of functions on $[0,T]$, define
\[B_K = B \cap \mM(K,\infty)\]
for the restriction to monotone, piecewise constant functions with $K$ jumps. A generic function $f\in \mM(K,\infty)$ will be parametrized as $f(t)= \sum_{\ell=0}^K a_\ell \mathbf{1}(t \geq t_\ell)$ assuming that the jump times are ordered such that $0=: t_0 < t_1 \leq t_2 \leq \hdots \leq t_K.$ Recall the definition of the prior in \eqref{eq.g_def2}. With the Bayes formula \eqref{eq.Bayes}, the posterior is given by
\begin{equation}\label{EqPosterior}
 \Pi(B\,|\,N)=\frac{\sum_{K=0}^\infty U_K(B_K)}{\sum_{K=0}^\infty U_K(\mM(K, \infty))}
 \end{equation}
with
\begin{align*}
	U_K(B_K) :=
	&\lambda^K \int_{B_K} e^{n \sum_{\ell=0}^K a_\ell(T-t_\ell) }\mathbf{1}(f \leq \widehat  f^{\MLE}) g_K(\ba)
	\mathbf{1}(0\leq t_1 \leq \hdots \leq t_K\leq 1) \, d \bt d \ba,
\end{align*}
$d\bt = dt_1 \ldots dt_K$ and $d\ba = da_1 \ldots da_K.$ The first step is the following lower bound for the denominator in the Bayes formula.

\begin{lem}
\label{lem.denom}
With $\widetilde f$ from \eqref{eq.widetildef_def} we have
\begin{align*}
	P_{f_0}\Big(\int e^{n \int_0^T f} \mathbf{1} (f\leq \widehat f^{\MLE}) d\Pi(f)
	\geq
	\frac{\lambda^{K_n} (1-n^{-1/2})^{2K_n+1}}{n^{2K_n+1} \prod_{k=0}^{K_n} (X_{k+1}'-X_k')} e^{n \int_0^{T -\frac 1n}\widetilde f}\, \Big)\xrightarrow{u} 1.
\end{align*}
\end{lem}

\begin{proof}
By \eqref{MLE1}, \eqref{MLE7} $P(H\cap\{\widetilde f \leq \widehat f^{\MLE}\})\xrightarrow{u} 1$ holds. As a lower bound, we only integrate over functions $f\leq \widetilde f$ which jump exactly once in the interval $[X_k',X_{k+1}')$ and all jumps are at least of size $\log n/(2\sqrt{n}).$ This gives on $\{\widetilde f\le \widehat f^{\MLE}\}$
\begin{align}
	&\int e^{n \int f} \mathbf{1} (f\leq \widehat f^{\MLE}) d\Pi(f) \notag \\
	&\geq
	\lambda^{K_n} \int e^{n\sum_{k=0}^{K_n} a_k(T-t_k)}
	\mathbf{1}\Big( \forall k \; :\; X_k' \leq t_k < X_{k+1}', a_k \geq \frac{\log n}{2\sqrt n}, \sum_{\ell=0}^k a_\ell \leq Y_k^* \Big)  g_{K_n}(\ba) \,
	d \bt \, d \ba
	\label{eq.lb_den1}
\end{align}
(the $\forall k $ in the previous inequality is a slight abuse of notation, since $ X_k' \leq t_k < X_{k+1}'$ is meant to hold for $k=1, \ldots, K_n$ and  the inequalities for $a_k$ for $k=0, \ldots, K_n$). Now
\begin{align*}
	&\int_{X_k'}^{X_{k+1}'} e^{na_k(T-t_k)}
	dt_k
	= \frac{1}{na_k} e^{na_k(T-X_k')}\big(1 -e^{-na_k(X_{k+1}'-X_k')}\big)
\end{align*}
On the event $H$ we have by  \eqref{MLE3}, $X_k'-X_{k-1}'\geq 1/\sqrt n.$ Hence on $H\cap \{a_k \geq \log n/(2\sqrt n)\}$, $na_k(X_{k+1}'-X_k') \geq \tfrac 12 \log n$ and
\begin{align*}
	&\int_{X_k'}^{X_{k+1}'} e^{na_k(T-t_k)}
	dt_k
	\geq \frac{1}{na_k} e^{na_k(T-X_k')}\big(1 -n^{-1/2}\big).
\end{align*}
Inserting this into \eqref{eq.lb_den1} gives the lower bound on $H\cap\{\widetilde f\le \widehat f^{\MLE}\}$
\begin{align}
	&\int e^{n \int f} \mathbf{1} (f\leq \widehat f^{\MLE}) d\Pi(f) \notag \\
	&\geq
	\frac{\lambda^{K_n}(1-n^{-1/2})^{K_n}}{n^{K_n}} \int e^{n\sum_{k=0}^{K_n} a_k(T-X_k')}
	\mathbf{1}\Big( \forall k \; :\; a_k \geq \frac{\log n}{2\sqrt n},  \sum_{\ell=0}^k a_\ell \leq Y_k^* \Big)  e^{-\sum_{k=0}^{K_n} a_k} \,
	d \ba.
	\label{eq.lb_den2}
\end{align}
Let $v_k\geq n^{-1/2}.$ With $Y_k^*-Y_{k-1}^* \geq  n^{-1/2}\log n$ for $k \geq 1$ and $Y_0^*\geq  n^{-1/2}\log n$ on $H$ by \eqref{MLE4}, we have on $H\cap\{\sum_{\ell=0}^{k-1} a_\ell \leq Y_{k-1}^*\}$
\begin{align*}
	\int e^{n a_k v_k} \mathbf{1}\Big( \frac{\log n}{2\sqrt n} \leq a_k \leq Y_k^* - \sum_{\ell=0}^{k-1} a_\ell  \Big) da_k
	&\geq
	\frac{ e^{n(Y_k^*-\sum_{\ell=0}^{k-1} a_\ell ) v_k}}{nv_k}
	\big( 1 -e^{\sqrt{n} v_k \log (n)/2-n(Y_k^*-Y_{k-1}^*) v_k}\big) \\
	&\geq \frac{1-n^{-1/2}}{nv_k} e^{n(Y_k^*-\sum_{\ell=0}^{k-1} a_\ell ) v_k}.
\end{align*}
With this inequality, we can now further lower bound the right hand side of \eqref{eq.lb_den2} by integrating successively over $a_{K_n}, a_{K_n-1}, \ldots , a_0.$ We need to choose $v_{K_n}= T-n^{-1}-X_{K_n}'$ and $v_k=X_{k+1}'-X_k'$ for $k<K_n.$ On $H,$ $v_k \geq n^{-1/2}$ for all sufficiently large $n.$ This shows that
\begin{align*}
	\int e^{n \int f} \mathbf{1} (f\leq \widehat f^{\MLE}) d\Pi(f)
	&\geq
	\frac{\lambda^{K_n}(1-n^{-1/2})^{2K_n+1}}{n^{2K_n+1} \prod_{k=0}^{K_n} v_k} e^{n \sum_{k=0}^{K_n} Y_k^*v_k}.
\end{align*}
Since $X_{K_n}'\leq 1,$ we have $\int_0^{T-\frac 1n} \widetilde f = - Y_{K_n}^*/n +\sum_{k=0}^{K_n} Y_k^* (X_{k+1}'-X_k') = \sum_{k=0}^{K_n} Y_k^* v_k$ and the assertion follows.
\end{proof}

Consider specifically $B:= \bigcup_{K\ge 1}B_K$ with
\begin{align}
	B_K=\Big\{f= \sum_{\ell=0}^K a_\ell \mathbf{1}(\cdot \geq t_\ell)\in \mM(K,\infty)\,\Big|\, \forall \ \ell=1,\ldots,K: f(t_\ell)>\widehat f^{\MLE}(t_{\ell-1}) \Big\}.\label{EqAK}
\end{align}
The set $B_K$ is empty if $K$ exceeds the number of jumps of the MLE. $B_0$ coincides with all constant functions (recall $t_0=0$).

\begin{lem}
\label{lem.AK_bd}
Consider the CPP prior with jump distributions \eqref{eq.g_def} and the event $B$ defined via \eqref{EqAK}. Then,
\begin{align*}
	\Pi\big(B^c \big | N \big) \le  \frac{\lambda \widehat f^{\MLE}(1) }{n}.
\end{align*}
\end{lem}

\begin{proof}
For $K\ge 1$ we show
\begin{align}
	U_K((B^c)_K) \leq  \frac{\lambda \widehat f^{\MLE}(1)}{n } U_{K-1}( \mM(K-1, \infty)),
	\label{eq.to_show_AK_bd}
\end{align}
which by \eqref{EqPosterior} immediately implies the assertion. For $f= \sum_{\ell=0}^K a_\ell \mathbf{1}(\cdot \geq t_\ell) \in  (B^c)_K$ there is a $q\in\{1,\ldots,K\}$ such that   with $a_\ell^{(-q)}:=a_\ell+a_q\mathbf{1}(\ell=q-1)$
\[f^{(-q)} := a_0^{(-q)}+\sum_{\ell \neq q} a_\ell^{(-q)} \mathbf{1}(\cdot \geq t_\ell) \leq \widehat{f}^{\MLE}.\]
 Using $\exp(-\sum_{k=0}^K a_k)\prod_{k=1}^K a_k \leq \exp(-\sum_{k\neq q} a_k^{(-q)})\prod_{k=1}^K a_k^{(-q)}$ and the prescription \eqref{eq.g_def} of $g_k$, we obtain
\begin{align}
	U_K((B^c)_K) \leq
	&\lambda^K\sum_{q=1}^K
	\int e^{n \sum_{\ell \neq  q} a_\ell^{(-q)} (T-t_\ell) - n a_q^{(-q)}(t_q-t_{q-1})}
	\mathbf{1}(f^{(-q)} \leq \widehat  f^{\MLE}) \notag \\
	&\cdot   e^{-\sum_{k\neq q} a_k^{(-q)}}\prod_{k=1}^K a_k^{(-q)} \mathbf{1}(0\leq t_1 \leq \hdots \leq t_K\leq 1) \, d\bt d\ba^{(-q)}.
	\label{eq.AK_bd2}
\end{align}
We have $a_q^{(-q)} \leq a_{q-1}^{(-q)}$ and
\begin{align*}
	\int_0^{a_{q-1}^{(-q)}} \int_{t_{q-1}}^{ t_{q+1}} e^{-na_q^{(-q)}(t_q-t_{q-1})} a_q^{(-q)} \, dt_q da_q^{(-q)} \leq \int_0^{a_{q-1}^{(-q)}} \frac 1n da_q^{(-q)} = \frac {a_{q-1}^{(-q)}}n.
\end{align*}
Thus, we can integrate over $t_q$ and $a_q^{(-q)}$ on the right hand side of \eqref{eq.AK_bd2}. For any fixed $q,$ we then rename $a_\ell^{(-q)}$ as $a_\ell$ if $\ell <q$ and $a_{\ell-1}$ if $\ell >q.$ We also rename $t_\ell$ in $t_{\ell-1}$ if $\ell >q$ and find
\begin{align*}
	U_K((B^c)_K) \leq
	&\frac{\lambda^K}n \sum_{q=1}^K
	\int a_{q-1} e^{n \sum_{\ell=0}^{K-1} a_\ell (T-t_\ell) }
	\mathbf{1}\Big(\sum_{\ell=0}^{K-1} a_\ell \mathbf{1}(\cdot \geq t_\ell) \leq \widehat  f^{\MLE} \Big) \notag \\
	&\cdot  g_{K-1}((a_0, \ldots, a_{K-1})) \mathbf{1}(0\leq t_1 \leq \hdots \leq t_{K-1}\leq 1) \, d\bt da_0 \ldots d a_{K-1}.
\end{align*}
We can bound the sum $\sum_{q=1}^K a_q$ by $\widehat  f^{\MLE}(1)$ and the remaining integral is then over step functions with $K-1$ jumps. Using the definition of $U_{K-1}( \mM(K-1, \infty))$ yields \eqref{eq.to_show_AK_bd}.
\end{proof}

We now derive bounds for the posterior mass of the events $B_K.$ Since $B_K$ is empty if $K$ exceeds the number $M$ of jumps of the MLE, it is sufficient to consider $K\leq M.$ Let $\bs=(s_0,s_1,\ldots,s_{K+1})$ with $0=: s_0 < s_1 < \ldots < s_K < s_{K+1}:=T$ and $s_k \in \{\widehat t_\ell^{\MLE}, \ell=1,\ldots, M\}$ for $k=1,\ldots,K.$ In particular, this implies that $s_K \leq 1$ as all the jumps of the MLE occur before time one. Given $\bs$ consider the function
\begin{align*}
	\fMLEKs = \sum_{k=0}^K \widehat f^{\MLE}(s_k) \mathbf{1}_{[s_k,s_{k+1})},
\end{align*}
which satisfies $\fMLEKs\le\widehat f^{\MLE}$ and whose $K$ jump points form a subset of the MLE jump points.

\begin{lem}
The event $B_K \cap \{ f\leq \fMLEKs\}$ consist of functions $f$ whose $K$ jump points $(t_i,f(t_i))$ lie in the rectangles $[s_i,s_{i+1})\times (\fMLE(s_{i-1}),\fMLE(s_i)]$ for $i=1,\ldots K$.
\end{lem}

\begin{proof}
If $f= \sum_{\ell=0}^K a_\ell \mathbf{1}(\cdot \geq t_\ell) \in B_K  \cap \{ f\leq \fMLEKs\} \subseteq B_K  \cap \{ f\leq \fMLE\},$ then $f(t_\ell)> \fMLE(t_{\ell-1}) \geq f(t_{\ell-1})$ and therefore $f$ has $K$ jumps. Suppose that $f\in B_K$ jumps twice in $[s_i,s_{i+1})$, then with the second jump it must jump strictly above $\widehat f^{\MLE}(s_i)$ violating the constraint $\{ f\leq \fMLEKs\}$. Consequently, all $f\in B_K \cap \{ f\leq \fMLEKs\}$ jump at most once in every $[s_i,s_{i+1}).$ A similar argument also shows that $f$ does not jump on $[0,s_1).$  Therefore, $f$ must jump in each of the $K$ intervals $[s_i,s_{i+1}),$ $i=1, \ldots, K$ exactly once and thus $t_i \in [s_i,s_{i+1}).$ Because of $\fMLE(s_{i-1}) = \fMLE(t_{i-1}) < f(t_i) \leq \fMLE(s_i)$ the result follows.
\end{proof}

Define
\begin{align}
	\ell(\bs) :=\big| \{X_1', \ldots, X_{K_n}'\} \cap \{s_1, \ldots, s_K\}\big|.
	\label{eq.Ll_def}
\end{align}

\begin{lem}
\label{lem.ub_num}
If $R \geq 1,$ then
\begin{align*}
	P_{f_0}\bigg(U_K\big(B_K \cap \{ f\leq \fMLEKs\}\big)
	\leq
	 \frac{\lambda^K ((K_n+2)R)^{K_n - \ell(\bs)}\exp(n \int_0^{T-\frac 1n} \fMLEKs)}{n^K (1-n^{-1/2}) \prod_{k=0}^{K} (4R)^{-1} \vee (n(s_{k+1}-s_k))} \ \ \forall K, \forall \bs \bigg) \xrightarrow{u} 1 .
\end{align*}
\end{lem}

\begin{proof}
For any $v_k > 0,$
\begin{align*}
	\int e^{na_k v_k} \mathbf{1}\Big( a_k \leq \fMLE(s_k) - \sum_{\ell=0}^{k-1} a_\ell \Big)  da_k
	\leq \frac{1}{nv_k} e^{nv_k(\fMLE(s_k) - \sum_{\ell=0}^{k-1} a_\ell)}.
\end{align*}
Integrating successively over $a_K, a_{K-1}, \ldots, a_0$  with $v_K:= T -n^{-1} - s_K$ and
$v_k :=s_{k+1}-s_k$ for $k<K,$ we find
\begin{align*}
	V := \int e^{n \sum_{k=0}^K a_k(T -n^{-1} - s_k)}
	\mathbf{1}\Big(\forall k \, : \, \sum_{\ell=0}^k a_\ell \leq \fMLE(s_k) \Big) \, d\ba
	\leq \frac {e^{n \sum_{k=0}^K v_k \fMLE(s_k)}}{n^{K+1}\prod_{k=0}^K v_k}.
\end{align*}
Since $\fMLEKs = \sum_{k=0}^K \widehat f^{\MLE}(s_k) \mathbf{1}_{[s_k,s_{k+1})}$ with $s_K \leq 1,$ $\sum_{k=0}^K v_k \fMLE(s_k)  = -\fMLEKs(1)/n + \int_0^T \fMLEKs = \int_0^{T-\frac 1n} \fMLEKs.$ Together with $v_K \geq (T-s_K)(1-n^{-1/2}),$
\begin{align}
	V  \leq \frac {e^{n \int_0^{T-\frac 1n} \fMLEKs}}{n^{K+1}(1-n^{-1/2})\prod_{k=0}^K (s_{k+1}-s_k)}.
	\label{eq.lem_ub_num0}	
\end{align}
Since $f_0 = \sum_{k=0}^{K_n} a_k^0 \mathbf{1}( \cdot \geq t_k^0) \in \mM_S(K_n,R),$ the jumps heights $a_k^0$ are all bounded by $R.$ On the event $H$ defined in \eqref{eq.H_def}, $f_0(t_k^0) \leq Y_k^* \leq f_0(t_{k+1}^0)$ for all $k=0,\ldots, K_n.$ Hence, if $\fMLE(s_{k+1})-\fMLE(s_k) > 2R$ then, there exists a $q^*,$ such that $\fMLE(s_k) < Y_{q^*}^* <\fMLE(s_{k+1}).$ The MLE jumps on $X_{q^*}'$ to $Y_{q^*}^*$  and this means that $s_k < X_{q^*}' < s_{k+1}.$ Using the definition of $\ell(\bs)$ in \eqref{eq.Ll_def} and $R\geq 1,$
\begin{align}
	&\prod_{k=1}^{K} 1 \wedge \big(n(s_{k+1}-s_k) (\fMLE(s_{k+1})-\fMLE(s_{k-1}))\big) \notag \\
	&\leq \big(1 \vee \fMLE(1) \big)^{K_n - \ell(\bs)}
	\prod_{k=1}^{K} 1 \wedge (4R n(s_{k+1}-s_k)).
	\label{eq.lem_ub_num1}
\end{align}
By the definition of $B_K$ in \eqref{EqAK}, $f = \sum_{\ell=0}^K a_\ell \mathbf{1} (\cdot \geq t_\ell) \in B_K$ with $f\le \fMLEKs$ implies for the jump times $t_k\in[s_k,s_{k+1})$ and $\fMLE(s_{k-1}) <\sum_{\ell=0}^k a_\ell \leq \fMLE(s_k).$ The latter implies in particular that $a_k \leq \fMLE(s_k)-\sum_{\ell=0}^{k-1} a_k < \fMLE(s_k)-\fMLE(s_{k-1}).$ Integrating out $t_1, \ldots, t_K$ and using the inequalities \eqref{eq.lem_ub_num1} and \eqref{eq.lem_ub_num0},
\begin{align*}
	&U_K(B_K \cap \{ f\leq \fMLEKs\}) \notag \\
	& \le
	\lambda^K
	\int e^{n \sum_{k=0}^K a_k(T- t_k)}
	\mathbf{1}\Big(\forall k \; : \; s_k \leq t_k < s_{k+1},  \fMLE(s_{k-1}) <\sum_{\ell=0}^k a_\ell \leq \fMLE(s_k) \Big) g_K(\ba) \, d\bt d\ba \notag  \\
	&\leq
	V\frac{\lambda^K}{n^K} \prod_{k=1}^{K} 1 \wedge [n(s_{k+1}-s_k) (\fMLE(s_{k+1})-\fMLE(s_{k-1}))] \notag \\
	&\leq
	 V\big( 1 \vee \fMLE(1)\big)^{K_n - \ell(\bs)} \frac{\lambda^K}{n^K} \prod_{k=1}^{K} 1 \wedge \big[4Rn(s_{k+1}-s_k)\big] \\
	 &\leq \frac{\lambda^K (1 \vee \widehat f^{\MLE}(1))^{K_n - \ell(\bs)}}{n^K (1-n^{-1/2}) \prod_{k=0}^{K} (4R)^{-1} \vee [n(s_{k+1}-s_k)]} e^{n \int_0^{T-\frac 1n} \fMLEKs},
\end{align*}
using for the last step that for $x,A >0,$ $(1 \wedge (Ax))/x = 1/(A^{-1} \vee x)$ and $s_1 -s_0 =s_1 \geq X_1' \geq 1/\sqrt{n}.$ Due to \eqref{eq.H_def} and $R \geq 1$ we have $1 \vee \widehat f^{\MLE}(1)\leq (K_n+2)R$ and this completes the proof.
\end{proof}

\begin{lem}
\label{lem.int_comps}
For any $\tau \in [(T+1)/2,T]$ and $\ell(\bs)$ as defined in \eqref{eq.Ll_def}, the event
\begin{align*}
	\mathcal{B}=\bigcap_{K, \bs }\Big\{ n\int_0^\tau \big(\fMLEKs- \widetilde f \big)  \leq (K-\ell(\bs) ) \big(\log K_n + C'(R)\log\log n \big)
	-(K_n -\ell(\bs))\log^2 n \Big\}
\end{align*}
with the intersection taken over all model dimensions $K$ and all subsets of jump locations $\bs$ has probability converging to one in the sense that
\begin{align*}
	P_{f_0}(\mathcal{B}) \xrightarrow{u} 1.
\end{align*}
\end{lem}

\begin{proof}
To compare the integrals, we can compare the areas under the curves, see Figure \ref{fig.the_argument}. For any $s_q \notin \{X_k', k=1, \ldots, K_n\},$ there exists an index $q^*$ such that $s_q\in (X_{q^*}',X_{q^*+1}').$ On the interval $[s_q, X_{q^*+1}']$ the function $\fMLEKs$ takes the value $\widehat f^{\MLE}(s_q)$ and $\widetilde f$ has the smaller value $Y_{q^*}^*$ implying that the areas under the curves differ by the  rectangle $[s_q, X_{q^*+1}' \wedge \tau] \times [Y_{q^*}^*, \widehat f^{\MLE}(s_1)].$ \eqref{MLE10} shows that there is a constant $C'(R),$ such that with probability tending uniformly to one, each of these areas have Lebesgue measure bounded by $(\log K_n + C' \log\log n)/n.$ On the contrary, if $X_k' \notin \{s_1,\ldots, s_K\},$ the area under $\widetilde f$ contains the rectangle $[X_k',X_k^* \wedge \tau] \times [Y_k',Y_k^*]$ that is not contained in the area under $\fMLEKs.$ \eqref{MLE11} shows that with probability tending uniformly to one, the Lebesgue measures of all of these rectangles is lower bounded by $\log^2 n/n.$ Multiplying all areas with $n$ shows that
\begin{align*}
	\bigcap_{K, \bs }\Big\{ n\int_0^\tau \big(\fMLEKs- \widetilde f \big)
	\leq \sum_{q: s_q \notin \{X_1',\ldots, X_{K_n}'\}}	\log K_n + C'(R) \log\log n - \sum_{k : X_k' \notin \{s_1,\ldots, s_K\}} \log^2 n\Big \}
\end{align*}
has probability tending uniformly to one.  The first and second sum are over $K-\ell(\bs)$ and $K_n-\ell(\bs)$ many terms, respectively, proving the assertion.
\end{proof}

\begin{lem}
\label{lem.A_K^c_bd}
Let $R \geq 1$ and $K_n \leq n^{1/2-\delta}$ for some $\delta>0.$ For any $\beta >0$
\begin{align*}
	P_{f_0}\Big(\forall\,K,\bs :\,\Pi\big(B_K \cap \{ f\leq \fMLEKs\} \, \big | \, N \big)
	\leq \lambda^{K-K_n} n^{-\frac 12 (1+\delta) (K-K_n)_+ - \beta (K_n -\ell(\bs))}\Big)\xrightarrow{u} 1.
\end{align*}
\end{lem}

\begin{proof}
By Lemma \ref{lem.ub_num} and Lemma \ref{lem.denom}, with probability tending uniformly to one
\begin{align}
	&\Pi\big( B_K \cap \{ f\leq \fMLEKs\} \, \big | \, N \big) \notag \\
	&\leq
	\frac{((K_n+2)R)^{K_n-\ell(\bs)}\lambda^{K-K_n} n^{2(K_n-K)} \prod_{k=0}^{K_n} (X_{k+1}'-X_k')  }{(1-n^{-1/2})^{2K_n+2} n^{-K-1}\prod_{k=0}^{K} (4R)^{-1}\vee [n(s_{k+1}-s_k)]} e^{n \int_0^{T-\frac 1n} (\fMLEKs-\widetilde f)}.
	\label{eq.1}
\end{align}

\begin{figure}
\begin{center}
	\includegraphics[scale=0.6]{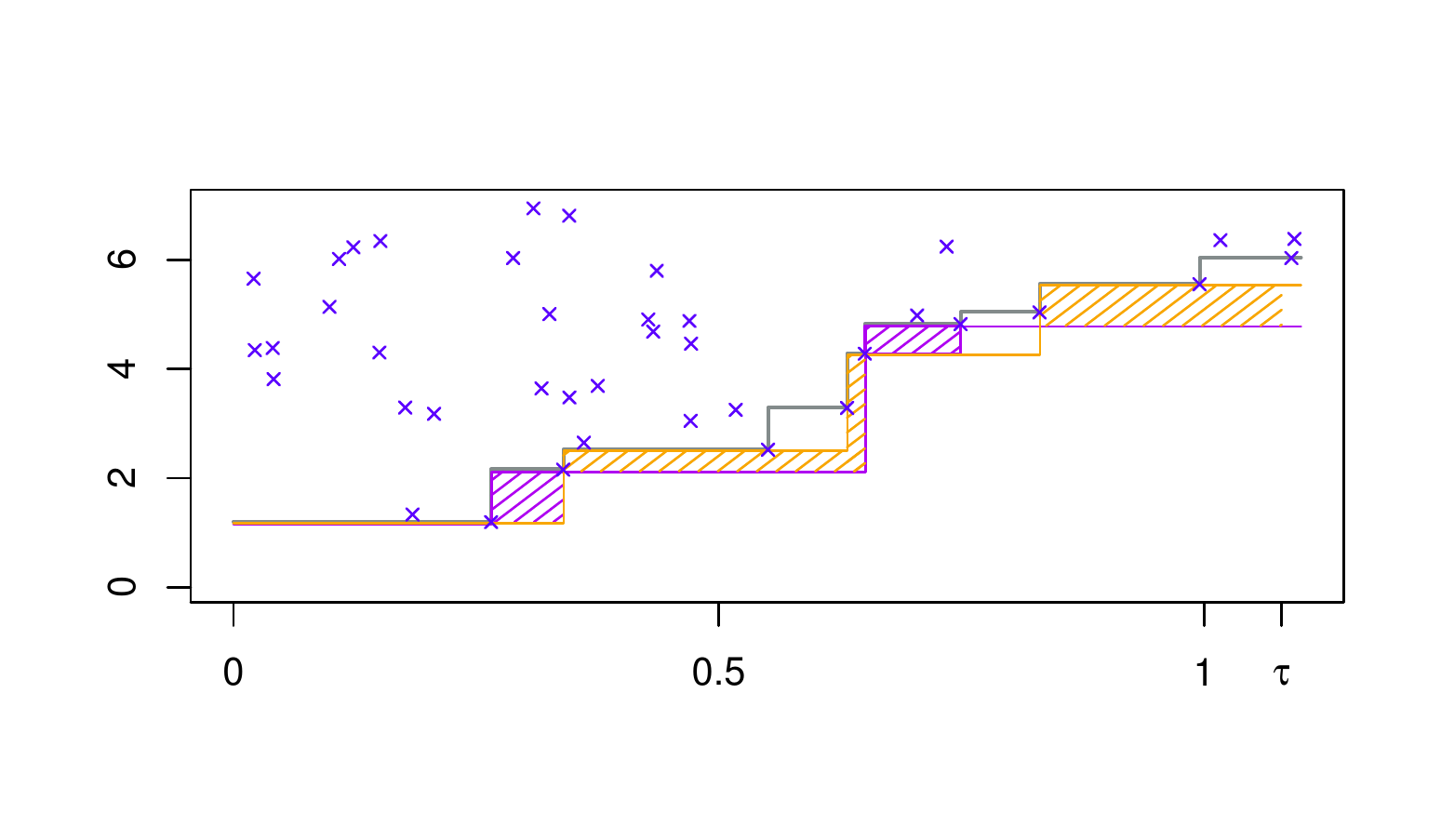}
	\vspace{-1cm}
	\caption{\label{fig.the_argument}  $\widetilde f$ (purple), $\fMLEKs$ (orange), MLE (gray) and the areas occurring in the proof of Lemma \ref{lem.int_comps}.}
\end{center}
\end{figure}

In a first step, we prove
\begin{align}
	P_{f_0}\Big(\forall K,\bs:\,\frac{\prod_{k=0}^{K_n} (X_{k+1}'-X_k') }{\prod_{k=0}^{K} (4R)^{-1}\vee [n(s_{k+1}-s_k)]} \leq \big( 8R n \big)^{K-\ell(\bs)}\Big)\xrightarrow{u}1.
	\label{eq.ub_prod_ratio_to_show}
\end{align}
Recall \eqref{eq.Ll_def}. Denote by $k(m),$ $m=1,\ldots, \ell(s)$ the  ordered entries of $L(\bs) := \{k: X_k' \in \{s_1, \ldots, s_K\}\}$ such that $0=: k(0) < k(1) < \ldots < k(\ell(\bs)) < k(\ell(\bs)+1) := K_n+1.$  We frequently use that $\ell(\bs) \leq K \wedge K_n.$ Obviously, $\prod_{k=0}^{K_n} (X_{k+1}'-X_k') \leq \prod_{m=0}^{\ell(\bs)} (X_{k(m+1)}'-X_{k(m)}').$ Denote by $r_m := \# \{ s_k \in (X_{k(m)}', X_{k(m+1)}')\}$ the number of jump locations $s_k$ that strictly lie between $X_{k(m)}'$ and $X_{k(m+1)}'.$ By construction, for each $m,$ there is an element in $\{s_j, j=1, \ldots, K\}$ with $s_j =X_{k(m)}'.$ If we split an interval in $q$ pieces, the longest piece must be larger than $1/(q+1)$ times the original interval length and
\begin{align}
	\frac{n^{K+1}\prod_{k=0}^{K_n} (X_{k+1}'-X_k') }{\prod_{k=0}^{K}  (4R)^{-1}\vee [n(s_{k+1}-s_k)]}
	&\leq n^{K+1}\prod_{m=0}^{\ell(\bs)} \frac{(X_{k(m+1)}'-X_{k(m)}')}{\prod_{k : s_k \in [X_{k(m)}', X_{k(m+1)}')}   (4R)^{-1}\vee [n(s_{k+1}-s_k)]} \notag \\
	&\leq \big(4R n \big)^{K-\ell(\bs)}\prod_{m=0}^{\ell(\bs)} (r_m+1).
	\label{eq.ub_prod_ratio}
\end{align}
Since $\sum_{m=0}^{ \ell(\bs)} r_m = K-\ell(\bs)$ and $r_m=0,1,\ldots$, we have  $\prod_{m=0}^{\ell(\bs)} (r_m+1) \leq 2^{K-\ell(\bs)}.$ Together with \eqref{eq.ub_prod_ratio} this yields \eqref{eq.ub_prod_ratio_to_show}.

As $K_n=o(n^{1/2})$, for sufficiently large $n$ we have $(1-n^{-1/2})^{2K_n+2}\ge 1/2$. With Lemma \ref{lem.int_comps}, \eqref{eq.1} can thus be simplified  to
\begin{align*}
	&\Pi\big( B_K \cap \{ f\leq \fMLEKs\} \, \big | \, N \big)  \notag \\
	&\leq
	2 \big((K_n+2)R\big)^{K_n -\ell(\bs)}\lambda^{K-K_n} (n^2)^{K_n-K} \big( 8RK_n n \log^{C'(R)} (n) \big)^{K-\ell(\bs)}
	n^{-(K_n-\ell(\bs)) \log n}
\end{align*}
with probability tending uniformly to one. If $K\leq K_n,$ then $K_n -K\leq K_n -\ell(\bs)$ and $K-\ell(\bs) \leq K_n -\ell(\bs).$ If $K > K_n, $ then we decompose $K-\ell(\bs) = (K-K_n )+ (K_n -\ell(\bs)).$ With $K_n \leq n^{1/2-\delta}$ and $n\to\infty$, the result follows.
\end{proof}

Indeed the posterior is asymptotically concentrated on the set $B_{K_n} \cap \{ f\leq \widetilde f\}.$ This means in particular that the posterior puts all mass on functions with the correct number $K_n$ of jumps.

\begin{thm}
\label{thm.model_selection}
Let  $K_n \leq n^{1/2-\delta}$ for some $\delta>0.$ Then,
\begin{align*}
	\lim_{n\to\infty}\inf_{f_0 \in \mM_S(K_n, R)} E_{f_0}\Big[\Pi\big( B_{K_n} \cap \{ f\leq \widetilde f\} \, \big | \, N \big) \Big]
	=1.
\end{align*}
\end{thm}

\begin{proof}
Since the function spaces are nested, it is enough to consider the case $R \geq 1.$ We show that the complement of $B_{K_n} \cap \{ f\leq \widetilde f\}$ has probability tending uniformly to zero. The complement can be decomposed as
\begin{align}
	\Big\{\bigcup_{K\neq K_n} B_K \Big\} \cup \Big\{\bigcup_{K\geq 1} (B^c)_K \Big\} \cup \Big\{ B_{K_n} \cap \{ f\leq \widetilde f\}^c \Big\}.
	\label{eq.model_selection_decomp}
\end{align}
As before, $M$ denotes the number of jumps of the MLE. To bound the first term observe that there are $\binom{K_n}{r}\binom{M-K_n}{K-r} \leq K_n^{K_n - r} M ^{K-r}$ possible functions $\fMLEKs$ with $K$ jumps and $\ell(\bs)=r.$ By \eqref{MLE8} there exists a constant $C(R)$ such that $P_{f_0}(M \leq C(R) (K_n+1) \log n)\xrightarrow{u}1.$ With Lemma \ref{lem.A_K^c_bd} (take $\beta=2$) and probability tending uniformly to one, we have
\begin{align}
	\Pi\Big(\bigcup_{K\neq K_n} B_K \, \Big | \, N \Big)
	&= \sum_{K\neq K_n} \Pi(B_K \cap \{ f\leq \widehat f^{\MLE} \} \, \big | \, N ) \notag \\
	&\leq \sum_{K\neq K_n}\sum_{\bs} \Pi(B_K \cap \{ f\leq \fMLEKs\}  \, \big | \, N ) \notag  \\
	&\leq \sum_{K\neq K_n}\sum_{r=0}^{K \wedge K_n}\sum_{\bs \, : \, \ell(\bs)=r} \lambda^{K-K_n} n^{-\frac 12 (1+\delta) (K-K_n)_+ - 2 (K_n -\ell(\bs))} \notag  \\
	& \leq \sum_{K\neq K_n}\sum_{r=0}^{K \wedge K_n} K_n^{K_n - r} (C(R) (K_n+1) \log n)^{K-r}  \lambda^{K-K_n} n^{-\frac 12 (1+\delta) (K-K_n)_+ - 2 (K_n -r)}  \notag  \\
	&\lesssim \sum_{K\neq K_n}\sum_{r=0}^{K \wedge K_n} n^{- \frac {\delta}2 (K-K_n)_+ - (K_n -r)} \notag  \\
	&= O(n^{-\delta/2}).
	\label{eq.model_selection_2}
\end{align}
On the event $H$ defined in \eqref{eq.H_def}, $\widehat f^{\MLE}(1)\leq R(K_n+2).$ The second term in decomposition \eqref{eq.model_selection_decomp} is therefore bounded uniformly in probability by Lemma \ref{lem.AK_bd}. For the third term, observe that $B_{K_n} \cap \{ f\leq \widetilde f\}^c$ means that $f\leq h_{\bs, K_n}$  for some $\bs$ with $h_{\bs,K_n} \neq \widetilde f$, implying $\ell(\bs) < K_n$. Arguing as for \eqref{eq.model_selection_2},
\begin{align*}
	\Pi\big(B_{K_n} \cap \{ f\leq \widetilde f\}^c \, \big | \, N \big)
	&\leq
	\sum_{r=0}^{K_n-1}\sum_{\bs \, : \, \ell(\bs)=r} n^{- 2 (K_n -\ell(\bs))}\\
	&\leq \sum_{r=0}^{K_n-1} K_n^{K_n - r} (C(R) (K_n+1) \log n)^{K_n-r}  n^{- 2 (K_n -r)}\\
	&=O(n^{-1}).
\end{align*}
This shows that all terms in \eqref{eq.model_selection_decomp} are bounded uniformly in probability.
\end{proof}

\subsection{Proof of Theorem \ref{thm.BvM2}}

Up to this point we proved that the posterior concentrates around the model with the correct number of jumps. In a next step, we derive contraction rates for the parameters. Given the sequences $(X_k')_k$ and $(Y_k^*)_k,$ define the intervals
\begin{align*}
	I_k :=\Big [Y_k^* - \frac {\log n}{2n (X_{k+1}'-X_k')} , Y_k^* \Big]
\end{align*}
and
\begin{align}
	T_k  :=\Big[ X_k',  X_k'+\frac{\log n}{2n(Y_k^*-Y_{k-1}^*)}\Big].
	\label{eq.Tk_def}
\end{align}
Define $I^*:= \{(a_0,\ldots, a_{K_n}): \sum_{\ell=0}^k a_\ell \in I_k \ \text{for all} \  k\}$ and $T^*:= \{(t_1, \ldots, t_{K_n}): t_k \in T_k \ \text{for all} \  k\}.$

\begin{lem}
\label{lem.localization}
Let $\lambda>0$ be fixed and $K_n \leq n^{1/2-\delta}$ for some $\delta >0.$ Then
\begin{align*}
	\lim_{n\to\infty}\inf_{f_0\in\mM_S(K_n,R)}E_{f_0}\Big[\Pi\Big( f=\sum_{k=0}^{K_n} a_k \mathbf{1}_{[t_k,T]} \ \text{with } \ \ba \in I^* \ \text{and} \ \bt\in T^* \, \Big | \, N \Big)\Big]=1.
\end{align*}
\end{lem}

\begin{proof}
By Theorem \ref{thm.model_selection} and Lemma \ref{lem.denom}, we have for any event $A$
\begin{align}
	 &\sup_{f_0 \in \mM_S(K_n, R)} E_{f_0}\big[\Pi(A |N)\big]  	\label{eq.loca1} \\
	 &\leq o(1) +  \frac{n^{2K_n+1} }{\lambda^{K_n} (1-n^{-1/2})^{2K_n+1}} E_{f_0} \Big[ \prod_{k=0}^{K_n} (X_{k+1}'-X_k') e^{-n \int_0^{T-\frac 1n} \widetilde f}	
	U_{K_n}\big( A \cap B_{K_n} \cap \{ f\leq \widetilde f\} \big) \Big].\notag
\end{align}
In a first step, we bound the posterior mass of the event
\[A_\ell :=\Big\{f: f=\sum_{k=0}^{K_n} a_k \mathbf{1}_{[t_k,T]},\ t_\ell \in T_\ell^c\Big\},\quad \ell \in \{1, \ldots, K_n\}.\]
A function $f \in B_{K_n} \cap \{ f\leq \widetilde f\}$  jumps $K_n$ times and the $k$-th jump lies in the interval $[X_k',X_{k+1}')$ for all $k.$  Define $X_{k, \ell}' := X_k' + \log n/[2n(Y_\ell^*-Y_{\ell-1}^*)] \delta_{k, \ell}.$ Using the definition of $T_\ell$ in \eqref{eq.Tk_def}, we therefore conclude that if  $f \in A_\ell \cap B_{K_n} \cap \{ f\leq \widetilde f\},$ then, $f$ has $K_n$ jumps and the $k$-th jump lies in $[X_{k,\ell}',X_{k+1}').$ Recall the definition of $\widetilde f$ in \eqref{eq.widetildef_def} and define
\begin{align*}
	\widetilde f _\ell = \sum_{k=0}^{K_n}Y_k^* \mathbf{1}_{[X_{k, \ell}',X_{k+1, \ell}')}
\end{align*}
such that $A_\ell \cap B_{K_n} \cap \{ f\leq \widetilde f\} \subseteq B_{K_n} \cap \{ f\leq \widetilde f_\ell\}$ (the sets are not necessarily equal). To bound $U_{K_n}(B_{K_n} \cap \{ f\leq \widetilde f_\ell\})$ we can now argue similarly as for the proof of Lemma \ref{lem.ub_num} with $K=K_n$ and replacing $\fMLE_{K_n,\bs}$ by $\widetilde f_\ell.$ This means that the jump locations $s_k$ are replaced by $X_{k,\ell}'$ and the function values $\fMLE_{K_n,\bs}(s_k)$ by $Y_k^*.$ If we upper bound \eqref{eq.lem_ub_num1} by one, we find
\begin{align}
	U_{K_n}\big(A_\ell \cap B_{K_n} \cap \{ f\leq \widetilde f\}\big)
	&\leq U_{K_n}\big(B_{K_n} \cap \{ f\leq \widetilde f_\ell\}\big) \notag  \\
	&\leq \frac{\lambda^{K_n} }{n^{2K_n+1} (1-n^{-1/2}) \prod_{k=0}^{K_n} (X_{k+1, \ell}'-X_{k, \ell}')}  e^{n \int_0^{T-1/n} \widetilde f_{\ell}}.
	\label{eq.localization_lem1}
\end{align}
By \eqref{MLE3} and \eqref{MLE4}, on $H$ it holds $(X_{\ell+1,\ell}'-X_{\ell,\ell}') \geq (X_{\ell+1}'-X_\ell')/2.$ For $k\neq \ell,$ $(X_{k+1,\ell}'-X_{k,\ell}') \geq (X_{k+1}'-X_k').$ With \eqref{eq.widetildef_def} and the definition of $X_{k,\ell}',$ we find $\int_0^{T-1/n} \widetilde f_{\ell} = \int_0^{T-1/n} \widetilde f - \log n/(2n)$ and we can further bound the right hand side in \eqref{eq.localization_lem1},
\begin{align*}	
	U_{K_n}\big(A_\ell \cap B_{K_n} \cap \{ f\leq \widetilde f\}\big)
	\leq \frac{2\lambda^{K_n} }{n^{2K_n+3/2} (1-n^{-1/2}) \prod_{k=0}^{K_n} (X_{k+1}'-X_k')} e^{n \int_0^{T-\frac 1n} \widetilde f }.
\end{align*}
Together with \eqref{eq.loca1}, a union bound for $\bigcup_{\ell=1}^{K_n}D_\ell$ yields
\begin{align*}
	\sup_{f_0 \in \mM_S(K_n, R)} E_{f_0}\Big[\Pi\Big( f=\sum_{k=0}^{K_n} a_k \mathbf{1}(\cdot \geq t_k) \ \text{with} \ \bt \in (T^*)^c \, \Big | \, N \Big) \Big]=o(1)+O\Big(\frac{K_n}{\sqrt{n}}\Big).
\end{align*}
Due to $K_n=o(n^{1/2})$ this converges to zero.

It remains to show that the posterior puts asymptotically all mass on sequences $(a_k)_k$ with $\sum_{\ell=0}^k a_\ell \in I_k.$ Since the likelihood is zero for $\sum_{\ell=0}^k a_\ell > Y_k^*,$ we only need to control the posterior mass of the sets $\sum_{\ell=0}^k a_\ell \leq  Y_k^* - \log n /(2n [X_{k+1}'-X_k']),$ $k=0,\ldots, K_n.$

Define $Y_{\ell, k}^*= Y_\ell^* - \log n /(2n [X_{k+1}'-X_k']) \delta_{k,\ell}$ and let $\widetilde f_{-k} =  \sum_{\ell=0}^{K_n}Y_{\ell,k}^* \mathbf{1}_{[X_{\ell}',X_{\ell+1}')}.$ Then we have
\begin{align*}
	\Big\{ \sum_{\ell=0}^k a_\ell \leq  Y_k^* - \frac{\log n}{2n (X_{k+1}'-X_k')}\Big\}
	\cap B_{K_n} \cap \{ f\leq \widetilde f\} \subseteq B_{K_n} \cap \{ f\leq \widetilde f_{-k}\}.
\end{align*}
Arguing as for \eqref{eq.localization_lem1}, we find that
\begin{align*}
U_{K_n}\big(B_{K_n} \cap \{ f\leq \widetilde f_{-k}\}\big)
	&\leq \frac{\lambda^{K_n} }{ n^{2K_n+1} (1-n^{-1/2}) \prod_{r=0}^{K_n} (X_{r+1}'-X_r')}  e^{n \int_0^{T-1/n} \widetilde f_{-k}}.
\end{align*}
Since $\int \widetilde f_{-k} = \int \widetilde f - \log n/(2n),$ we can argue as for the first part using \eqref{eq.loca1} and a union bound to show that
\begin{align*}
	\sup_{f_0 \in \mM_S(K_n, R)} E_{f_0}\Big[\Pi\Big( f=\sum_{k=0}^{K_n} a_k \mathbf{1}_{[t_k,T]} \ \text{with} \ \ba \in (B^*)^c \, \Big | \, N \Big) \Big]=o(1)+O\Big(\frac{K_n}{\sqrt{n}}\Big)
\end{align*}
which tends to zero. This completes the proof.
\end{proof}

\begin{proof}[Proof of Theorem \ref{thm.BvM2}]

We write $\Pi_{f_0,n}^\infty(\cdot | I^* \cap T^*) := \Pi_{f_0,n}^\infty(\cdot | \{(Y_k^*-E_k^*)_k \in I^*\} \cap \{(X_k'+E_k')_k \in T^*\}).$ In particular, $(Y_k^*-E_k^*)_k \in I_k$  means that $E_k^*\leq \log n/(2n(X_{k+1}'-X_k'))$ and $X_k'+E_k' \in T_k$ implies $E_k'\leq \log n/(2n(Y_k^*-Y_{k-1}^*)).$

By Lemma \ref{lem.TV_for_cond_distribs}, it is sufficient to show that uniformly over $f_0\in \mM(K_n,R)$
\begin{itemize}
\item[(i)] $ \sup_{f_0\in \mM_S(K_n,R)} E_{f_0}[ \| \Pi (\cdot \cap I^* \cap T^* | N)/\Pi(I^* \cap T^* |N) -\Pi_{f_0,n}^\infty(\cdot | I^* \cap T^*) \|_{\TV}] \rightarrow 0$
\item[(ii)] $\sup_{f_0\in \mM_S(K_n,R)} E_{f_0} [ \Pi( (I^*)^c \cup (T^*)^c |N)+\Pi_{f_0,n}^\infty((I^*)^c \cup (T^*)^c) \big] \rightarrow 0.$
\end{itemize}

We start with proving $(i).$ For any random variable $Z\leq 1,$ $E_{f_0}[Z ] \leq E_{f_0}[Z \cdot \mathbf{1}( H)]+P_{f_0}(H^c).$ By Lemma \ref{lem.diff_lb} it is therefore enough to prove $(i)$ on the event $H$ defined in \eqref{eq.H_def}. This means in particular, that we may use the inequalities \eqref{MLE3},\eqref{MLE4},\eqref{MLE5}.

We apply Lemma \ref{lem.TV_bd} and work therefore only up to multiplicative constants. The posterior density of the vectors $\ba$ and $\bt$ with respect to $d\bt d\ba$ is proportional to
\begin{align*}
	e^{n\sum_{k=0}^{K_n} a_k(T-t_k)} \mathbf{1}\big( \bt \in T^*, \ba  \in I^* \big)g_{K_n}(\ba) .
\end{align*}
Let us now prove that
\begin{align}
	\sup_{f_0\in \mM_S(K_n,R)} \, \sup_{\ba \in I^*}\Big|\prod_{k=1}^{K_n}\frac{a_k}{Y_k^*-Y_{k-1}^*} - 1\Big| =o(1).
	\label{eq.BvM0}
\end{align}
The constraints in $I^*$ imply that
\begin{align}
	Y_k^*-Y_{k-1}^* - \frac{\log n}{2n(X_{k+1}'-X_k')} \leq a_k \leq Y_k^*-Y_{k-1}^* + \frac{\log n}{2n(X_k'-X_{k-1}')}
	\label{eq.BvM01}
\end{align}
for all $k.$ Because of \eqref{MLE5} we consequently have
\begin{align}
	\Big|\frac{a_k}{Y_k^*-Y_{k-1}^*} - 1\Big| \leq \frac{2}{K_n \log(eK_n) \log^2 n}.
	\label{eq.BvM1}
\end{align}
For real numbers $\Delta_m,$ $m=1,\ldots,q$ define $\Delta := \max_m |\Delta_m|.$ Set $\Delta_0:=0.$ Then,
\begin{align}
	\Big | \prod_{m=1}^q (1+\Delta_m) -1 \Big|
	\leq \sum_{r=1}^q \Big|\prod_{m=0}^r (1+\Delta_m) - \prod_{m=0}^{r-1} (1+\Delta_m) \Big|
	\leq (1+\Delta)^q q\Delta.
	\label{eq.BvM11}
\end{align}
Thanks to \eqref{eq.BvM1} and setting $\Delta:=\frac{2}{K_n \log(eK_n) \log^2 n}$ and $q=K_n$, this proves \eqref{eq.BvM0}. If $\sum_{\ell=0} ^{K_n} a_\ell \in I_{K_n},$ then
\begin{align*}
	\big | Y_{K_n}^* - \sum_{\ell=0}^{K_n} a_\ell \big |
	\leq \frac{\log n}{2n (X_{K_n+1}'-X_{K_n}')} \leq \frac{\log n}{2n (T-1)}.
\end{align*}
Combining this with \eqref{eq.BvM0} and using \eqref{eq.BvM11} with $q=2$ yields
\begin{align}
	\sup_{f_0\in \mM_S(K_n,R)} \, \sup_{\ba \in I^*}\Big| \frac{g_{K_n}(\ba)}{\prod_{k=1}^{K_n} (Y_k^*-Y_{k-1}^*)}e^{Y_{K_n}^*} -1 \Big| = o(1).
	\label{eq.BvM2}	
\end{align}
This shows that the prior is asymptotically a constant over $\ba \in I^*$ and this will imply that it is washed out in the limit. In the next step, we show that the product term $\sum_k a_k t_k$ in the likelihood can be decoupled. For $\ba \in I^*, \bt \in T^*,$ we have due to \eqref{eq.BvM01}, the definition of $T^*$ and \eqref{MLE5}
\begin{align*}
	\big| \big( a_k - Y_k^* + Y_{k-1}^*\big) (X_k'-t_k) \big|
	\leq \frac{\log^2 n}{4n^2 (Y_k^*-Y_{k-1}^*)[(X_{k+1}'-X_k') \wedge (X_k'-X_{k-1}')]}
	\leq \frac{1}{K_n n\log n}.
\end{align*}
Hence, for $\ba \in I^*$ and  $\bt \in T^*,$
\begin{align*}
	\big | a_k(T-t_k) - a_k(T- X_k') - (Y_k^*-Y_{k-1}^*)(t_k - X_k') \big| \leq \frac{1}{K_n n\log n} \quad \text{for all} \ k,
\end{align*}
and with \eqref{eq.BvM11}
\begin{align*}
	\sup_{f_0\in \mM_S(K_n,R)} \, \sup_{\ba \in I^*, \, \bt \in T^*} \Big|\frac{e^{n \sum_k a_k(T-t_k)}}{e^{n \sum_k a_k(T- X_k') + (Y_k^*-Y_{k-1}^*)(t_k - X_k')}} - 1 \Big| 		
	=o(1).
\end{align*}
With \eqref{eq.BvM2} and by Lemma \ref{lem.TV_bd}, we see that the posterior converges in total variation and uniformly over $f_0\in \mM_S(K_n,R)$  to the distribution with Lebesgue density
\begin{align}
	&\propto e^{n\sum_{k=0}^{K_n} a_k(T-X_k')+ n \sum_{k=0}^{K_n}(Y_k^*-Y_{k-1}^*) t_k} \mathbf{1}(\bt \in T^*, \ba \in I^*) d\ba d\bt.
	\label{eq.BvM3}
\end{align}
To complete the proof, let us now show that this is the density of the distribution $\Pi_{f_0,n}^\infty(\cdot |I^*\cap T^*).$ Because we work conditionally on $T^*,$ we have that $X_k'+E_k' \in T_k$ and with \eqref{MLE5}, $E_k'\leq \log n/(2n(Y_k^*-Y_{k-1}^*))< X_{k+1}'-X_k'.$ On $T^*$ we therefore never have to take care of the truncation by $X_{k+1}'-X_k'$ that appears in the definition of $E_k'.$ Rewriting $f = \sum_{k=0}^{K_n} a_k \mathbf{1}_{[t_k,T]} = \sum_{k=0}^{K_n} \sum_{\ell=0}^k a_\ell \mathbf{1}_{[t_k,t_{k+1})}$ and comparing this with the unconditional limit distribution $\Pi_{f_0,n}^\infty(\cdot ),$ we find $\sum_{\ell=0}^k a_\ell = Y_k^*-E_k^*$ with $E_k^* \sim \Exp(n(X_{k+1}'-X_k'))$ and $t_k = X_k'+E_k'$ with $E_k' \sim \Exp(n(Y_k^*-Y_{k-1}^*)) \wedge (X_{k+1}'-X_k').$ Recall that the random variables $E_k^*$ and $E_k',$ $k=0,\ldots, K_n,$ are also independent. The Lebesgue density of $\Pi_{f_0,n}^\infty(\cdot |I^*\cap T^*)$ is therefore up to constants
\begin{align*}
	&d\Pi_{f_0,n}^\infty\Big(f= \sum_{k=0}^{K_n}  a_k \mathbf{1}_{[t_k,T]} \Big | I^*\cap T^* \Big) \\
	&\propto
	 e^{n\sum_{k=0}^{K_n} (X_{k+1}'-X_k')\sum_{\ell=0}^k a_\ell+ n \sum_{k=0}^{K_n}(Y_k^*-Y_{k-1}^*) t_k} \mathbf{1}(\bt \in T^*, \ba \in I^*)  \, d\ba d\bt \\
	 &= e^{n\sum_{k=0}^{K_n} a_k(T-X_k')+ n \sum_{k=0}^{K_n}(Y_k^*-Y_{k-1}^*) t_k} \mathbf{1}(\bt \in T^*, \ba \in I^*)  \, d\ba d\bt
\end{align*}
using partial summation for the last step. This is the same as \eqref{eq.BvM3} and the assertion in $(i)$ follows.

To prove $(ii)$ notice that $E_{f_0} [ \Pi( (I^*)^c \cup (T^*)^c |N)]\rightarrow 0$ follows from Lemma \ref{lem.localization}. Let $(E_k)_k$ be an i.i.d. sequence of $\Exp(1)$ random variables. Using the definition of $I^*, T^*$ and $E_k', E_k^*,$ we find with \eqref{MLE3}, \eqref{MLE4} and \eqref{MLE5} that on $H,$
\begin{align*}
	\Pi_{f_0,n}^\infty((I^*)^c \cup (T^*)^c)
	&=
	\Pi_{f_0,n}^\infty\Big( \bigcup_{k=0}^{K_n} \Big\{E_k^*> \frac{\log n}{2n(X_{k+1}'-X_k')}\Big\} \cup \bigcup_{k=1}^{K_n} \Big\{E_k' > \frac{\log n}{2n(Y_k^*-Y_{k-1}^*)}\Big\} \Big)\\
	&=
	\P\Big(\bigcup_{k=0}^{2K_n+1} \Big\{E_k \geq \frac{\log n}{2} \Big\}\Big)
	\leq \frac{2K_n+1}{\sqrt{n}} \rightarrow 0
\end{align*}
and consequently $\sup_{f_0\in \mM_S(K_n,R)}  E_{f_0}[\Pi_{f_0,n}^\infty((I^*)^c \cup (T^*)^c)] \rightarrow 0.$
This completes the proof of $(ii).$
\end{proof}

\subsection{Proof of Proposition \ref{prop.MLE_with_Kn_jumps}}

The likelihood process for functions with at most $K_n$ jumps that occur in $[0,1]$ is $f \mapsto e^{n \int_0^T f}\mathbf{1}(f \leq \widehat f^{\MLE}).$ The MLE over these functions must be a function of the form $\fMLE_{K_n, \bs}.$ Notice that $\widetilde f$ is the only function of this form with $\ell(\bs)=K_n.$ To show that asymptotically all other functions $\fMLE_{K_n, \bs}$ have a smaller likelihood, it suffices to prove
\begin{align*}
	P_{f_0}\Big(   \int_0^T \widetilde f > \int_0^T \fMLE_{K_n, \bs}, \ \forall \bs  \ \text{with}  \ \ell(\bs)<K_n \, \Big) \xrightarrow{u} 1.
\end{align*}
This follows from Lemma \ref{lem.int_comps} with $K=K_n.$

\subsection{Proof of Corollary \ref{cor.marg_BvM_CPP}}
As in the proof of Theorem \ref{thm.BvM2}, it will be enough to work on the event $H.$ Moreover, Theorem \ref{thm.BvM2} shows that it is sufficient to prove the assertion with the posterior replaced by the limit distribution $\Pi_{f_0,n}^\infty.$ Under the limit distribution, the functional $\vartheta$ can be written in the form \eqref{eq.vartheta_under_limit}.

To control the remainder term in \eqref{eq.vartheta_under_limit}, observe that on the set $I^* \cap T^*,$ by \eqref{MLE5},
\begin{align*}
	 \Big|\sum_{k=0}^{K_n} E_k^* (E_{k+1}'-E_k')\Big|
	 &\leq  \sum_{k=0}^{K_n-1} E_k^* E_{k+1}' \vee \sum_{k=1}^{K_n} E_k^* E_k' \\
	 &\leq  \sum_{k=0}^{K_n-1}  \frac{\log^2 n}{4n^2 (X_{k+1}'-X_k')(Y_{k+1}^*-Y_k^*)}
	 \vee \sum_{k=1}^{K_n}  \frac{\log^2 n}{4n^2 (X_{k+1}'-X_k')(Y_{k}^*-Y_{k-1}^*)} \\
	 &\leq \frac 1{n\log n}
\end{align*}
with probability tending uniformly to one. Next we define two new probability measures. Under $\Pi_{f_0,n}^{\infty,(1)},$ $\vartheta$ has distribution $\vartheta = \int_0^T \widetilde f - \sum_{k=0}^{K_n} E_k^* (X_{k+1}'-X_k') - \sum_{k=1}^{K_n} E_k' (Y_k^*-Y_{k-1}^*).$ Under $\Pi_{f_0,n}^{\infty,(2)},$ $\vartheta$ has distribution $\vartheta = \int_0^T \widetilde f - \sum_{k=0}^{K_n} E_k^* (X_{k+1}'-X_k') - \sum_{k=1}^{K_n} E_k'' (Y_k^*-Y_{k-1}^*)$ with independent $E_k'' \sim \Exp(n(Y_k^*-Y_{k-1}^*)).$ For the latter probability measure, $\vartheta$ does not have point masses anymore and can also be written as
\begin{align}
	\vartheta = 	 \int_0^T \widetilde f - \frac{1}{n} \sum_{k=1}^{2K_n+1} E_k, \quad \text{with} \ E_k \sim \Exp(1) \text{\ independent.}
	\label{eq.vtheta_under_(2)}
\end{align}
Moreover,  the densities of $E_k'$ and $E_k''$ are the same on the interval $[0, X_{k+1}'-X_k').$ If $(X_k'+E_k')_k \in I^*$ and $(X_k'+E_k'')_k \in I^*,$ then, on the event $H,$ $E_k'\vee E_k''\leq \log n/(2n(Y_k^*-Y_{k-1}^*)) < X_{k+1}'-X_k'$ for all $k.$ This implies that for any event $A$
\begin{align*}
	\Pi_{f_0,n}^{\infty,(1)} ( A \cap I^*)
	:=\Pi_{f_0,n}^{\infty,(1)} \big( A \cap 	\{(X_k+E_k')_k\in I^*\}\big)
	&= \Pi_{f_0,n}^{\infty,(2)} \big( A \cap 	\{(X_k+E_k'')_k\in I^*\}\big) \\
	&:= \Pi_{f_0,n}^{\infty,(2)} ( A \cap I^*).
\end{align*}
With exactly the same argument as in part $(ii)$ of the proof of Theorem \ref{thm.BvM2}, we have that $\Pi_{f_0,n}^{\infty,(2)}((I^*)^c \cup (T^*)^c) \leq (2K_n+1)/n \rightarrow 0.$ The following inequalities hold thus uniformly over $f_0\in \mM_S(K_n,R)$ and any $x\in \mathbb{R}.$ Set $m_n:=1/(n\log n),$ then
\begin{align}
	\Pi_{f_0,n}^{\infty,(2)}\big( (-\infty, x-m_n]) \big)
	&= \Pi_{f_0,n}^{\infty,(1)}\big( (-\infty, x-m_n]) \cap I^* \cap T^*\big) +o(1)  \notag \\
	&\leq \Pi_{f_0,n}^\infty\big( (-\infty, x]) \cap I^* \cap T^*\big) +o(1)  \label{eq.marg_BvM_CPP1}\\
	&\leq \Pi_{f_0,n}^{\infty,(1)}\big((-\infty, x+m_n]) \cap I^* \cap T^*\big) +o(1) \notag  \\
	&\leq \Pi_{f_0,n}^{\infty,(2)}\big((-\infty, x+m_n])\big) +o(1) \notag
\end{align}
and
\begin{align}
	\big\| \Pi_{f_0,n}^\infty( (-\infty, \cdot] \cap I^* \cap T^*) - \Pi_{f_0,n}^\infty( (-\infty, \cdot]) \big\|_\infty =o(1). \label{eq.marg_BvM_CPP2}
\end{align}
Denote the limit distribution $\mathcal{N}(\int \widetilde f - (2K_n+1)/n, (2K_n+1)/n^2)$ by $Q_n^\infty.$ Using \eqref{eq.vtheta_under_(2)} and Lemma \ref{lem.CLT_in_TV} we find
\begin{align}
	\sup_{f_0\in \mM_S(K_n,R)} E_{f_0}^n \Big[\Big\| \Pi_{f_0,n}^{\infty,(2)}(\vartheta \in \cdot ) -Q_n^\infty \Big\|_{\TV}\Big] \rightarrow 0.
	\label{eq.marg_BvM_CPP3}
\end{align}
Write $Q_{\mu,\sigma}$ for the normal distribution with mean $\mu$ and variance $\sigma^2.$ If $v>0,$ then $Q_{\mu,\sigma}((-\infty,x+v])\leq Q_{\mu,\sigma}((-\infty,x])+ v/\sqrt{2\pi \sigma^2}$ and $Q_{\mu,\sigma}((-\infty,x-v])\geq Q_{\mu,\sigma}((-\infty,x])-v/\sqrt{2\pi \sigma^2}.$ In particular, this shows that uniformly over $x\in \mathbb{R},$
\begin{align*}
	Q_n^\infty\big((-\infty,x-m_n] \big) = Q_n^\infty \big((-\infty,x]\big) +O\Big(\frac 1{\log n}\Big) = Q_n^\infty\big((-\infty,x+m_n]\big)+O\Big(\frac 1{\log n}\Big).
\end{align*}
Together with \eqref{eq.marg_BvM_CPP1}, \eqref{eq.marg_BvM_CPP2} and \eqref{eq.marg_BvM_CPP3} the assertion follows.

\subsection{Proof of Proposition \ref{prop.jump_nr_lb}}
The Bayes formula  \eqref{eq.Bayes} gives for any $m \geq 0,$
\begin{align*}
	\Pi(K \geq m | N) \leq \frac{\int_{K\geq m} e^{- n \int (f_0-f)_+} \frac{dP_{f\vee f_0}}{dP_{f_0}} (N) d\Pi(f)}{e^{-\sqrt{n\log n}} \Pi( X :  \| X -f_0 \|_1 \leq \sqrt{\log n /n}, X\leq f_0 )}
\end{align*}
with $X$ a CPP with intensity $\lambda.$ Bounding $e^{- n \int (f_0-f)_+} \leq 1$ and taking expectation with respect to $f_0$ yields
\begin{align}
	E_{f_0}\big[\Pi(K \geq m | N)\big] \leq \frac{e^{\sqrt{n\log n}}\Pi(K \geq m)}{ \Pi(X: \| X -f_0 \|_1 \leq \sqrt{\log n /n}, X\leq f_0 )}.
	\label{eq.jump_nr_lb1}
\end{align}
If $m \geq 1,$ we find by Stirling's approximation $m^m e^{-m} \leq  \sqrt{2\pi} m^{m+1/2} e^{-m} \leq m! \leq m^m$ and since $K$ follows under the prior a Poisson distribution with intensity $\lambda,$
\begin{align*}
	\Pi(K \geq m) \leq  e^{-\lambda}\frac{\lambda^m}{m!} \sum_{\ell=0}^\infty \frac {\lambda^\ell}{\ell !} = \frac{\lambda^m}{m!}
	\leq \lambda^m e^{m- m \log m}
\end{align*}
as well as $\Pi(K = m) \geq  \lambda^m e^{-\lambda-m \log m}.$ The latter inequality will be used to derive a lower bound for the denominator. For any $K\geq 1,$
\begin{align*}
	M_K &:=\Big\{ X= \sum_{k=0}^K a_k \mathbf{1} (\cdot \geq t_k) \, : \, t_k \in \Big[\frac{2k-1}{2K}, \frac{k}{K} \Big), f_0(t_{k+1}) - \frac {3}{2K} \leq  \sum_{\ell=0}^k a_\ell \leq f_0(t_k)\Big\} \\
	&\subset \Big\{ X: \| X -f_0 \|_\infty \leq \frac{3}{2K}, X\leq f_0\Big\}
\end{align*}
where $k=0, \ldots, K$ (except for $t_0:=0$) and $t_{K+1}:=1.$ On $M_K,$ for any $k=1, \ldots, K,$
\begin{align*}
	\sum_{\ell=0}^{k-1} a_\ell \leq f_0(t_{k-1})
	\leq \frac{k-1}{K} + \frac 12 \leq f_0(t_{k+1})  - \frac{3}{2K} \leq \sum_{\ell =0}^k a_\ell,
\end{align*}
and subtracting $\sum_{\ell=0}^{k-1} a_\ell$ on both sides yields $a_k \geq 0.$ The difference between the upper bound and the lower bound for $\sum_{\ell=0}^k a_\ell $ in the definition of $M_K$ is  $f_0(t_k)-f_0(t_{k+1}) + 3/(2K) \leq 1/K.$ Each of the $a_k$ ranges therefore over an interval of length $\geq 1/K$ in $[0,1].$ For $K_n := \lceil \sqrt{ n/\log n} \rceil,$ this gives with \eqref{eq.gk_lb} the lower bound,
\begin{align*}
	\Pi \Big(X: \| X -f_0 \|_1 \leq \sqrt{\frac {\log n}n}, X\leq f_0 \Big)
	&\geq
	\frac{\Pi(K = K_n)}{(2K_n)^{K_n}} \prod_{k=0}^{K_n}  \inf_{\eta_k\in [0,1-1/K_n]}g_k\Big(\Big[\eta_k, \eta_k+\frac 1{K_n}\Big]\Big) \\
	&\geq
	\frac{\lambda^{K_n}e^{-\lambda - K_n \log K_n}}{(2K_n)^{K_n}}\Big( \frac{c}{(\gamma+1)K_n^{\gamma+1}} \Big)^{K_n+1},
\end{align*}
where we used that $x\mapsto x^\gamma$ is monotone for the last inequality. Consequently, there exists a constant $C=C(\lambda, c, \gamma),$ such that with \eqref{eq.jump_nr_lb1},
\begin{align*}
	E_{f_0}\big[\Pi(K \geq m | N)\big] \leq  e^{\lambda + A \sqrt{n\log n} + m \log \lambda - m \log m + m}.
\end{align*}
Choosing $m = c^* \sqrt{n / \log n}$ with $c^*$ large enough, the right hand side converges to zero.

\subsection{Proof of Theorem \ref{thm.marg_BvM2}}

\begin{lem}
\label{lem.lb_nr_jumps_approx}
Let $\PC(K,R)$ be the space defined in \eqref{eq.PC_def}. If $f_0(x)=ax+b,$ then,
\begin{align*}
	\inf_{f\in \PC(K,\infty)} \int_0^1 |f_0(x) -  f(x) | \, dx \geq \frac{a}{4K}.
\end{align*}
\end{lem}

\begin{proof}
For any real $c$ and $r<s$ we have $\int_r^s |f_0(x)- c| dx \geq a(s-r)^2/4$ and hence
\begin{align*}
	\inf_{f\in \PC(K,\infty)}  \int_0^1 |f_0(x) -  f(x) |\,dx
	&= \inf_{0=: t_0 \leq t_1 \leq \ldots \leq t_K:=1} \, \sum_{k=1}^K \inf_{c_k \in \mathbb{R}}\int_{t_{k-1} }^{t_k} |f_0(x) -c_k| dx \\
	&\geq \frac a4 \inf_{0=: t_0 \leq t_1 \leq \ldots \leq t_K:=1} \, \sum_{k=1}^K (t_k-t_{k-1})^2 \geq \frac {a}{4K},
\end{align*}
where we use Jensen's inequality for the last step.
\end{proof}

\begin{lem}
\label{lem.MLE_concentration}
For $f_0 = (\tfrac 12 + \cdot) \wedge \tfrac 32$ and any sequence $M_n \rightarrow \infty,$
\begin{align*}
	P_{f_0}\Big( \int_0^1 \big(\widehat f^{\MLE}(x) - f_0(x) \big) \, dx \geq \frac{M_n}{\sqrt{n}}\Big) \rightarrow 0.
\end{align*}
\end{lem}

\begin{proof}
By Markov inequality
\begin{align*}
	P_{f_0}\Big( \int_0^1 \big(\widehat f^{\MLE}(x) - f_0(x) \big) \, dx \geq \frac{M_n}{\sqrt{n}}\Big)
	&\leq \frac{\sqrt{n}}{M_n}  \int_0^1 E_{f_0}\Big[\widehat f^{\MLE}(x) - f_0(x)  \Big] \, dx.
\end{align*}
The proof of Theorem 3.9 in \cite{reiss2014}, specifically the last equation display of the proof and replacing $[0,1]$ by $[0,T]$ with $\eps=T-1$, yields $ \int_0^1E_{f_0}[\widehat f^{\MLE}(x) - f_0(x)]\,dx=O(n^{-1/2})$ and thus the result.
\end{proof}

\begin{proof}[Proof of Theorem \ref{thm.marg_BvM2}]
Lemma \ref{lem.MLE_concentration} shows that it is enough to prove existence of a positive constant $c',$ such that
\begin{align}
	E_{f_0}\Big[\Pi\Big( \vartheta \geq \int_0^1 f_0(x) dx - \widetilde c\sqrt{\frac{\log n}{n}} \Big | \, N \, \Big)
	\mathbf{1}\Big( \int_0^1 \big(\widehat f^{\MLE}(x) - f_0(x) \big) \, dx \leq c'\sqrt{\frac{\log n}{n}} \Big)\Big] \rightarrow 0.
	\label{eq.thm_marg_BvM2_to_show}
\end{align}
By Proposition \ref{prop.jump_nr_lb}, we know that the posterior concentrates on models with $K_n \leq c^* \sqrt{n/\log n}$ for some positive constant $c^*.$ Applying Lemma \ref{lem.lb_nr_jumps_approx}, this means that the posterior puts asymptotically all mass on paths $f$ with
\begin{align*}
		\int_0^1 |f_0(x) -f(x)| dx \geq \frac{1}{8c^*} \sqrt{\frac{\log n}{n}}.
\end{align*}
Since the posterior also puts only mass on functions $f$ with $f\leq \widehat f^{\MLE},$ the posterior puts asymptotically all mass on $\vartheta$ with
\begin{align*}
	\vartheta
	&= \int_0^1 f_0(x) dx + \int_0^1 \big(f(x) -f_0(x) \big) \, dx \\
	&\leq \int_0^1 f_0(x) dx + 2\int_0^1 \big(\widehat f^{\MLE}(x) -f_0(x) \big) \, dx - \int_0^1 \big| f(x) - f_0(x)\big| \, dx\\
	&\leq \int_0^1 f_0(x) dx + 2\int_0^1 \big(\widehat f^{\MLE}(x) -f_0(x) \big) \, dx - \frac{1}{8c^*} \sqrt{\frac{\log n}{n}}.
\end{align*}
Choosing $c'= \tfrac{1}{32c^*}$ in \eqref{eq.thm_marg_BvM2_to_show} yields the assertion for $\widetilde c = \tfrac{1}{8c^*} -2c' = \tfrac{1}{16c^*}.$
\end{proof}

\section{Proofs of MLE properties}\label{AppMLE}

\begin{lem}
\label{lem.diff_lb}
The probability of the event $H$ tends uniformly to one and for $f_0\in \mM_S(K_n, R)$ on $H,$
\begin{align*}
	X_k'-X_{k-1}'\geq \frac{1}{\sqrt n}, \quad
	Y_k^*-Y_{k-1}^* \geq \frac{\log n}{\sqrt n } , \quad \text{for all} \ k=1, \ldots, K_n, \ \text{and} \ \  Y_0^* \geq \frac{\log n}{\sqrt{n}}.
\end{align*}
Furthermore,
\begin{align*}
	(Y_k^*-Y_{k-1}^*)\big[(X_{k+1}'-X_k') \wedge (X_k'-X_{k-1}') \big] \geq \frac{K_n \log (eK_n) \log^3 n}{4 n} \quad \text{for all} \ k=1, \ldots, K_n.
\end{align*}
\end{lem}

\begin{proof}
By construction of $(X_k^*,Y_k^*)_k$ and $(X_k',Y_k'),$ we have for $ k=1,\ldots,K_n$ that
\begin{equation}\label{EqYk*Xk'}
 Y_k^* - f_0(t_k^0)\sim \Exp(n(t_{k+1}^0-t_k^0)),\quad t_k^0-X_k'\sim \Exp(na_k^0)\wedge (t_k^0-t_{k-1}^0),
 \end{equation}
denoting a truncated exponential distribution with density $ce^{-\beta x} \mathbf{1}_{[0,t]}(x)$  by $\Exp(\beta)\wedge t.$
By $a_k^0\ge 2\log(n)/\sqrt n$ and $t_{k+1}^0-t_k^0\ge 2/\sqrt n$ we have
\[ P(X_k'<t_k^0-1/(2\sqrt n))\le \exp(-\log(n)),\,  P(Y_k^*>f_0(t_k^0)+\log(n)/(2\sqrt n)))\le \exp(-\log(n))\]
for all $k=1,\ldots,K_n$. Moreover, $\widehat f^{\MLE}(1)> (K_n+2)R \geq f(t_{K_n}^0) +R$ implies that no observation point lies in $[1,T] \times [f(t_{K_n}^0), f(t_{K_n}^0)+R].$ The smallest $y$-value among the   observation points on $[1,T]$ follows an $\Exp(n(T-1))$-distribution and
\begin{align*}
	P\big(\widehat f^{\MLE}(1)> f(t_{K_n}^0) +R \big)
	\leq P\big( \Exp(n(T-1)) \geq R \big) \rightarrow 0.
\end{align*}
A union bound shows $\sup_{f_0\in \mM_S(K_n, R)}P_{f_0}(H^c)\le 2K_nn^{-1}+o(1).$ By Remark \ref{RemMS} an asymptotically non-void set $\mM_S(K_n, R)$ implies $K_n =o(\sqrt n)$ and we deduce that the probability of $H$ tends uniformly to one.

On $H,$ $X_k'-X_{k-1}'\ge t_k^0-1/(2\sqrt n)-t_{k-1}^0$ and for $f_0\in \mM_S(K_n, R)$ this is larger than $3/(2\sqrt n)$. Similarly,
\begin{align*}
	Y_k^* - Y_{k-1}^* \geq f_0(t_k^0) - f_0(t_{k-1}^0) - \frac{\log n}{2\sqrt n}
	= a_k^0 - \frac{\log n}{2\sqrt n} \geq \frac{\log n}{\sqrt n}
\end{align*}
follows. The same arguments also gives $ Y_0^* \geq \log n /\sqrt{n}.$ For the last assertion we combine $X_k'-X_{k-1}' \geq t_k^0-t_{k-1}^0 - (2\sqrt{n})^{-1} \geq (t_k^0-t_{k-1}^0)/2$ and, similarly, $Y_k^* - Y_{k-1}^* \geq a_k^0/2$ with Assumption \ref{assump.min_rect}.
\end{proof}

We introduce the event
\begin{equation}\label{EqD}
D=\Big\{ \{(X_i,Y_i)\,|\,i\ge 1\}\cap\bigcup_{k=1}^{K_n}([X_k',t_k^0] \times [f(t_k^0), Y_k^*])=\varnothing\Big\}
\end{equation}
that there is no observation in any $[X_k',t_k^0] \times [f(t_k^0), Y_k^*].$ The rectangles are displayed in Figure \ref{fig.MLE}.

\begin{lem}
\label{lem.seqs_on_MLE}
We have $P_{f_0}(D)\xrightarrow{u}1$ and $P_{f_0}(\sum_{k=0}^{K_n}Y_k^* \mathbf{1}_{[X_k',X_{k+1}')} \leq \widehat f^{\MLE})\xrightarrow{u}1$.
\end{lem}

\begin{proof}
In view of \eqref{EqYk*Xk'} write $Y_k^*= f(t_k^0) + E_k^*/(n(t_{k+1}^0-t_k^0))$ and $X_k'=t_k^0 - (E_k'/(na_k^0))\wedge (t_k^0-t_{k-1}^0)$  with independent random variables $E_k, E_k' \sim \Exp(1).$ The union of all rectangles $[X_k',t_k^0] \times [f(t_k^0), Y_k^*]$ has Lebesgue measure bounded by
\begin{align*}
	\sum_{k=1}^{K_n} \frac{E_k^*E_k'}{n^2(t_{k+1}^0-t_k^0) a_k^0}
	\leq \frac{1 }{n K_n \log (eK_n) \log^3 n} \sum_{k=1}^{K_n} E_k^*E_k'.
\end{align*}

For each $k$ the PPP restricted to $S_k^-:=[t_{k-1}^0,t_{k}^0)\times (-\infty,f_0(t_{k})]$ is independent of the PPP restricted to $S_k^+=[t_{k-1}^0,t_{k}^0)\times (f_0(t_{k}),\infty]$. Since $[X_k',t_k^0) \times (f(t_k^0), Y_k^*\wedge f_0(t_{k+1}^0)]\subset S_k^+$ and $(X_k',Y_k')\in S_k^-$, $(X_k^*,Y_k^*\wedge f_0(t_{k+1}^0))\in S_{k+1}^-$, we obtain
\begin{align*}
&P\Big(\not\exists (X_i,Y_i)\in \bigcup_k  [X_k',t_k^0) \times (f(t_k^0), Y_k^*\wedge f_0(t_{k+1}^0) ]\,\Big|\, (X_k',Y_k^*\wedge f_0(t_{k+1}^0))_k\Big)\\
&=\exp\Big(-n\sum_k (t_k^0-X_k')(Y_k^*\wedge f_0(t_{k+1}^0)-f(t_k^0))\Big)\\
&\ge  \exp\Big(\frac{-1 }{K_n \log (eK_n) \log^3 n} \sum_{k=1}^{K_n} E_k^*E_k'\Big).
\end{align*}
For $\alpha\ge 0$ we have
\[\E[\exp(-\alpha E_k^*E_k')]=\int_0^\infty\int_0^\infty e^{-\alpha xy-x-y}dxdy=\int_0^\infty(1+\alpha x)^{-1}e^{-x}dx\ge 1-\alpha.\]
Thus by taking expectations
\[ P\Big(\not\exists (X_i,Y_i)\in \bigcup_k  [X_k',t_k^0) \times (f(t_k^0), Y_k^*\wedge f_0(t_{k+1}^0)]\Big)\ge \Big(1-
 \frac{1 }{K_n \log (eK_n) \log^3 n}\Big)^{K_n}\to 1.
\]
In view of the definition of the event $H$ from \eqref{eq.H_def}, we infer $P_{f_0}(Y_k^*\le f_0(t_{k+1}^0))\xrightarrow{u}1$ from \eqref{MLE1} and thus may drop the minimum in the last display. We conclude for the closed rectangles by the fact that on the boundaries of the rectangles there is with probability one no observation.

The construction of the sequences $(X_k')_k$ and $(Y_k^*)_k$ yields directly $P_{f_0}(\widetilde f \leq \widehat f^{\MLE})\xrightarrow{u}1$.
\end{proof}

\begin{lem}
\label{lem.ub_jumps_of_MLE}
There exists a constant $C(R)$ such that
\[ P_{f_0}\Big(\forall k=0,\ldots,K_n:\,\#\{\widehat t_\ell^{\MLE}\in [t_{k}^0,t_{k+1}^0\wedge 1]\,|\,\ell=1,\ldots,M\}\le C(R)\log(n)\Big)\xrightarrow{u} 1.
\]
\end{lem}

\begin{proof}
We count the number of jumps of the MLE on each interval $[t_{k}^0, t_{k+1}^0].$ By Lemma \ref{lem.seqs_on_MLE} it is sufficient to work on the event, where all rectangles $[X_k',t_k^0] \times [f(t_k^0), Y_k^*]$ contain no observation. Then, on $[t_{k}^0,X_k^*)$ the MLE equals $Y_k^*.$

Starting with $(X_{k,0},Y_{k,0}) :=(X_k^*,Y_k^*)$ introduce inductively
\begin{equation}\label{EqXkr}
(X_{k,r},Y_{k,r}):=\argmin_{(X_i,Y_i)_i}\{ Y_i\,|\, X_i\in(X_{k,r-1},t_{k+1}^0)\}, \quad r\ge 1,
\end{equation}
the $r$-th observation on the graph of the MLE for the model on $[t_k^0,t_{k+1}^0]$ to the right of $(X_k^*,Y_k^*).$ We have for a sequence $(U_{k,r})_{r\geq 1}$ of i.i.d. $\Unif[0,1]$ random variables with $U_{k,r}$, independent of $(X_{k,\ell})_{\ell\le r}$, that  $X_{k,r+1} = t_{k+1}^0-(t_{k+1}^0-X_{k,r})U_{k,r}$ and $Y_{k,r+1}|(X_{k,r},Y_{k,r}) \sim Y_{k,r} + \Exp(n(t_{k+1}^0-X_{k,r}))$.

A formal derivation of the laws of $(X_{k,r+1},Y_{k,r+1})$ uses conditioning on stopping sets. Conditional on the observations on a closed stopping set the PPP on the complement remains a PPP of intensity $n$, independent of the observations on the stopping set, see  Thm. 6.2 in \cite{baldin2016} for the analogous case of compact sets. By definition of $(X_{k,r},Y_{k,r})$, the random sets $\widehat S_{k,r}=[t_k^0,t_{k+1}^0]\times (-\infty,Y_{k,r}]$ form a stopping set in the sense that the event $\{\widehat S_{k,r}\subset S\}$ for any closed set $S$ is contained in the $\sigma$-field generated by the observations in $S$. Hence, the PPP on $\widehat S_{k,r}^c$ is independent of $(X_{k,\ell},Y_{k,\ell})_{\ell\le r}$ and the conditional laws $X_{k,r+1}\sim U([X_{k,r},t_{k+1}^0])$, $Y_{k,r+1}-Y_{k,r}\sim \Exp(n(t_{k+1}^0-X_{k,r}))$ independently follow from standard PPP properties.

Let $L= \lceil 2\log (Rn^2)/\log(4/3) + 8 \log(n) \rceil.$ Define the events $C_k:=\{t_{k+1}^0 -X_{k,L} > (\frac{3}{4})^{L/2}\}$ and $A_k := \{(U_{k,\ell})_{\ell=1,\ldots, L} : \#\{s:U_{k,s} \leq 3/4\}< L/2\}.$ On $A_k^c,$ we have $t_{k+1}^0 -X_{k,L} = (t_{k+1}^0 -X_{k,0}) \prod_{\ell=0}^{L-1} U_{k,\ell} \leq (3/4)^{L/2}$ and therefore $A_k^c \subseteq C_k^c.$ With $(Z_s)_{s\geq 1}$ a sequence of i.i.d. Bernoulli random variables with success probability $1/4,$ this shows
\begin{align*}
	P(C_k) \leq P(A_k)
	\leq
	P\Big(\sum_{s=1}^L (1-Z_s) \leq  L/2\Big)
	\leq P\Big(\sum_{s=1}^L (Z_s-\tfrac 14) \geq L/4\Big)
	\leq e^{-L/8}\leq \frac 1n,
\end{align*}
using Hoeffding's inequality and $L\geq 8\log n.$

Using that $Y_{k,r+1}|(X_{k,r},Y_{k,r}) \sim Y_{k,r} + \Exp(n(t_{k+1}^0-X_{k,r}))$ and $L \geq 2\log (Rn^2)/\log(4/3),$ we obtain for any $k=0,1,\ldots, K_n-1,$
\begin{align*}
	&P\big(Y_{k,L+1} \leq f(t_{k+1}^0)\big)\\
	&\leq E\Big[ P\Big(Y_{k,L+1} \leq f(t_k^0) +R \, \Big | \, (X_{k,L},Y_{k,L})\Big){ {\bf 1}_{ C_k^c}}\Big]
	+ \frac 1n\\
	&\leq 1-\exp\Big( -Rn \Big(\frac{3}{4}\Big)^{L/2}\Big)+ \frac 1n\\
	&\leq Rn\Big(\frac{3}{4}\Big)^{L/2}+ \frac 1n \\
	&\leq \frac 2{n}.
\end{align*}
The case $k=K_n$ is special. Since
\begin{align*}
	P\big(\widehat f^{\MLE}(1) - f(t_{K_n}^0) \geq R\big)
	\leq
	P\big(\min_i \{Y_i :X_i \in [1,T]\} \geq R\big)
	= e^{-Rn(T-1)}\rightarrow 0
\end{align*}
we can argue as above. This means that the number of jumps of the MLE on the interval $[t_k^0,t_{k+1}^0\wedge 1]$ is bounded by $L+1$ with probability at least $1-2/n-e^{-Rn(T-1)}.$ Because of $K_n=o(n)$ the assertion follows with the union bound.
\end{proof}

\begin{lem}
\label{lem.max_gamma}
Let $M\geq 1$ and consider any family of random variables $Z_k \sim \Gamma(2,1),$ for $k=1,2, \ldots, M.$ Then for any $t\geq 0$
\begin{align*}
	P\Big(\max_{k=1, \ldots, M} Z_k \geq \log M + \log \log (eM) + t\Big) \le (2+t)e^{-t}.
\end{align*}
\end{lem}

\begin{proof}
Let $Z\sim \Gamma(2,1).$ Since $ue^{-u/2}\leq 2/e\leq 1,$
\begin{align*}
	P(Z\geq x) = \frac 12	\int_x^\infty u e^{-u} du
	\leq \frac 12 \int_x^{2x} u e^{-u} du + \frac 12 \int_{2x}^\infty e^{-u/2} du
	\leq (x+1)e^{-x}.
\end{align*}
Hence, $P(Z\geq \log M + \log \log (eM) + t) \leq M^{-1}(2+t)e^{-t}.$ The assertion follows using the union bound.
\end{proof}


The MLE has much more jumps than the true function, see Figure \ref{fig.MLE}. Based on the next lemma, we can bound the increase of the likelihood induced by artificial jumps.

\begin{lem}
\label{lem.ub_areas}
With probability tending uniformly to one we have for all $k=0,\ldots,K_n$
\begin{align*}
\max_{\ell: t_k^0\le \widehat t_\ell^{\MLE}\le t_{k+1}^0} \big( \widehat f^{\MLE}(\widehat t_\ell^{\MLE})-f(t_k^0) \big)\big(t_{k+1}^0 - \widehat t_\ell^{\MLE} \big)  \leq \frac{\log K_n + C' \log \log n} n.
\end{align*}
\end{lem}

\begin{proof}
Consider $(X_{k,r}, Y_{k,r})_{r\ge 0}$ from \eqref{EqXkr} and set $X_{k,-1}:=t_k^0$. It is argued in the proof of Lemma \ref{lem.ub_jumps_of_MLE} that, with probability uniformly tending to one, all jump locations $\widehat t_\ell^{\MLE}$ of the MLE, which lie in $[t_k^0,t_{k+1}^0]$ are given by $(X_{k,r})_{r=0,\ldots,R}$ where the number $R$ of these jumps is by the statement of Lemma \ref{lem.ub_jumps_of_MLE} of order $\log n$. It thus suffices to prove that for any constant $C>0$ there exists a constant $C'>0$ such that
\begin{align}\label{EqOldAssertion}
	P_{f_0}\Big(\max_{k=0, \ldots, K_n, 0\leq r \leq C \log n}\Leb([X_{k,r-1},t_{k+1}^0]\times[f(t_k^0),Y_{k,r}]) \leq \frac{\log K_n + C' \log \log n} n\Big)\xrightarrow{u} 1.
\end{align}

Set $A_{k,r} := \Leb([X_{k,r-1},t_{k+1}^0]\times[f(t_k^0),Y_{k,r}])$.
By construction, conditional  on  $(X_{k,\ell})_{0\le\ell\le r-2}$ and $(Y_{k,\ell})_{0\le\ell\le r-1}$, $X_{k,r-1}$ is uniformly distributed on $[X_{k,r-2},t_{k+1}^0]$ and  $Y_{k,r}$ is the minimum in $y$-direction of an independent PPP on $[t_k^0,t_{k+1}^0]\times (Y_{k,r-1},\infty)$ of intensity $n$. These properties imply that
\begin{align*}
U_{k,r}&:=(t_{k+1}^0-X_{k,r-1})/(t_{k+1}^0-X_{k,r-2})\sim U([0,1]),\\
E_{k,r}&:=n(t_{k+1}^0-X_{k,r-1})(Y_{k,r}-Y_{k,r-1})\sim \Exp(1)
\end{align*}
are independent and independent of $(X_{k,\ell})_{\ell\le r-2},(Y_{k,\ell})_{\ell\le r-1}$. This gives the  recurrence property
\begin{align*}
	A_{k,0}=\tfrac1n E_{k,0},\quad A_{k,r} = U_{k,r}A_{k,r-1}+ \tfrac1n E_{k,r}, \quad \text{for} \ r=1,2,\ldots
\end{align*}
with random variables $U_{k,r}\sim U([0,1])$, $E_{k,r}\sim \Exp(1)$, all independent.

A formal PPP proof relies on the stopping set property of
\[S_{k,r-2}:=([t_k^0,t_{k+1}^0]\times[f(t_k^0),Y_{k,r-2}])\cup ([t_k^0,X_{k,r-2}]\times [f(t_k,0),\infty)),\]
which shows that the PPP on the complement $(X_{k,r-2},t_{k+1}^0]\times(Y_{k,r-2},\infty)$ remains a PPP independent of $(X_{k,\ell},Y_{k,\ell})_{\ell\le r-2}$  such that in turn $X_{k,r-1}\sim U([X_{k,r-2},t_{k+1}^0])$ and $Y_{k,r-1}-Y_{k,r-2}\sim \Exp(n(t_{k+1}^0-X_{k,r-2}))$ are independent, given $(X_{k,\ell},Y_{k,\ell})_{\ell\le r-2}$. This yields consecutively $X_{k,r-1}\sim U([X_{k,r-2},t_{k+1}^0])$ given $(X_{k,\ell})_{\ell\le r-2}$ and $(Y_{k,\ell})_{\ell\le r-1}$ and then, increasing $r$, $Y_{k,r}-Y_{k,r-1}\sim \Exp(n(t_{k+1}^0-X_{k,r-1}))$ given $(X_{k,\ell},Y_{k,\ell})_{\ell\le r-1}$.

We now use that for independent $U \sim \Unif[0,1],$ $G\sim \Gamma(2,1)$ and $E\sim \Exp(1)$ we have $UG+E\sim \Gamma(2,1)$, which is easily checked via $UG\sim\Exp(1)$. For $nA_{k,0}\sim\Gamma(2,1)$ the recurrence relation would thus give $nA_{k,r}\sim\Gamma(2,1)$, a stationary solution. Since $nA_{k,0}\sim\Exp(1)\le\Gamma(2,1)$ with respect to stochastic ordering, the monotonicity of the recurrence relation shows that $nA_{k,r}$ can be upper bounded by a $\Gamma(2,1)$-distributed random variable.
Consequently, $\max_{k=0, \ldots, K_n, 0\leq r \leq C \log n}nA_{k,r}$ can be bounded from above by the maximum over $(K_n+1)(1+C \log n)$ many $\Gamma(2,1)$-distributed random variables. The assertion \eqref{EqOldAssertion} follows from Lemma \ref{lem.max_gamma} with $t=\log\log n$ and some elementary algebra.
\end{proof}

\begin{lem}
\label{lem.lb_areas}
For $\tau = (T+1)/2,$
\begin{align*}
	P_{f_0}\Big(\min_{k=1, \ldots, K_n} (Y_k^* - Y_k') (X_k^*\wedge \tau-X_k') \geq \frac{\log^2 n}{n}\Big)\xrightarrow{u} 1.
\end{align*}
\end{lem}

\begin{proof}
Observe first $(Y_k^* - Y_k') (X_k^*\wedge \tau-X_k')\ge (f(t_k^0)-Y_k')(X_k^*\wedge \tau-t_k^0).$ Using the independence of the PPP on $[0,t_k^0)\times\R$ and $[t_k^0,T]\times\R$, $U_k^*:=(X_k^*\wedge \tau-t_k^0)/(t_{k+1}^0\wedge \tau-t_k^0)$ and $U_k':=(f(t_k^0)-Y_k')/a_k^0$ are independent with $U_k^*, U_k'$  stochastically larger than $U([0,1])$ (recall $Y_k'=f_0(t_{k-1}^0)$ in case $R_k=\varnothing$). By the properties of $f_0$ in Definition \ref{assump.min_rect} and using $(t_{k+1}^0\wedge \tau-t_k^0) \geq \tfrac 12 (t_{k+1}^0-t_k^0),$
\begin{align*}
	\min_{k=1, \ldots, K_n} (Y_k^* - Y_k') (X_k^*\wedge \tau-X_k')
	\geq \frac{K_n \log (eK_n) \log^3 n}{n} \min_{k=1, \ldots, K_n}  U_k^* U_k'.
\end{align*}
From $-\log(U_1U_2)\sim \Gamma(2,1)$ for independent $U_1,U_2\sim U([0,1])$  we deduce
\[ P\Big(\min_{k=1, \ldots, K_n} (Y_k^* - Y_k') (X_k^*\wedge \tau-X_k') < \frac{\log^2 n}{n}\Big)\le P\Big(K_n\log(eK_n)\log(n)e^{-V*}< 1\Big)
\]
with $V^*$ the maximum over $K_n$ independent $\Gamma(2,1)$-random variables. The result follows from Lemma \ref{lem.max_gamma} with $t= \log \log n.$
\end{proof}

\section{Results on total variation distance}\label{AppTV}

\begin{lem}
\label{lem.TV_for_cond_distribs}
Let $P,Q$ be probability measures on the same measurable space $(\Omega, \mathcal{A}).$ For any $A \in \mathcal{A}$ with $P(A)>0$, $Q(A)>0$
\begin{align*}
	\TV(P,Q) \leq \TV\big( P(\cdot | A) , Q(\cdot |A) \big) + 2P(A^c) +2Q(A^c).
\end{align*}
\end{lem}

\begin{proof}
The assertion follows from
\begin{align*}
	\TV\big(P, P(\cdot | A) \big)
	&= \sup_{D \in A} \Big| P(D) - \frac{P(D\cap A)}{P(A)}\Big| \\
	&\leq  \sup_{D \in A}  \Big| P(D \cap A)\Big( 1 - \frac{1}{P(A)} \Big) \Big| +P(A^c)  \\
	&=  \sup_{D \in A} \frac{P(D \cap A)}{P(A)} P(A^c) +P(A^c)  \\
	&\leq 2P(A^c)
\end{align*}
and the triangle inequality.
\end{proof}

By a slight abuse of notation we write $\|P_X-P\|_{\TV}$ as $\TV(X,P)$ when $X\sim P_X.$

\begin{lem}
\label{lem.CLT_in_TV}
If $\xi_j \sim \Exp(1),$ $j=1, 2, \ldots$ are independent, then, for any real number $\alpha$ and any sequence of integers $M \rightarrow \infty,$
\begin{align*}
	\TV\Big( \alpha- n^{-1}\sum_{j=1}^M \xi_j , \mathcal{N}\Big(\alpha-\frac Mn, \frac{M}{n^2} \Big) \Big)\rightarrow 0.
\end{align*}
\end{lem}

\begin{proof}
Invertible transformations do not change the total variation distance. Therefore,
\begin{align*}
	\TV\Big( \alpha- n^{-1}\sum_{j=1}^M \xi_j , \mathcal{N}\Big(\alpha-\frac Mn, \frac{M}{n^2} \Big) \Big)
	= \TV\big( M^{-1/2}\sum_{j=1}^M (1-\xi_j), \mathcal{N}(0,1)\big).
\end{align*}
By the CLT in total variation (cf. \cite{bally2016}, Theorem 2.5), the latter converges to zero as $M \rightarrow \infty.$
\end{proof}

\begin{lem}
\label{lem.TV_bd}
Consider two probability measures $P,Q$ on the same probability space with $P\ll Q.$ If $h(x) \propto dP/dQ(x),$ and $\int |h(x) -1| dQ(x) \leq \delta$ for some $\delta\in(0,1)$, then
\begin{align*}
	\TV(P,Q) \leq \frac{\delta	}{1-\delta}.
\end{align*}
\end{lem}

\begin{proof}
By triangle inequality, we have $\int h\, dQ \geq 1-\delta$ and
\begin{align*}
	\TV(P,Q)
	&= \frac 12 \int \Big| \frac{h(x)}{\int h\, dQ }- 1\Big| dQ(x)
	\leq \frac{1}{2} \int  \Big(\frac{|h(x) -1|}{\int h\, dQ} + \frac{|1-\int h\, dQ|}{\int h\, dQ}\Big)\, dQ(x).
\end{align*}
Both terms in the integral are upper bounded by $\delta/(1-\delta)$.
\end{proof}

\bibliographystyle{acm}       
\bibliography{bibIM}           

\end{document}